\documentclass[11pt]{amsart}

\usepackage{amsmath,amssymb,amsthm,mathtools}
\usepackage{xcolor}
\usepackage{tikz-cd}
\usetikzlibrary{arrows.meta,positioning}
\usepackage[colorlinks=true,citecolor=blue,linkcolor=blue,urlcolor=blue]{hyperref}

\newcommand{\CC}{\mathbb C}
\newcommand{\RR}{\mathbb R}
\newcommand{\QQ}{\mathbb Q}

\newcommand{\cL}{\mathcal L}
\newcommand{\cO}{\mathcal O}
\newcommand{\Id}{\operatorname{Id}}
\newcommand{\Tor}{\operatorname{Tor}}
\newcommand{\Gr}{\operatorname{Gr}}
\newcommand{\Sym}{\operatorname{Sym}}
\newcommand{\Vol}{\operatorname{Vol}}
\newcommand{\codim}{\operatorname{codim}}
\newcommand{\pardeg}{\operatorname{par-deg}}
\newcommand{\parch}{\operatorname{par}ch}
\newcommand{\parc}{\operatorname{par}c}
\newcommand{\rk}{\operatorname{rk}}
\newcommand{\Tr}{\operatorname{Tr}}
\newcommand{\Supp}{\operatorname{Supp}}
\newcommand{\Hom}{\operatorname{Hom}}
\newcommand{\End}{\operatorname{End}}
\newcommand{\im}{\operatorname{im}}
\newcommand{\diag}{\operatorname{diag}}
\newcommand{\op}{\mathrm{op}}
\newcommand{\an}{\mathrm{an}}

\newcommand{\ii}{\sqrt{-1}}

\theoremstyle{plain}
\newtheorem{thm}{Theorem}[section]
\newtheorem{lem}[thm]{Lemma}
\newtheorem{prop}[thm]{Proposition}
\newtheorem{cor}[thm]{Corollary}

\newtheorem{alphthm}{Theorem}

\theoremstyle{definition}
\newtheorem{defn}[thm]{Definition}
\newtheorem{setup}[thm]{Setup}

\theoremstyle{remark}

\numberwithin{equation}{section}

\title[Parabolic BG inequality and pluriharmonic metrics]{Parabolic Bogomolov--Gieseker Inequality and Pluriharmonic Metrics on Regular Parabolic Higgs Sheaves over Compact Kähler Manifolds}
\author{Tianshu Jiang, Jiayu Li}
\date{}
\subjclass[2020]{53C07, 32Q15, 14J60}
\keywords{regular parabolic Higgs sheaf, Bogomolov--Gieseker inequality, pluriharmonic metric, tame harmonic bundle, compact Kähler manifold}

\let\originalsection\section
\RenewDocumentCommand{\section}{s o m}{%
  \IfBooleanTF{#1}{%
    \originalsection*{#3}%
    \markboth{#3}{#3}%
  }{%
    \IfNoValueTF{#2}{%
      \originalsection{#3}%
      \expanded{\noexpand\markboth
        {\number\value{section}.\ \unexpanded{#3}}%
        {\number\value{section}.\ \unexpanded{#3}}}%
    }{%
      \originalsection[#2]{#3}%
      \expanded{\noexpand\markboth
        {\number\value{section}.\ \unexpanded{#2}}%
        {\number\value{section}.\ \unexpanded{#2}}}%
    }%
  }%
}
\makeatletter
\renewcommand{\leftmark}{\expandafter\@firstoftwo\botmark{}{}}
\makeatother

\begin{document}
\raggedbottom

\begin{abstract}
We prove a parabolic Bogomolov--Gieseker inequality for stable regular parabolic Higgs sheaves on a compact Kähler manifold \((X,\omega)\) equipped with a simple normal crossing divisor \(D\). We also prove that a polystable regular parabolic Higgs sheaf admits a pluriharmonic metric on \(X\setminus D\), provided that each stable summand has parabolic degree zero and vanishing integrated second parabolic Chern character.
\end{abstract}

\maketitle

\section{Introduction}\label{sec:introduction}

\subsection{Background}

The Kobayashi--Hitchin correspondence relates the algebraic stability of a
holomorphic bundle to the existence of a Hermitian--Einstein metric.
For Higgs bundles on compact K\"ahler manifolds, this correspondence underlies
nonabelian Hodge theory; see, for example,
\cite{Simpson1988,Simpson1992}. On the complement of a divisor, however,
the metric is generally singular, and its asymptotic
behavior becomes part of the correspondence. Parabolic filtrations encode
this behavior algebraically.

For a noncompact algebraic curve \(U=\overline X\setminus D\), Simpson introduced tame harmonic bundles and established their correspondence with filtered regular Higgs bundles and filtered regular \(\mathcal D\)-modules on the smooth compactification \(\overline X\) \cite{Si2}. In this correspondence, the parabolic filtration records the growth of the harmonic metric at the punctures. This provides the one-dimensional starting point for the results considered here.

Mochizuki extended this picture to arbitrary dimension. Let
\(\overline X\) be a smooth projective variety and let \(D\subset
\overline X\) be a simple normal crossing divisor, and put
\(U=\overline X\setminus D\). Mochizuki established the higher-dimensional
Kobayashi--Hitchin correspondence between tame harmonic bundles on \(U\)
and stable parabolic Higgs or \(\lambda\)-flat bundles on
\((\overline X,D)\) whose parabolic degree and second parabolic
Chern-character number vanish; he also proved the corresponding
parabolic Bogomolov--Gieseker inequality. See
\cite{Mochizuki2006}. Biquard's work on logarithmic Higgs bundles
for a smooth divisor \cite{Bi}, and the results of Li--Narasimhan and Li
on parabolic Hermitian--Einstein metrics and Chern-number inequalities
over K\"ahler manifolds \cite{Li-Na,Li}, form closely related parts of
this development.

The present paper extends these two conclusions to regular parabolic
Higgs sheaves on compact K\"ahler pairs. We prove a parabolic
Bogomolov--Gieseker inequality and construct an adapted
pluriharmonic metric when each stable summand satisfies the numerical
vanishing conditions in Theorem~\ref{thm:harmonic-main}.

\subsection{Difficulties}

Three difficulties arise in carrying out this extension. The first comes
from the residues of the Higgs field. If the residue induced on a graded
piece of the parabolic filtration has a nonzero nilpotent part, the direct
construction of an adapted Hermitian metric with the required curvature
control becomes problematic. Mochizuki's solution is to perturb the
parabolic structure so that the resulting object is graded semisimple,
that is, the nilpotent part of every induced graded residue vanishes
\cite[Subsection~2.2 and Section~16]{Mochizuki2006}. We follow this
approach by refining the residue-invariant filtrations and then
separating their weights.

The second difficulty is specific to carrying out this perturbation
directly in higher dimension. In Mochizuki's projective argument, the
essential perturbation is performed on a surface, and the
higher-dimensional statement is recovered by a Mehta--Ramanathan type
restriction theorem \cite[Sections~17 and~26]{Mochizuki2006}. On a
surface, the graded pieces lie on curves, where their saturated
torsion-free refinements are locally free; the perturbed object therefore
remains a parabolic Higgs bundle. The projective restriction argument is
not available in this form for an arbitrary compact K\"ahler manifold.
We must refine the graded residues in the ambient dimension, and the
resulting saturated eigensheaves and kernel filtrations need only be
torsion-free along the components of \(D\). Consequently, the perturbed
object may be a parabolic Higgs sheaf rather than a parabolic Higgs
bundle.

Although our main goal is to treat locally abelian parabolic Higgs
bundles, the intermediate perturbation requires us to allow sheaves.
We work with regular parabolic Higgs sheaves, whose underlying
sheaves are reflexive and whose divisor filtrations are saturated. We
then resolve such a sheaf into a logarithmic parabolic Higgs bundle on a
modification. This construction must preserve the prescribed flags and
residue information. We also need stability under small perturbations
and a comparison of the relevant parabolic characteristic numbers.

After this resolution, our proof uses a sequence of strictly ordered
rational parabolic weights converging to the original weights. For each
member of this approximating sequence, we construct a global adapted
Hermitian--Einstein metric on the divisor complement of the
modification. Its defining equation, together with the parabolic
Chern--Weil energy formula, shows that the corresponding
parabolic discriminant number is nonnegative.
Taking the limit of these numbers proves the Bogomolov--Gieseker
inequality. To obtain a pluriharmonic metric, however,
one must also take a limit of the metrics themselves.

This leads to the third, and principal analytic, difficulty. The
existence of a Hermitian--Einstein metric for each fixed approximation
does not by itself give convergence as the weights return to their
original values. Moreover, the K\"ahler forms used on the resolution also
vary and degenerate toward the pullback of \(\omega\). We must obtain
compactness for the global Hermitian--Einstein metrics, control the
source integrals near the divisor, and retain every parabolic growth
exponent in the limit. Under the vanishing Chern number conditions, the limiting
metric must then be shown to have zero Hitchin--Simpson curvature energy and
to recover the original filtered lattices. Establishing this metric limit is a main contribution of the paper.

\subsection{Main results}

Let \(X\) be a connected compact K\"ahler manifold of complex dimension
\(n\ge2\), let \(\omega\) be a K\"ahler form, and let \(D\subset X\) be a
simple normal crossing divisor. We work with the decreasing real-indexed
convention for parabolic filtrations described in
Section~\ref{sec:preliminaries}. For a rank-\(r\) regular parabolic Higgs
sheaf \((E_*,\theta)\), set
\[
 \Delta_{\mathrm{par}}(E_*)
 =\parc_1(E_*)^2-2r\,\parch_2(E_*).
\]

\begin{alphthm}\label{thm:bg-main}
Let \((E_*,\theta)\) be an \(\omega\)-stable regular parabolic Higgs sheaf of rank \(r\) on \((X,D)\). Then
\[
 \int_X\Delta_{\mathrm{par}}(E_*)
 \wedge\frac{\omega^{n-2}}{(n-2)!}\ge0.
\]
\end{alphthm}

The second theorem constructs a pluriharmonic Hermitian metric \(H\) on
\(E|_{X\setminus D}\) adapted to the parabolic structure \(E_*\).
Adaptedness means that the growth of holomorphic sections with respect
to \(H\) recovers the given parabolic filtration. More precisely, on an
SNC chart where \(D=\{z_1\cdots z_\ell=0\}\), and for a real
multi-index \(\mathbf c=(c_1,\ldots,c_\ell)\), let \(P^{\mathbf c}(H)\)
consist of the holomorphic sections on the divisor complement satisfying
\[
 |s|_H\le C_{s,\epsilon}\prod_{i=1}^{\ell}|z_i|^{c_i-\epsilon}
 \qquad\text{for every }\epsilon>0
\]
on every smaller chart. The metric \(H\) is \emph{adapted} to \(E_*\) if
\[
 P^{\mathbf c}(H)=E^{\mathbf c}
 \qquad\text{for every real local multi-index }\mathbf c=(c_i).
\]

\begin{alphthm}
\label{thm:harmonic-main}
Let \((E_*,\theta)\) be a polystable regular parabolic Higgs
sheaf on \((X,D)\). Assume that every stable summand
\((E_{\alpha,*},\theta_\alpha)\) satisfies
\[
 \pardeg_\omega(E_{\alpha,*})=0,
 \qquad
 \int_X\parch_2(E_{\alpha,*})
 \wedge\frac{\omega^{n-2}}{(n-2)!}=0.
\]
Then there is a Hermitian metric \(H\) on \(E|_{X\setminus D}\) such that
\[
 F_H+[\theta,\theta^{\dagger H}]=0,
 \qquad
 \partial_H\theta=0.
\]
In other words, \(H\) is pluriharmonic. Moreover, it is adapted to
\(E_*\), and the Hermitian bundle \((E|_{X\setminus D},H)\) is
acceptable in the sense of Mochizuki
\cite[Definition~2.7]{Moc03-ATHB-I}. The growth filtration
of \(H\), together with \(\theta\), is a regular filtered Higgs bundle
\cite[Subsection~2.5.1, Proposition~2.50]{Moc06-KH2}.
Equivalently, \(E_*\) is a locally abelian parabolic logarithmic Higgs bundle.
\end{alphthm}

\subsection{Outline of the proof}

For both proofs we work with a stable parabolic Higgs sheaf
\((E_*,\theta)\).
We first resolve the singularities of the sheaf and refine its divisor
filtrations. This produces a smooth modification
\(\pi:(Y,B)\to(X,D)\) carrying a logarithmic Higgs bundle
\((V,\theta_Y)\) and a fixed collection \(\mathcal F_B\) of divisor
flags whose graded residues are semisimple. The modification is an
isomorphism away from \(D\), so the metrics constructed on
\(Y\setminus B\) are also metrics on the original open manifold
\(X\setminus D\).

On these fixed flags we choose rational weights approaching values that
recover the original filtration away from the exceptional set. Together
with suitable K\"ahler forms \(\omega_\nu\), these weights give stable
parabolic Higgs bundles \(V_{\nu,*}\) admitting Hermitian--Einstein
metrics \(h_\nu\). The approximation changes the numerical weights,
the K\"ahler form, and the metric, while keeping the modification,
holomorphic bundle, Higgs field, and flags along the divisor fixed.

The two proofs use the curvature estimates for this family in different
ways. For Theorem~\ref{thm:bg-main}, they give inequalities between
parabolic Chern numbers, which pass to the original sheaf as the weights
converge. For Theorem~\ref{thm:harmonic-main}, the vanishing assumptions
force the curvature energies to tend to zero. We then need uniform
metric and boundary estimates to obtain a pluriharmonic limit whose
growth recovers the original parabolic structure.
Figure~\ref{fig:proof-outline} summarizes this common construction and
the two proofs.

\begin{figure}[t]
\centering
\begin{tikzpicture}[
  trunk/.style={draw, rounded corners=2pt, fill=black!3,
    align=center, text width=8.1cm, minimum height=.88cm,
    inner sep=4pt, font=\scriptsize},
  branch/.style={draw, rounded corners=2pt, fill=black!7,
    align=center, text width=5.15cm, minimum height=1.02cm,
    inner sep=4pt, font=\scriptsize},
  conclusion/.style={draw, double, rounded corners=2pt,
    align=center, text width=5.15cm, minimum height=.82cm,
    inner sep=4pt, font=\scriptsize},
  flow/.style={-{Latex[length=1.8mm]}, line width=.45pt},
  node distance=3.5mm
]
\node[trunk] (original) {Original regular parabolic Higgs sheaf
  \((E_*,\theta)\) with arbitrary real weights};
\node[trunk, below=of original] (refine) {Insert saturated
  \(\theta\)-invariant residue filtrations; separate the new weights\\
  \(\Longrightarrow\) graded-semisimple parabolic Higgs sheaves};
\node[trunk, below=of refine] (resolve) {Passage to a root stack,
  flattening, boundary resolution, destackification, and
  an exceptional twist\\
  \(\Longrightarrow\) one fixed model \((Y,B;V,\theta_Y,\mathcal F_B)\)};
\node[trunk, below=of resolve] (rational) {Strict rational weights
  \(w^{(\nu)}\to w^\infty\) and smooth K\"ahler forms
  \(\omega_\nu\to\pi^*\omega\) off \(B\)};
\node[trunk, below=of rational] (he) {Global adapted
  Hermitian--Einstein metrics \(h_\nu\) on the fixed Higgs bundle over
  \(Y\setminus B\)};

\node[branch, below=6mm of he, xshift=-2.75cm] (bg) {Parabolic Chern--Weil identities and curvature energy bounds
  for the global Hermitian--Einstein metrics};
\node[branch, below=6mm of he, xshift=2.75cm] (harm) {Numerical vanishing on each stable summand
  forces the Hitchin--Simpson energies of these metrics to tend to zero};
\node[branch, below=of bg] (bglimit) {Nonnegative stage discriminants;
  coalescence of weights and classes, followed by pushforward to \(X\)};
\node[branch, below=of harm] (compact) {Comparison metrics, scalar barriers, and parabolic stability
  give uniform metric bounds on compact subsets};
\node[conclusion, below=of bglimit] (thma)
  {Theorem~\ref{thm:bg-main}};
\node[branch, below=of compact] (prolong) {A smooth pluriharmonic limit whose growth
  recovers every parabolic lattice; locally free prolongations};
\node[conclusion, below=of prolong] (thmb)
  {Theorem~\ref{thm:harmonic-main}};

\draw[flow] (original) -- (refine);
\draw[flow] (refine) -- (resolve);
\draw[flow] (resolve) -- (rational);
\draw[flow] (rational) -- (he);
\draw[flow] (he) -- (bg);
\draw[flow] (he) -- (harm);
\draw[flow] (bg) -- (bglimit);
\draw[flow] (harm) -- (compact);
\draw[flow] (bglimit) -- (thma);
\draw[flow] (compact) -- (prolong);
\draw[flow] (prolong) -- (thmb);
\end{tikzpicture}
\caption{The common construction and its use in the proofs of
Theorems~\ref{thm:bg-main} and~\ref{thm:harmonic-main}.}
\label{fig:proof-outline}
\end{figure}
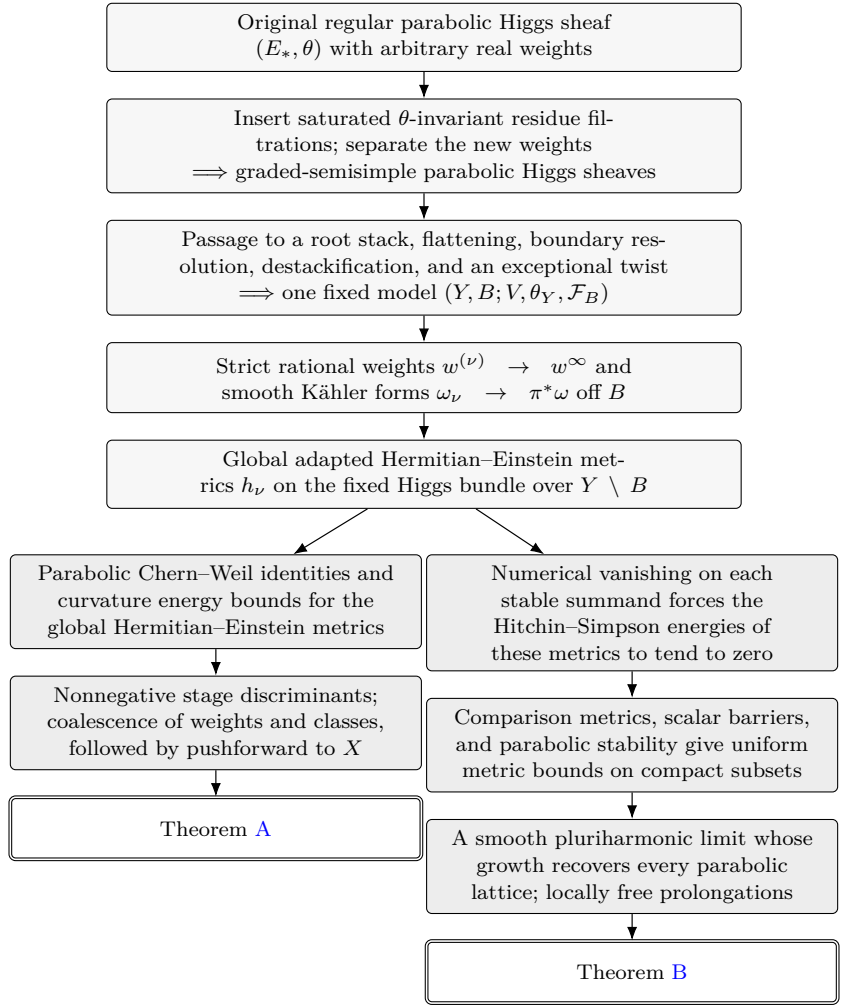

\medskip
\noindent\textit{The fixed model.}
Since the Higgs field preserves the parabolic filtration, its residue
induces an endomorphism on each graded quotient. We refine the divisor
flags by saturated residue-invariant subsheaves
so that the new graded residues are semisimple. Rational auxiliary
weights allow us to encode the refined sheaf on a root stack.
Flattening, boundary resolution, and destackification then give the
fixed logarithmic flag model of
Theorem~\ref{thm:fixed-terminal-data}. These operations leave
\(X\setminus Z\) unchanged, where \(Z\subset D\) has codimension at
least three. Proposition~\ref{prop:fixed-model-correction} supplies
the exceptional twist needed to include the torsion-free pullback of
the original sheaf in the resulting bundle.

We next approximate the original parabolic structure, whose weights may
be irrational, using rational weights on these fixed flags. To specify
the target of this approximation, assign each refined quotient along a strict transform the
weight of the original block containing it, and assign zero weights on
the exceptional components. Denote this weight vector by \(w^\infty\).
Giving the quotients inserted within one original block the same weight
makes them a single filtration jump, recovering the original filtration
on \(X\setminus Z\). We then choose strictly ordered rational weights
\(w^{(\nu)}\to w^\infty\), together with K\"ahler forms \(\omega_\nu\),
so that the resulting parabolic Higgs bundles remain stable, as ensured
by Corollary~\ref{cor:stable-rational-approximations}.
The accompanying K\"ahler forms \(\omega_\nu\) are defined in
\eqref{eq:kahler-schedule-form}; their explicit choice provides the
uniform scalar estimates needed for the metric limit.

\medskip
\noindent\textit{Hermitian--Einstein metrics at each rational stage.}
The parabolic weights define a metric \(K_\nu\), smooth on the
root stack and adapted to the stage filtration.
Proposition~\ref{prop:compact-source-reference} constructs a comparable
metric \(H_{0,\nu}\) whose Hermitian--Einstein error is bounded and
compactly supported. With \(H_{0,\nu}\) as Dirichlet boundary data on an
exhaustion, stability gives estimates independent of the domain and
hence a global Hermitian--Einstein metric \(h_\nu\);
see Theorem~\ref{thm:fixed-stage-he-metrics}. These estimates
may depend on \(\nu\).

\medskip
\noindent\textit{The Bogomolov--Gieseker inequality.}
The curvature integrals of \(K_\nu\) represent the relevant parabolic
Chern numbers. Cutoff and transgression arguments transfer the resulting
bounds to \(h_\nu\), giving both the discriminant inequality
and the curvature energy estimate in
Theorem~\ref{thm:fixed-stage-energy-ledger}.
The parabolic characteristic numbers vary continuously as the weights
coalesce, and Proposition~\ref{prop:parabolic-class-bridge} identifies
their limits with the original numbers on \(X\).
Passing to the limit proves Theorem~\ref{thm:bg-main}.

\medskip
\noindent\textit{The pluriharmonic metric.}
To prove Theorem~\ref{thm:harmonic-main}, we start with the sequence of
Hermitian--Einstein metrics \(h_\nu\) constructed above for each stable
summand. We seek a subsequence which, after constant rescaling, converges
to a smooth positive metric on \(X\setminus D\) adapted to the original
parabolic structure. The preceding estimate gives a uniform bound for
their curvature energies; under the numerical vanishing assumptions,
Corollary~\ref{cor:factor-energy} shows that the energies in fact tend to
zero. Thus a smooth positive limit will be pluriharmonic.

To obtain this limit, Section~\ref{sec:compactness} constructs comparison
metrics \(k_\nu\) and scalar barriers. They reduce uniform metric bounds
to integral bounds for the relative logarithms on a fixed compact domain.
If these integral bounds failed, the normalized Donaldson identity
would produce proper invariant parabolic subsheaves. The degree
identities convert the limiting inequality into a contradiction with
the stability of the original summand.
Proposition~\ref{prop:moving-endpoints} therefore gives smooth
subsequential convergence on compact subsets, and
Corollary~\ref{cor:nonexceptional-growth} identifies the limiting growth
lattices on \(X\setminus Z\).
Proposition~\ref{prop:acceptable-prolongation} extends this identification across \(Z\)
and establishes acceptability of the Hermitian bundle and local freeness
of its growth prolongations.
Taking the finite orthogonal sum proves Theorem~\ref{thm:harmonic-main}.

\subsection{Organization of the paper}

Section~\ref{sec:preliminaries} fixes the definitions and conventions.
Section~\ref{sec:resolution-fixed-model} constructs the logarithmic flag model.
Section~\ref{sec:fixed-model} develops the rational approximations and
their curvature estimates, from which Section~\ref{sec:bg} proves
Theorem~\ref{thm:bg-main}. The extension results in
Section~\ref{sec:analytic-tools} are used in
Section~\ref{sec:compactness}, which proves compactness and boundary
growth estimates and concludes with Theorem~\ref{thm:harmonic-main}.

\subsection{Suggestions for reading}

We apologize to the reader for the length of this paper. Much of the
detail concerns the resolution of the sheaf and its flags, and the
estimates needed near the divisor. The main strategy is nevertheless
straightforward: construct a fixed geometric model, approximate its
weights by rational ones, and study the resulting Hermitian--Einstein
metrics. The preceding outline and Figure~\ref{fig:proof-outline}
describe how this construction leads to the two main theorems.

On a first reading, we suggest focusing on the statements of
Theorem~\ref{thm:fixed-terminal-data} and
Proposition~\ref{prop:fixed-model-correction}, together with
Subsection~\ref{subsec:fixed-flag-weights}. These explain what the
resolution produces and how the original weights are recovered.
The proofs of the intermediate resolution results may be deferred.
Next, read the statements of
Corollary~\ref{cor:stable-rational-approximations},
Theorem~\ref{thm:fixed-stage-he-metrics}, and
Corollary~\ref{cor:factor-energy}: they provide the stable rational
approximations, the sequence of Hermitian--Einstein metrics, and its
energy decay under the hypotheses of Theorem~\ref{thm:harmonic-main}.

With these results in hand, the reader can proceed directly to
Section~\ref{sec:compactness}. We recommend reading carefully the
construction of scalar barriers in
Subsection~\ref{subsec:curvature-barriers}, the use of stability in
Proposition~\ref{prop:normalized-tail-no-escape}, and especially the
convergence proof in Proposition~\ref{prop:moving-endpoints}.
Corollary~\ref{cor:nonexceptional-growth} then explains how the limit
recovers the parabolic growth, and
Subsection~\ref{subsec:adapted-pluriharmonic-metrics} completes the
proof of Theorem~\ref{thm:harmonic-main}. The proofs of the results
from Sections~\ref{sec:fixed-model}--\ref{sec:analytic-tools} used
along this route can be read when they are invoked.

\section{Definitions and conventions}\label{sec:preliminaries}

Unless explicitly stated otherwise, every complex manifold and every
complex-analytic Deligne--Mumford stack used as an ambient space below is
connected.

We use the filtered-sheaf formalism of Mochizuki
\cite[Chapter~3, especially Sections~3.1--3.3]{Mochizuki2006}, with two
conventions adapted to the present paper.  Our indices are decreasing,
whereas Mochizuki uses increasing indices.  We also formulate graded
semisimplicity for reflexive saturated sheaves by testing their torsion-free
graded pieces at the generic points of the boundary components.

\begin{defn}[Regular parabolic Higgs sheaves]\label{defn:all-real}
Let \(X\) be a complex manifold, let
\(D=\bigcup_{i\in I}D_i\) be an SNC divisor, and put
\[
 U=X\setminus D,\qquad j_D:U\hookrightarrow X.
\]
A regular parabolic Higgs sheaf \((E_*,\theta)\) consists of the following
data and conditions.

\begin{enumerate}
\item \(E\) is a reflexive coherent \(\cO_X\)-module and
      \(E_U:=E|_U\) is locally free.
\item For every \(i\in I\), there is a decreasing filtration
      \(\{F_i^aE\}_{a\in\RR}\) by coherent \(\cO_X\)-submodules of
      \(j_{D,*}E_U\) such that
      \[
       F_i^bE\subset F_i^aE\quad(a\le b),\qquad
       F_i^{a+1}E=F_i^aE(-D_i),
      \]
      and
      \[
       F_i^0E=E,\qquad F_i^1E=E(-D_i).
      \]
      The filtration has only finitely many jumps in every unit interval
      and is left-continuous: for every \(a\) there is an
      \(\epsilon>0\) such that
      \[
       F_i^bE=F_i^aE\qquad(a-\epsilon<b\le a).
      \]
\item The filtration is saturated.  More precisely, for
      \(0\le a\le1\), the quotient \(E/F_i^aE\), regarded as an
      \(\cO_{D_i}\)-module, is torsion-free.  Equivalently, if
      \[
       F_i^{>a}E:=\bigcup_{b>a}F_i^bE,
       \qquad
       \Gr_i^a(E_*):=F_i^aE/F_i^{>a}E,
      \]
      then every nonzero \(\Gr_i^a(E_*)\) is torsion-free on \(D_i\).
\item The logarithmic Higgs field
      \[
       \theta:E\longrightarrow
       E\otimes\Omega_X^1(\log D)
      \]
      is integrable, \(\theta\wedge\theta=0\), and preserves every
      filtration term:
      \[
       \theta(F_i^aE)\subset
       F_i^aE\otimes\Omega_X^1(\log D)
       \qquad(i\in I,\ a\in\RR).
      \]
\end{enumerate}

We write parabolic multi-indices in bold, such as
\(\mathbf{c}=(c_i)_{i\in I}\), while their scalar components remain
\(c_i\). Inequalities between multi-indices are componentwise, and
\(\mathbf{e}_i\) denotes the \(i\)-th standard basis vector.
For \(\mathbf{c}=(c_i)\in[0,1)^I\), we use the notation
\[
 E^{\mathbf{c}}:=\bigcap_{i\in I}F_i^{c_i}E
 \subset j_{D,*}E_U,
 \tag{2.1}\label{eq:parabolic-lattice-intersection}
\]
where the intersection is taken inside the common ambient sheaf
\(j_{D,*}E_U\).  Saturation makes every \(F_i^{c_i}E\) reflexive, and
finite intersections of reflexive lattices remain reflexive.
For arbitrary \(\mathbf{c}\in\RR^I\), write uniquely
\(\mathbf{c}=\mathbf{a}+\mathbf{m}\) with \(\mathbf{a}\in[0,1)^I\) and \(\mathbf{m}\in\mathbb Z^I\), and define
\[
 E^{\mathbf{c}}:=E^{\mathbf{a}}\left(-\sum_{i\in I}m_iD_i\right).
 \tag{2.2}\label{eq:periodic-lattice-notation}
\]
Then every \(E^{\mathbf{c}}\) is reflexive, restricts to \(E_U\) on \(U\), and
\[
 E^{\mathbf{d}}\subset E^{\mathbf{c}}\quad(\mathbf{c}\le\mathbf{d}),\qquad
 E^{\mathbf{c}+\mathbf{e}_i}=E^{\mathbf{c}}(-D_i).
\]
\end{defn}

\begin{defn}[Locally abelian parabolic structures]
\label{defn:locally-abelian}
A regular parabolic Higgs sheaf is locally abelian at \(x\in X\) if there
are a neighborhood \(W\) of \(x\) and parabolic line bundles
\(L_{1,*},\ldots,L_{r,*}\) on \((W,D|_W)\) such that its underlying
parabolic sheaf decomposes as
\[
 E_*|_W=\bigoplus_{\nu=1}^rL_{\nu,*}.
\]
The parabolic structure is locally abelian if this holds at every point.
Hereafter, we write ``parabolic bundle'' for ``locally abelian parabolic
bundle.'' We use the same convention for parabolic Higgs bundles.
\end{defn}

\begin{defn}[Graded semisimplicity]
\label{defn:graded-semisimple}
Because \(\theta\) preserves the parabolic filtration, its residue along \(D_i\) induces an endomorphism
\[
 R_{i,a}:=\Gr_i^a\operatorname{Res}_{D_i}(\theta):
 \Gr_i^a(E_*)\longrightarrow\Gr_i^a(E_*)
\]
for every jump \(a\).  The parabolic Higgs sheaf is graded semisimple if
every \(R_{i,a}\) is semisimple at the generic point of \(D_i\), or
equivalently if the nilpotent part of every nonzero graded residue
vanishes.  Since \(\Gr_i^a(E_*)\) is torsion-free, this generic condition
is well-defined.
\end{defn}

\begin{defn}[Parabolic Chern classes]\label{defn:parabolic-chern}
Let \(X\) be compact Kähler of dimension \(n\), and let \(E_*\) be a
regular parabolic sheaf on \((X,D)\). With our decreasing indices, the
weighted formula of \cite[Theorem~1.1 and Section~8.5]{IS06} gives the
following definition of its parabolic Chern character:
\[
 \parch(E_*)=
 \frac{\displaystyle\int_{[0,1]^I}
 e^{-\sum_{i\in I}t_i[D_i]}\operatorname{ch}(E^{-\mathbf{t}})
 \,d\mathbf{t}}
 {\displaystyle\int_{[0,1]^I}
 e^{-\sum_{i\in I}t_i[D_i]}\,d\mathbf{t}}
 \in\bigoplus_{k=0}^n H^{2k}(X,\RR),
 \tag{2.3}\label{eq:weighted-parabolic-chern}
\]
where \(\mathbf{t}=(t_i)_{i\in I}\) and
\(\operatorname{ch}(E^{-\mathbf{t}})\) is the ordinary Chern character of
the coherent lattice. The integrals are taken coefficientwise,
with all expressions truncated above degree \(2n\); the denominator is
invertible because its degree-zero part is one. Only finitely many
lattices occur on the cube, so the same formula applies to real weights.
The negative index \(-\mathbf{t}\) converts the increasing-index convention
of \cite{IS06} to ours.

Writing \(\parch_k(E_*)\) for the degree-\(2k\) component, we define
\[
 \begin{aligned}
 \parc_1(E_*)&=\parch_1(E_*),\\
 \parc_2(E_*)&=\frac12\parc_1(E_*)^2-\parch_2(E_*).
 \end{aligned}
\]
In particular, our sign convention gives
\[
 \parc_1(E_*)=c_1(E)+
 \sum_{i\in I}\sum_{a\in[0,1)}
 a\,\rk\Gr_i^a(E_*)[D_i],
\]
where the inner sum runs over the finitely many jumps along \(D_i\).
\end{defn}

\begin{defn}[Parabolic slope and stability]\label{defn:parabolic-stability}
For a nonzero regular parabolic Higgs sheaf \((E_*,\theta)\) on a compact
Kähler manifold \((X,\omega)\) of dimension \(n\), set
\[
 \begin{aligned}
 \pardeg_\omega(E_*)&=
 \int_X\parc_1(E_*)\wedge\frac{\omega^{n-1}}{(n-1)!},\\
 \mu_\omega(E_*)&=\frac{\pardeg_\omega(E_*)}{\rk E}.
 \end{aligned}
\]
These quantities depend only on \([\omega]\). For a saturated
\(\theta\)-invariant subsheaf \(F\subset E\), let \(F_*\) carry the
filtrations obtained by saturated intersection with those of \(E_*\),
with the inherited weights. The sheaf \((E_*,\theta)\) is
\(\omega\)-stable if
\[
 \mu_\omega(F_*)<\mu_\omega(E_*)
 \qquad\text{for every such }F\text{ with }0<\rk F<\rk E.
\]
Replacing \(<\) by \(\le\) defines semistability. It is polystable if it
is a direct sum of stable parabolic Higgs sheaves of the same parabolic
slope.
\end{defn}

\begin{defn}[Metric growth prolongation]\label{defn:growth}
Let \(E\) be a reflexive coherent sheaf on \(X\) which is locally free on
\(X\setminus D\), let \(j_D:X\setminus D\hookrightarrow X\), and let
\(H\) be a smooth positive Hermitian metric on
\(E|_{X\setminus D}\). The decreasing growth
prolongation \(P^{\mathbf{c}}(H)\subset j_{D,*}(E|_{X\setminus D})\) consists of
the holomorphic sections which, on every sufficiently small SNC chart
where \(D\) has equation \(z_1\cdots z_\ell=0\), satisfy
\[
 |s|_H\le C_{s,\epsilon}\prod_{i=1}^\ell|z_i|^{c_i-\epsilon}
\]
for every \(\epsilon>0\), with a finite constant \(C_{s,\epsilon}\)
allowed to depend on \(s\) and \(\epsilon\).
These prolongations satisfy
\[
 P^{\mathbf{d}}(H)\subset P^{\mathbf{c}}(H)\quad(\mathbf{c}\le\mathbf{d}),\qquad
 P^{\mathbf{c}+\mathbf{e}_i}(H)=P^{\mathbf{c}}(H)(-D_i).
\]
For a regular parabolic structure on \(E\), the torsion-free lattice
\(E^{\mathbf{c}}\) restricts to \(E|_{X\setminus D}\), so restriction gives
the canonical inclusion
\[
 E^{\mathbf{c}}\hookrightarrow j_{D,*}(E^{\mathbf{c}}|_{X\setminus D})
 =j_{D,*}(E|_{X\setminus D}).
\]
Using this inclusion, we say that \(H\) is \emph{adapted} to \(E_*\) if
\[
 P^{\mathbf{c}}(H)=E^{\mathbf{c}}
 \qquad\text{inside }j_{D,*}(E|_{X\setminus D})
 \quad\text{for every }\mathbf{c}\in\RR^I.
\]
\end{defn}

\begin{defn}[Pluriharmonic metrics]\label{defn:pluriharmonic}
Let \(h\) be a Hermitian metric on a holomorphic Higgs bundle
\((V,\theta)\). Denote by \(\partial_h\) and \(F_h\) the \((1,0)\)-part
and the curvature of its Chern connection, respectively, and by
\(\theta^{\dagger h}\) the adjoint of \(\theta\). Set
\[
 \Psi_h=F_h+[\theta,\theta^{\dagger h}],
 \qquad
 B_h=\partial_h\theta.
\]
The Hitchin--Simpson connection is
\[
 D_h=\bar\partial_V+\partial_h+\theta+\theta^{\dagger h}.
\]
The metric \(h\) is pluriharmonic if \(D_h\) is flat. Equivalently,
\[
 \Psi_h=0,\qquad B_h=0.
\]
\end{defn}

\section{Resolution and the fixed logarithmic flag model}
\label{sec:resolution-fixed-model}

We construct a smooth compact K\"ahler modification
\(\pi:(Y,B)\to(X,D)\) carrying one logarithmic Higgs bundle with fixed
divisor flags, on whose graded quotients the residues are semisimple.
The construction preserves the original data outside a set
\(Z\subset D\) of codimension at least three.
Theorem~\ref{thm:fixed-terminal-data} provides the modification and
flags; Subsection~\ref{subsec:fixed-flag-weights} assigns their limiting
weights, and Proposition~\ref{prop:fixed-model-correction} arranges
the pullback inclusion used later in extending subsheaves.

\subsection{Local splitting and refinement of parabolic filtrations}

\begin{lem}
\label{lem:simultaneous-splitting}
Let \(E_*\) be a regular parabolic sheaf on an SNC pair \((X,D)\). Let
\(\Sigma\) be the non-locally-free locus of \(E\), and let
\(Z_{\mathrm{flag}}\subset D\) be the union of the rank-jump loci of the
finite parabolic flags and their intersections, together with the
intersections of three or more components of \(D\). Then every component of
\[
 Z=\Sigma\cup Z_{\mathrm{flag}}
\]
has codimension at least three in \(X\), and
\(E_*|_{X\setminus Z}\) is a parabolic bundle.
\end{lem}

\begin{proof}
Reflexivity implies \(\codim_X\Sigma\ge3\). The parabolic graded pieces
are torsion-free on the smooth divisors \(D_i\), so their rank-jump loci,
and the corresponding rank-jump loci at double crossings, have codimension
at least three in \(X\); the same is true of triple intersections. After
removing these finitely many loci, the parabolic flags are subbundle flags
and admit a common local splitting. Thus the restriction to
\(X\setminus Z\) is a parabolic bundle.
\end{proof}

\begin{lem}
\label{lem:residue-refinement}
Let \(X\) be compact, let \(D\) be an SNC divisor, and let
\((E_*,\theta)\) be a regular parabolic Higgs sheaf on \((X,D)\). For every
nonzero graded block
\[
 Q=\Gr_i^a(E_*),
\]
there is a finite filtration by saturated subsheaves
\[
 0=Q_0\subset Q_1\subset\cdots\subset Q_m=Q
\]
preserved by the residue \(R_{i,a}\) and by the tangential Higgs operators
induced on \(Q\), such that every nonzero quotient \(Q_k/Q_{k-1}\) is
torsion-free and the residue induced on it is semisimple. The inverse images
of the \(Q_k\) under
\[
 F_i^aE\longrightarrow \Gr_i^a(E_*)
\]
insert \(\theta\)-invariant steps between \(F_i^{>a}E\) and \(F_i^aE\).
These refinements give a graded-semisimple regular
parabolic Higgs sheaf without changing any original filtration term.
\end{lem}

\begin{proof}
Fix an original graded block \(Q\) on a connected component of \(D_i\), and let
\(R_i\) be the induced residue. The coefficients of the characteristic
polynomial are holomorphic on the compact component, hence constant.
Factor the minimal polynomial as
\[
 \prod_{\lambda}(T-\lambda)^{m_\lambda}.
\]
The coprime factors give the generalized-eigensheaf decomposition on the
generic locus. Take the saturated kernels of these factors in \(Q\); their
quotients are torsion-free. Inside the \(\lambda\)-summand, put
\[
 K_{\lambda,k}
 =\operatorname{Sat}_Q\ker(R_i-\lambda)^k,
 \qquad 0\le k\le m_\lambda.
\]
After deleting the codimension-two nonfree loci on \(D_i\), these are
subbundles. Since
\[
 (R_i-\lambda)K_{\lambda,k}\subset K_{\lambda,k-1},
\]
the induced residue on
\(K_{\lambda,k}/K_{\lambda,k-1}\) is multiplication by \(\lambda\), hence is
semisimple. Saturation makes every nonzero quotient torsion-free. Taking
inverse images in the original filtration term lifts these chains without
altering either endpoint of the block. Integrability makes every
tangential logarithmic Higgs coefficient commute with \(R_i\); it therefore
preserves the generalized eigensheaves and all the saturated kernels \(K_{\lambda,k}\).
\end{proof}

\subsection{Parabolic sheaves and analytic root stacks}

The equivalence between parabolic bundles with fixed rational denominators
and vector bundles on the corresponding root stack is due to Borne
\cite{Borne2007}. Borne--Vistoli extended it to parabolic sheaves on
logarithmic schemes \cite{BorneVistoli}. The following proposition gives the
analytic logarithmic-Higgs version needed here.

For denominators \(\mathbf r=(r_i)\), a coherent parabolic system means
the version of \cite[Definition~5.6]{BorneVistoli} with coherent
analytic terms and decreasing indices.
Its transition maps commute, and the periodicity isomorphisms
\[
 E^{(\mathbf k+r_i\mathbf e_i)/\mathbf r}
 \simeq E^{\mathbf k/\mathbf r}(-D_i)
\]
commute with one another and with the transitions. Under these
isomorphisms, the composite of \(r_i\) successive transitions in the
\(i\)-th direction must be the canonical divisor morphism
\(E^{\mathbf k/\mathbf r}(-D_i)\to E^{\mathbf k/\mathbf r}\).
Morphisms commute with all these maps. The transitions need not be
injective, unlike the filtration inclusions in
Definition~\ref{defn:all-real}.

\begin{prop}
\label{prop:analytic-root-correspondence}
Let \(X\) be a complex manifold, let
\(D=\sum_{i\in I}D_i\) be an ordered SNC divisor, and choose a tuple
\(\mathbf r=(r_i)\) of positive integers. Write
\[
 p:\mathcal R=\sqrt[(r_i)]{(X,D)}\longrightarrow X
\]
for the analytic root stack, with tautological line bundles
\(\mathcal N_i\) and reduced boundary \(D_{\mathcal R}\).
The assignment
\[
 \mathcal F\longmapsto
 \left\{E^{\mathbf k/\mathbf r}
 =p_*\left(\mathcal F\otimes\bigotimes_i\mathcal N_i^{-k_i}\right)
 \right\}_{\mathbf k\in\mathbb Z^I},
\]
where \(\mathbf k/\mathbf r=(k_i/r_i)_{i\in I}\), defines an exact
equivalence between coherent analytic sheaves on \(\mathcal R\) and
coherent parabolic systems on \((X,D)\) with these denominators.

Under this equivalence, vector bundles correspond exactly to locally
abelian parabolic bundles. It extends to an equivalence for integrable
logarithmic Higgs sheaves, with poles along \(D_{\mathcal R}\) upstairs
and compatible logarithmic Higgs fields along \(D\) on the parabolic
terms downstairs. If every term \(E^{\mathbf k/\mathbf r}\) is
torsion-free, then \(\mathcal F\) is torsion-free.
\end{prop}

\begin{proof}
We follow the graded-module proof of
\cite[Theorem~6.1 and Example~6.2]{BorneVistoli}, with the following
analytic and logarithmic-Higgs adaptations. Our decreasing index
\(\mathbf k/\mathbf r\) corresponds to the index
\(-\mathbf k/\mathbf r\) there, accounting for the inverse
tautological twists in the statement.

On a coordinate polydisc \(P\) with
\(D\cap P=\{z_1\cdots z_s=0\}\), the root stack is
\([\widetilde P/G]\), where \(\widetilde P\to P\) is the finite cover
with algebra
\[
 B=\cO_P[w_1,\ldots,w_s]/(w_i^{r_i}-z_i),
 \qquad G=\prod_{i=1}^s\mu_{r_i}.
\]
Coherent analytic sheaves on \([\widetilde P/G]\) are represented by
coherent \(B\)-modules with a compatible \(G\)-action. Give \(w_i\) character
\(-\mathbf e_i\) and a local frame \(n_i\) of \(\mathcal N_i\)
character \(\mathbf e_i\), so that the tautological section \(w_in_i\)
is invariant. The functor in the statement selects the character component
of a module \(M\) indexed by the class
\(\overline{\mathbf k}\in\prod_i\mathbb Z/r_i\mathbb Z\).
In this chart, the graded reconstruction in the cited proof reduces,
by periodicity, to the finitely many terms \(0\le k_i<r_i\), with
the transitions giving multiplication by \(w_i\).
The periodicity condition gives the relation
\(w_i^{r_i}=z_i\). These operations preserve analytic coherence:
\(B\) is finite locally free over \(\cO_P\), and character projections
are finite averages. The same observations prove exactness.

The functor is defined by the global tautological line
bundles and sections and commutes with restriction. Local full
faithfulness lifts the identity of a restricted parabolic system to
unique overlap isomorphisms between its reconstructed sheaves.
Uniqueness gives the cocycle condition, and coherent analytic sheaves
glue for open covers. Thus the local equivalences give the asserted
global equivalence.

At the center of a root chart, a character basis of a vector-bundle
fiber lifts, after averaging and shrinking, to an equivariant frame.
The resulting character lines correspond to rank-one parabolic
bundles, and conversely their reconstruction is locally free. This
proves the vector-bundle assertion.

For logarithmic Higgs fields, the local identity
\(p^*(dz_i/z_i)=r_i\,dw_i/w_i\) gives
\[
 p^*\Omega_X^1(\log D)
 \simeq\Omega_{\mathcal R}^1(\log D_{\mathcal R}).
\]
The logarithmic coordinate forms are \(G\)-invariant. Hence equivariant
\(B\)-linear Higgs fields correspond to fields on the character
components compatible with transitions and periodicity.
The condition \(\theta\wedge\theta=0\) is checked componentwise, so
the equivalence preserves integrability.

Finally, work at the center \(x\) of a root chart, with
\(A=\cO_{P,x}\). The algebra \(B_x\) is a domain with \(B_x^G=A\).
If all parabolic terms are torsion-free, their finite character
decomposition makes \(M_x\) torsion-free over \(A\).
Every nonzero \(b\in B_x\) has the nonzero invariant multiple
\(\prod_{g\in G}g(b)\in A\). Thus \(M_x\) is also torsion-free over
\(B_x\), proving the final assertion.
\end{proof}

\subsection{Canonical flattening and boundary-compatible resolution}

The next propositions make a torsion-free logarithmic Higgs sheaf locally
free after a modification, with an SNC boundary on the source. The
construction leaves unchanged a prescribed open subset where the sheaf
is already locally free.

\begin{prop}
\label{prop:grassmannian-flattening}
Let \(\mathcal X\) be a smooth reduced complex-analytic Deligne--Mumford
stack, and let \(\mathcal F\) be a torsion-free coherent sheaf of constant
rank \(r\). Let \(\mathcal L\subset\mathcal X\) be its locally free
locus. There exist a reduced analytic Deligne--Mumford stack
\(\mathcal X_{\mathcal F}\), a representable projective proper modification
\[
 \rho:\mathcal X_{\mathcal F}\longrightarrow\mathcal X
\]
which is an isomorphism over \(\mathcal L\), and a rank-\(r\) vector bundle
\[
 \mathcal Q
 =\rho^*\mathcal F/\Tor(\rho^*\mathcal F).
\]
The construction is canonical under automorphisms of
\((\mathcal X,\mathcal F)\), and is equivariant on every finite quotient
chart.
\end{prop}

\begin{proof}
Let
\[
 \mathcal G=\operatorname{Gr}_{\mathcal X}(\mathcal F,r)
\]
be the relative analytic Grassmannian of locally free rank-\(r\) quotients.
On an atlas chart \(W\), choose a finite presentation
\[
 \cO_W^a\longrightarrow\cO_W^b\longrightarrow\mathcal F|_W
 \longrightarrow0.
\]
The ordinary relative Grassmannian of rank-\(r\) quotients of \(\cO_W^b\)
is \(W\times\operatorname{Gr}(r,b)\). The condition that the image of
\(\cO_W^a\) vanish in the universal quotient is a closed analytic condition.
These local spaces represent the same quotient functor on overlaps, so they
glue to a representable analytic stack \(\mathcal G\), projective over
\(\mathcal X\), with a universal quotient
\[
 p^*\mathcal F\twoheadrightarrow\mathcal Q_{\mathcal G}.
\]

Over \(\mathcal L\), a surjection between locally free sheaves of the same
rank is an isomorphism. Thus \(\mathcal G|_{\mathcal L}\) is the section
defined by the identity quotient. Define \(\mathcal X_{\mathcal F}\) to be its
reduced analytic closure in \(\mathcal G\). Over each connected component
of the smooth stack \(\mathcal X\), this closure is irreducible. Its
restriction \(\rho\) is
projective and proper, its image is closed and contains the dense open
\(\mathcal L\), and it is the identity over \(\mathcal L\).

Restrict the universal quotient to obtain
\[
 \rho^*\mathcal F\twoheadrightarrow\mathcal Q.
\]
Its kernel vanishes generically, hence is torsion. Conversely, a torsion
section maps to zero in the locally free target, since multiplication by a
local non-zero-divisor is injective on a
free module over the reduced source. The kernel is therefore exactly
\(\Tor(\rho^*\mathcal F)\), proving the asserted description of \(\mathcal Q\).

An automorphism of \((\mathcal X,\mathcal F)\) acts on the represented
quotient functor, preserves the identity section, and hence preserves its
reduced closure and universal quotient. This proves canonicity and
equivariance.
\end{proof}

\begin{prop}
\label{prop:relative-boundary-resolution}
Let \(\mathcal X\) be a reduced compact complex-analytic Deligne--Mumford
stack with a finite ordered divisorial boundary \(\mathcal B\), and let
\(\mathcal W\subset\mathcal X\) be an open subset on which
\((\mathcal X,\mathcal B)\) is a smooth SNC pair. Then there is a finite
composition of representable analytic blow-ups in smooth centers,
\[
 g:(\mathcal X',\mathcal B')
 \longrightarrow(\mathcal X,\mathcal B)
\]
such that every center is disjoint from the inverse image of
\(\mathcal W\). Here \(\mathcal B'\) is the ordered complete transform of
\(\mathcal B\); in particular, \(g\) is boundary-compatible, i.e.,
\[
 g^{-1}\bigl(\Supp\mathcal B\bigr)\subset\Supp\mathcal B'.
\]
Moreover, \(\mathcal X'\) is smooth and compact, \(\mathcal B'\) is SNC,
and \(g\) is a representable projective proper morphism which is an
isomorphism over \(\mathcal W\). If the stabilizers of \(\mathcal X\) are
finite abelian, the same holds for \(\mathcal X'\); if the stabilizer of
\(\mathcal X\) is trivial on a dense open of every irreducible component,
the same holds for \(\mathcal X'\).
\end{prop}

\begin{proof}
Apply the strong boundary desingularization of
\cite[Thms.~1.1.6 and~1.1.13]{Temkin12} to an analytic étale atlas
\(U\to\mathcal X\). We use ``semi-regular'' and ``strict regular'' in
the terminology of that reference. The centers avoid all inverse images
of the semi-regular locus, and the construction is functorial for strict regular morphisms after
empty blow-ups are omitted. For
\[
 R=U\times_{\mathcal X}U,
\]
both projections to \(U\) are étale and strict regular for the pulled-back
ordered boundaries. The two pullbacks of the resolution to \(R\) therefore
agree. The centers, blow-ups, and ordered complete transforms descend stage
by stage through the analytic groupoid.

We check that only finitely many blow-ups are needed. When the atlas is
not quasi-compact, the analytic theorem gives a locally finite
hypersequence. Every point of the compact stack has a
neighborhood meeting only finitely many nonempty centers. A finite
subcover therefore meets only finitely many stages, so deletion of empty
stages gives the required finite sequence of blow-ups.
Smoothness of the centers and terminal atlas is étale-local. The terminal
complete transform is a divisorial semi-regular boundary on a smooth stack,
hence is SNC. By construction it contains the reduced inverse image of the
original boundary.

The open subset \(\mathcal W\) is contained in the semi-regular locus, so
every center avoids its inverse image. Analytic blow-ups are representable,
projective, and proper, and
these properties are stable under composition. For a representable
morphism, each stabilizer upstairs injects into a stabilizer downstairs;
subgroups of finite abelian groups remain finite abelian. Finally, the
composite is an isomorphism on a dense open subset of each original
component. Intersecting this subset with the locus of trivial stabilizers
gives a dense open with trivial stabilizers on its strict transform.
Thus generic triviality is preserved.
\end{proof}

\begin{lem}
\label{lem:total-torsion-free-pullback}
Let \(h:\mathcal T\to\mathcal X\) and
\(\delta:\mathcal Y\to\mathcal T\) be holomorphic morphisms between smooth
reduced analytic Deligne--Mumford stacks. Let \(\mathcal F\) be torsion-free
on \(\mathcal X\), and let
\[
 q_T:h^*\mathcal F\twoheadrightarrow
 \mathcal V_T:=h^*\mathcal F/\Tor(h^*\mathcal F)
\]
be the canonical quotient map. Suppose that \(\mathcal V_T\) is locally
free and that \(q_T\) is an isomorphism over an open subset
\(\mathcal W\subset\mathcal T\) such that
\(\delta^{-1}(\mathcal W)\) is dense in \(\mathcal Y\). If
\(f=h\circ\delta\), then \(q_T\)
induces an isomorphism
\[
 \overline q:\mathcal V_Y
 :=f^*\mathcal F/\Tor(f^*\mathcal F)
 \xrightarrow{\ \simeq\ }\delta^*\mathcal V_T.
 \tag{3.1}\label{eq:composed-tf-pullback}
\]

Suppose in addition that the three stacks carry SNC divisors
\(D_{\mathcal X},D_{\mathcal T},D_{\mathcal Y}\), and that
\[
 h:(\mathcal T,D_{\mathcal T})\longrightarrow
   (\mathcal X,D_{\mathcal X}),\qquad
 \delta:(\mathcal Y,D_{\mathcal Y})\longrightarrow
   (\mathcal T,D_{\mathcal T})
\]
are boundary-compatible. Thus there are natural logarithmic differential
maps
\[
\begin{aligned}
 d_{\log}h:\;&h^*\Omega_{\mathcal X}^1(\log D_{\mathcal X})
 \longrightarrow\Omega_{\mathcal T}^1(\log D_{\mathcal T}),\\
 d_{\log}\delta:\;&
 \delta^*\Omega_{\mathcal T}^1(\log D_{\mathcal T})
 \longrightarrow\Omega_{\mathcal Y}^1(\log D_{\mathcal Y}).
\end{aligned}
\]
Let \(\Theta\) be an integrable logarithmic Higgs field on \(\mathcal F\).
It induces a Higgs field \(\Theta_T\) on \(\mathcal V_T\) and a Higgs
field \(\Theta_Y\) on \(\mathcal V_Y\). Then
\eqref{eq:composed-tf-pullback} is an isomorphism of logarithmic Higgs
bundles; equivalently, the following diagram commutes:
\[
\begin{tikzcd}[column sep=5.2em,row sep=2.8em]
 \mathcal V_Y
   \arrow[r,"\overline q","\sim"']
   \arrow[d,"\Theta_Y"']
 & \delta^*\mathcal V_T
   \arrow[d,"\delta^*\Theta_T"] \\
 \mathcal V_Y\otimes\Omega_{\mathcal Y}^1(\log D_{\mathcal Y})
   \arrow[r,"\overline q\otimes\Id"']
 & \delta^*\mathcal V_T\otimes
   \Omega_{\mathcal Y}^1(\log D_{\mathcal Y}).
\end{tikzcd}
\]
Here \(\delta^*\Theta_T\) denotes logarithmic pullback using
\(d_{\log}\delta\).
\end{lem}

\begin{proof}
Pullback is right exact, so pulling \(q_T\) back by \(\delta\) gives the
natural surjection
\[
 q:=\delta^*q_T:
 f^*\mathcal F\twoheadrightarrow\delta^*\mathcal V_T.
\]
It is an isomorphism over \(\delta^{-1}(\mathcal W)\). Its kernel therefore
vanishes on a dense open subset and is torsion.
Conversely, every torsion section maps to zero in the locally free target.
Thus
\[
 \ker q=\Tor(f^*\mathcal F),
\]
so \(q\) induces the isomorphism \(\overline q\) in
\eqref{eq:composed-tf-pullback}.

For the final assertion, let \(\widetilde\Theta_T\) be the logarithmic
Higgs operator on \(h^*\mathcal F\) obtained by pulling \(\Theta\) back
using \(d_{\log}h\). It preserves the torsion subsheaf: if \(bs=0\) for a
local non-zero-divisor \(b\), then
\[
 b\widetilde\Theta_T(s)=\widetilde\Theta_T(bs)=0.
\]
Since the logarithmic cotangent sheaf is locally free, the operator
therefore descends to \(\Theta_T\) on \(\mathcal V_T\).

The logarithmic differential for \(f=h\circ\delta\) is the composite
\[
 f^*\Omega_{\mathcal X}^1(\log D_{\mathcal X})
 \xrightarrow{\ \delta^*(d_{\log}h)\ }
 \delta^*\Omega_{\mathcal T}^1(\log D_{\mathcal T})
 \xrightarrow{\ d_{\log}\delta\ }
 \Omega_{\mathcal Y}^1(\log D_{\mathcal Y}).
\]
The same torsion-invariance argument applied to the direct pullback by
\(f\) gives \(\Theta_Y\) on \(\mathcal V_Y\). Compatibility of the
logarithmic differentials implies that \(q=\delta^*q_T\) intertwines the
two pulled-back Higgs operators. Passing to the torsion-free quotient gives
the displayed commutative diagram. Finally, pullback and passage to a
Higgs-invariant quotient preserve the equation
\(\Theta\wedge\Theta=0\), so all induced Higgs fields remain integrable.
\end{proof}

\begin{prop}
\label{prop:protected-total-boundary}
Let \(\mathcal X\) be a compact smooth reduced complex-analytic
Deligne--Mumford stack, let \(D_{\mathcal X}\) be an SNC divisor,
and let \((\mathcal F,\Theta)\) be a torsion-free coherent
logarithmic Higgs sheaf on \((\mathcal X,D_{\mathcal X})\). Let
\(\mathcal W\subset\mathcal X\) be a dense open subset, assume that
\(\mathcal X\setminus\mathcal W\) is closed analytic,
and assume that \(\mathcal F|_{\mathcal W}\) is locally free.

Then there exist a smooth compact complex-analytic Deligne--Mumford stack
\(\mathcal T\), a representable projective proper modification
\[
 h:\mathcal T\longrightarrow\mathcal X
\]
which is an isomorphism over \(\mathcal W\), an SNC divisor
\(\Delta\) on \(\mathcal T\), and a logarithmic Higgs bundle
\((\mathcal V,\Theta_{\mathcal T})\) on \((\mathcal T,\Delta)\) such
that
\[
 \Supp(\Delta)\supset
 \Supp\bigl(h^{-1}D_{\mathcal X}\bigr)_{\mathrm{red}}
 \cup h^{-1}(\mathcal X\setminus\mathcal W),
\]
and canonically
\[
 \mathcal V
 =h^*\mathcal F/\Tor(h^*\mathcal F).
\]
Finite abelian stabilizers and trivial generic stabilizers are preserved.
\end{prop}

\begin{proof}
Apply Proposition~\ref{prop:grassmannian-flattening} to \(\mathcal F\).
It gives a reduced analytic Deligne--Mumford stack
\(\mathcal X_{\mathcal F}\), a representable projective proper
modification
\[
 \rho:\mathcal X_{\mathcal F}\longrightarrow\mathcal X,
\]
and a vector bundle
\[
 \mathcal Q
 =\rho^*\mathcal F/\Tor(\rho^*\mathcal F).
\]
The modification is an isomorphism over the locally free locus of
\(\mathcal F\), hence over \(\mathcal W\).

Let \(S\) be the reduced closed analytic complement of
\(\rho^{-1}(\mathcal W)\), and let \(\mathcal I_S\) be its
coherent ideal.
Since \(\mathcal X_{\mathcal F}\) is the reduced closure of the
section defined by the identity quotient and \(\mathcal W\) is dense in \(\mathcal X\),
the open subset \(\rho^{-1}(\mathcal W)\) is dense in
\(\mathcal X_{\mathcal F}\). Thus \(S\) is a proper closed analytic subset.
Form the analytic blow-up
\[
 \beta:\mathcal X_P
 =\operatorname{Bl}_{\mathcal I_S}\mathcal X_{\mathcal F}
 \longrightarrow\mathcal X_{\mathcal F}.
\]
It is representable, projective, proper, and an isomorphism outside
\(S\). Its universal property gives
\[
 \mathcal I_S\mathcal O_{\mathcal X_P}
 =\mathcal O_{\mathcal X_P}(-A)
\]
for an effective Cartier divisor \(A\) whose support is the inverse
image of \(S\). The blow-up is reduced: locally its Rees algebra
is a graded subalgebra of a polynomial algebra over the reduced local ring
of \(\mathcal X_{\mathcal F}\), and therefore contains no nonzero nilpotent
element.

On \(\mathcal X_P\), let \(\mathcal B_P\) be the ordered divisorial
boundary whose support is the union of \(A\) and the reduced divisorial
inverse image of \(D_{\mathcal X}\). Over the inverse image of
\(\mathcal W\), the stack is the original smooth stack and
\(\mathcal B_P\) is the original SNC divisor.
Apply Proposition~\ref{prop:relative-boundary-resolution} with the open
subset \((\rho\circ\beta)^{-1}(\mathcal W)\). It gives a representable
projective proper map
\[
 g:(\mathcal T,\Delta)\longrightarrow(\mathcal X_P,\mathcal B_P)
\]
which is an isomorphism over \(\mathcal W\), where \(\mathcal T\) is
smooth and compact and \(\Delta\) is the SNC complete transform of
\(\mathcal B_P\). Put
\[
 h=\rho\circ\beta\circ g,\qquad
 \mathcal V=g^*\beta^*\mathcal Q.
\]
All three factors are projective and proper and are isomorphisms over
\(\mathcal W\), so the same holds for \(h\). Since the complete transform
contains the inverse images of \(A\) and \(D_{\mathcal X}\),
the asserted containment of supports follows.

There is a natural surjection
\[
 h^*\mathcal F\twoheadrightarrow\mathcal V.
\]
It is an isomorphism over \(h^{-1}(\mathcal W)\), which is dense in
\(\mathcal T\). Its kernel is therefore torsion.
Conversely, every torsion section maps to zero in the locally free target.
Thus the kernel is exactly \(\Tor(h^*\mathcal F)\), proving the asserted identification of \(\mathcal V\).

Since \(\Delta\) contains the reduced inverse image of
\(D_{\mathcal X}\), pullback gives
\[
 h^*\Omega_{\mathcal X}^1(\log D_{\mathcal X})
 \longrightarrow\Omega_{\mathcal T}^1(\log\Delta).
\]
The torsion-invariance calculation in
Lemma~\ref{lem:total-torsion-free-pullback} shows that the pulled-back Higgs
field descends to \(\mathcal V\). Pullback and passage to this quotient
commute with wedge products, so the descended field remains integrable.

Finally, the flattening, the principalizing blow-up, and every boundary
resolution blow-up are representable. Stabilizers upstairs therefore
inject into stabilizers downstairs, so finite abelian stabilizers are
preserved. If the stabilizer of \(\mathcal X\) is trivial on a dense open
subset \(\mathcal U\), then \(\mathcal U\cap\mathcal W\) is dense and its
unchanged inverse image in \(\mathcal T\) has trivial stabilizer. Thus
generic triviality is preserved.
\end{proof}

\subsection{Analytic destackification}\label{sec:destackification}

A smooth stack need not have a smooth coarse space. To apply the inverse
root-stack correspondence after resolving the sheaf, we need a root stack
over a smooth manifold. The following theorem constructs such a stack
without changing a specified open subset where the required local
structure is already present.

\begin{thm}\label{thm:analytic-destackification}
Let \(\mathcal S\) be a compact smooth analytic Deligne--Mumford stack.
Assume that its diagonal is closed with finite fibers, that every
stabilizer is finite abelian, and that the stabilizer is trivial on a dense
open subset. Let \(S\) be its analytic coarse space, and let
\[
 \boldsymbol{\mathcal C}=(\mathcal C^1,\ldots,\mathcal C^m)
\]
be an ordered simple normal crossing divisor on \(\mathcal S\).

Let \(\mathcal W\subset\mathcal S\) be a dense open subset. Assume that
every \(x\in\mathcal W\) has a diagonal quotient chart
\([P_x/G_x]\), centered at a fixed lift of \(x\),
with coordinates \(z_1,\ldots,z_n\), and with boundary
\(z_1\cdots z_s=0\). If
\[
 A_x=\Hom(G_x,\CC^*)
\]
and \(a_j\in A_x\) is the character of \(z_j\), assume
\[
 a_j=0\quad(j>s),\qquad
 A_x=\langle a_1\rangle\oplus\cdots\oplus\langle a_s\rangle,
 \tag{3.2}\label{eq:root-clean}
\]
where trivial cyclic factors are omitted.

Then there is a finite sequence
\[
 (\mathcal Y,\boldsymbol{\mathcal B})
 =(\mathcal S_N,\boldsymbol{\mathcal C}_N)
 \longrightarrow\cdots\longrightarrow
 (\mathcal S_0,\boldsymbol{\mathcal C}_0)
 =(\mathcal S,\boldsymbol{\mathcal C})
\]
with the following properties.

\begin{enumerate}
\item Each morphism is boundary-compatible and is either a blow-up in a
smooth closed substack of positive codimension having normal crossings
with the boundary, or a finite root construction along a boundary
component.

\item Every blow-up center and every rooted boundary component is disjoint from the inverse image of \(\mathcal W\). Thus the composite is an isomorphism over \(\mathcal W\).

\item The analytic coarse space \(Y\) of \(\mathcal Y\) is a smooth
compact complex manifold. After separating the connected components of
each coarse boundary divisor, the resulting ordered boundary
\[
 \boldsymbol B=(B^1,\ldots,B^q)
\]
is simple normal crossing, there are positive integers \(d_1,\ldots,d_q\), and
\[
 \mathcal Y\simeq
 \sqrt[(d_1,\ldots,d_q)]{(Y,\boldsymbol B)}.
 \tag{3.3}\label{eq:terminal-root}
\]

\item The induced coarse morphism
\[
 \delta\colon Y\longrightarrow S
\]
is projective and proper and is an isomorphism over the coarse image of \(\mathcal W\).
\end{enumerate}
\end{thm}

\begin{proof}
We use the destackification algorithm of
\cite[Theorem~1.2 and Algorithms~A--E]{Bergh17}. Its rules for choosing
centers and its termination arguments are expressed in terms of finite
stabilizer representations and coordinate blow-ups. The aim is to arrange
that the stabilizer acts only in independent boundary directions, as in
\eqref{eq:root-clean}. We explain why the algorithm defines analytic
modifications and how the resulting stack has the required coarse space.

\smallskip
\noindent\emph{Local coordinates.}
Finite quotient charts are supplied by
\cite[Theorem~3.1]{Doan24}. We first choose coordinates compatible with
both the group action and the boundary.

Fix \(x\in\mathcal S\), a quotient chart \([P_x/G_x]\) about \(x\),
and a lift \(p\in P_x\) fixed by \(G_x\). The boundary branches through
\(p\) are \(G_x\)-invariant. Let
\(\ell_1,\ldots,\ell_s\in T_p^*P_x\) be their nonzero conormal
covectors. They are eigenvectors, and abelianness allows us to extend
them to an eigenbasis
\(\ell_1,\ldots,\ell_n\).

For \(j\le s\), choose a defining germ \(f_j\) for the corresponding
branch with \(df_j(p)=\ell_j\). If \(a_j\) is the character of
\(\ell_j\), put
\[
 f_j^{\mathrm{eig}}
 =\frac1{|G_x|}\sum_{g\in G_x}a_j(g)^{-1}g^*f_j.
\]
Every summand vanishes on the same branch and has differential \(\ell_j\)
at \(p\). Thus \(f_j^{\mathrm{eig}}\) is still a defining germ for that
branch. Applying the same projectors to germs with differentials
\(\ell_j\), \(j>s\), gives a coordinate map with invertible differential.
The holomorphic inverse-function theorem gives diagonal equivariant
coordinates after shrinking to an invariant neighborhood. In these
coordinates the boundary branches are coordinate hyperplanes.

\smallskip
\noindent\emph{Centers and gluing.}
The algorithm uses the stabilizer representation on the cotangent space,
with the conormal lines of the ordered boundary components marked. This
representation is intrinsic, although its expression by characters
\(a_1,\ldots,a_n\) uses the chosen coordinates. We use the conormal
invariants of \cite[Sections~6--8]{Bergh17}, with their definitions and
ordering as in that reference.

In a diagonal chart, the stabilizer at a point is the subgroup on which
all characters of its nonzero coordinates vanish. Thus the character data
on each coordinate stratum are the same as in the corresponding
algebraic diagonal representation. Bergh's coordinate calculations for
upper semicontinuity and the smoothness of the chosen maximal loci
therefore apply here. Locally, each center has the form
\[
 Z=\{z_j=0:j\in J\},
 \tag{3.4}\label{eq:maximal-center}
\]
where \(J\) is the set of coordinate directions selected by the
invariant. It is a smooth analytic submanifold with normal crossings
with the boundary. For the aggregate invariant in Algorithm~E, its
maximal locus is a union of connected components of the smooth maximal
locus of the independency index; this is the property of the invariants
established in \cite[Section~7]{Bergh17}.

On overlapping quotient charts the representations, marked boundary
lines, and selected loci agree. Their coordinate ideals define the same
reduced analytic submanifold on a common étale refinement, so these
ideals glue. Consequently each center is a globally defined smooth
closed substack. Intersections of the ordered boundary components, used
as centers in the toric part of the algorithm, are likewise globally
defined.

\smallskip
\noindent\emph{The modification sequence.}
The analytic blow-up of \eqref{eq:maximal-center} has charts with
coordinates \(z_p\), \(z_j/z_p\) for \(j\in J\setminus\{p\}\), and
the unchanged coordinates, where \(p\in J\). Its character rule is
\[
 a_j\longmapsto a_j-a_p
 \quad\text{for }j\in J\setminus\{p\},
 \tag{3.5}\label{eq:character-blowup}
\]
with \(a_p\) corresponding to the exceptional divisor. A root of order
\(c\) along \(z=0\) is given by \(t^c=z\), with the corresponding
extension of the character group. These are the same coordinate rules
as in the algebraic algorithm, and they commute with restriction and
changes of quotient chart.

We can therefore perform Algorithms~C and~E on the analytic stack,
including the toric procedures and the procedures along divisors used
in~E. After
each blow-up we take the ordered complete transform of the boundary;
after each root construction we take its reduced inverse image. The
coordinate descriptions show that all stacks remain smooth and all
boundaries remain SNC, and that every transition is boundary-compatible.

For finiteness, we use the termination proofs of these algorithms in
\cite[Sections~4 and~8]{Bergh17}. Their decreasing invariants depend on
\eqref{eq:character-blowup} and the root transformations: Algorithm~C reduces the
divisorial index, the inner algorithms terminate, and each completed
iteration of~E reduces its maximal lexicographic invariant. At each
analytic stage compactness gives finitely many quotient charts, hence
finitely many values of the invariants and finitely many components of
the centers and boundary. The same termination arguments thus give a
finite global analytic sequence.

On \(\mathcal W\), condition \eqref{eq:root-clean} says that the group
acts only in the boundary directions and that their characters are
independent. The centers selected by Algorithms~C and~E therefore avoid
this open subset. In each iteration of~E, the first exceptional divisor
is marked as distinguished, and the subsequent operations are supported
over that iteration's first center, as in the construction of
\cite[Algorithm~E]{Bergh17}. Induction shows that every blow-up and root
construction is an isomorphism over \(\mathcal W\).

\smallskip
\noindent\emph{The coarse space.}
An analytic blow-up is representable and projective, whereas a root
construction induces the identity on coarse spaces. To check
projectivity for a blow-up on coarse spaces, choose a relatively ample
line bundle \(L\) on the stack. Compactness and finite inertia give a
common multiple \(M\) of the stabilizer orders. Every stabilizer acts
trivially on \(L^M\), so invariant sections on finite quotient charts
descend it to a line bundle on the coarse space. These descents agree on
overlaps, and relative ampleness can be checked after finite quotient
pullback. Thus the induced coarse morphism is projective; composition
gives a projective proper map \(Y\to S\).

At the final stage the divisorial and independency indices vanish.
Thus every character is generated by the boundary characters and the
cyclic groups generated by the individual coordinate characters are
independent. It follows that \(a_j=0\) for \(j>s\). The stabilizer
action is effective because the generic stabilizer is trivial; hence
the remaining characters generate its entire character group. Therefore
\[
 A_x=\langle a_1\rangle\oplus\cdots\oplus\langle a_s\rangle.
\]
If \(d_i\) is the order of \(a_i\), then
\[
 \CC\{z_1,\ldots,z_n\}^{G_x}
 =
 \CC\{z_1^{d_1},\ldots,z_s^{d_s},z_{s+1},\ldots,z_n\}.
\]
This identifies the quotient chart with a root stack over a smooth
polydisc with SNC boundary.

The ramification order is locally constant along each smooth coarse
boundary divisor. Separating its connected components makes the orders
\(d_i\) constant. If \(q:\mathcal Y\to Y\) is the coarse projection,
the divisor identity
\[
 d_i\mathcal B^i=q^*B^i
\]
provides the root line bundle and tautological section globally. The
universal property gives a morphism to the root stack in
\eqref{eq:terminal-root}. By the local calculation this is an
isomorphism. Finally, \(Y\) is compact because \(\mathcal Y\) is
compact, and \(Y\to S\) is an isomorphism over the image of
\(\mathcal W\), as proved above.
\end{proof}

\subsection{Resolving the logarithmic flags}
\label{subsec:resolving-logarithmic-flags}

We now apply the preceding constructions to the refined parabolic flags.
The result is one smooth modification carrying a logarithmic Higgs bundle
and locally split flags.

\begin{thm}
\label{thm:fixed-terminal-data}
Let \(X\) be connected and compact Kähler, let \(D\) be a finite SNC
divisor, and let \((E_*,\theta)\) be a regular parabolic
Higgs sheaf. There exist a closed analytic subset \(Z\subset D\) of
codimension at least three, a connected smooth compact Kähler manifold
\(Y\), and a projective proper modification
\[
 \pi:Y\longrightarrow X
\]
which is biholomorphic over \(X\setminus Z\), with the following
properties. The reduced union \(B\) of the strict transform of \(D\)
and \(\pi^{-1}(Z)\) is an SNC divisor. On \((Y,B)\) there is a
logarithmic Higgs bundle \((V,\theta_Y)\) with a finite
\(\theta_Y\)-invariant flag along each component of \(B\). These flags
split simultaneously locally, and their nonzero graded quotients are
locally free with semisimple induced residues.

On \(X\setminus Z\), the sheaf \(E\) is locally free and all original
flags split simultaneously. Under the identification with
\(\pi^{-1}(X\setminus Z)\), the Higgs bundle \((V,\theta_Y)\) agrees
with \((E,\theta)\), and its flags agree with the refined flags of
Lemma~\ref{lem:residue-refinement}.
\end{thm}

\begin{proof}
\emph{Passage to a root stack.}
Apply Lemma~\ref{lem:residue-refinement} to refine the original graded
blocks. Let \(\Sigma\subset D\) be the non-locally-free locus of \(E\),
and let \(Z_{\mathrm{flag}}\) be the exceptional set in
Lemma~\ref{lem:simultaneous-splitting} for the refined flags. Set
\[
 Z=\Sigma\cup Z_{\mathrm{flag}}.
\]
Every component has codimension at least three. On \(X\setminus Z\), the
ambient sheaf is locally free and the refined, hence also the original,
flags split simultaneously.

Choose strictly increasing rational weights on each refined flag and
clear their denominators. These weights serve only to apply the
root-stack correspondence. The associated divisor lattices and their
reflexive intersections form a coherent rational parabolic system.
By Proposition~\ref{prop:analytic-root-correspondence}, this
system corresponds to a torsion-free logarithmic Higgs sheaf
\[
 (\mathcal F,\Theta)
 \quad\text{on}\quad
 \mathcal R=\sqrt[\mathbf r]{(X,D)}.
\]
If \(p:\mathcal R\to X\) is the root projection, the simultaneous split
formula makes \(\mathcal F\) locally free on
\[
 \mathcal W=p^{-1}(X\setminus Z).
\]

\smallskip
\noindent\emph{Resolution and destackification.}
Apply Proposition~\ref{prop:protected-total-boundary} to
\[
 (\mathcal R,D_{\mathcal R},\mathcal F,\Theta,\mathcal W),
\]
where \(D_{\mathcal R}\) is the root boundary. It gives a smooth compact
stack \((\mathcal S,\mathcal C)\), a representable projective proper
modification
\[
 h:(\mathcal S,\mathcal C)\longrightarrow
   (\mathcal R,D_{\mathcal R})
\]
which is the identity over \(\mathcal W\), and a logarithmic Higgs bundle
\[
 \mathcal V_{\mathcal S}
 =h^*\mathcal F/\Tor(h^*\mathcal F).
\]
The boundary \(\mathcal C\) contains both the complete inverse image of
the root boundary and the inverse image of
\(\mathcal R\setminus\mathcal W\).

Over \(\mathcal W\), the stack \(\mathcal S\) is the original root
stack, so its boundary characters satisfy \eqref{eq:root-clean}.
Theorem~\ref{thm:analytic-destackification} gives a modification
\(\delta:\mathcal Y\to\mathcal S\), unchanged over \(\mathcal W\),
with
\[
 \mathcal Y\simeq\sqrt[\boldsymbol d]{(Y,B)}
\]
over a smooth compact SNC pair \((Y,B)\). The induced coarse morphism
\(\pi:Y\to X\) is projective, proper, and biholomorphic over
\(X\setminus Z\). Since \(X\) is compact Kähler and \(Y\) is smooth,
projectivity implies that \(Y\) is Kähler
\cite[Lemma~2.16]{Xia19}. The dense open subset \(X\setminus Z\) is
connected, so \(Y\) is connected as well.

Put \(f=h\circ\delta:\mathcal Y\to\mathcal R\). The pullback
\[
 \mathcal V=\delta^*\mathcal V_{\mathcal S}
\]
is a vector bundle and agrees with \(f^*\mathcal F\) over the dense open
\(\mathcal W\). Lemma~\ref{lem:total-torsion-free-pullback} gives
\[
 f^*\mathcal F/\Tor(f^*\mathcal F)
 \simeq\mathcal V,
\]
and identifies the induced integrable logarithmic Higgs fields.

\smallskip
\noindent\emph{Residues and the resulting flags.}
It remains to check semisimplicity, in particular along the new boundary
components. On an initial root divisor, the character components of
\(\mathcal F\) restricted to that divisor correspond, after the
tautological line-bundle twists, to the refined parabolic quotients.
Let \(x_i\) be its root coordinate and \(A_i\) the coefficient of
\(dx_i/x_i\) in \(\Theta\). If \(\Lambda_i\) is the finite set of
residue eigenvalues, set
\[
 P_i(T)=\prod_{\lambda\in\Lambda_i}(T-\lambda).
\]
The refined residues are semisimple, so this square-free polynomial
annihilates every character component modulo \(x_i\). Hence
\[
 P_i(A_i)\mathcal F\subset x_i\mathcal F.
 \tag{3.6}\label{eq:square-free-residue}
\]
This relation persists after pullback and passage to the torsion-free
quotient \(\mathcal V\).

The initial root boundary
maps into \(D\), while every additional boundary component lies over
\(\mathcal R\setminus\mathcal W=p^{-1}(Z)\) and \(Z\subset D\). Hence every
terminal boundary component maps into \(D\). Let \(C\) be one such
component. Near its generic point, each initial root-boundary equation has
the form
\[
 f^*x_i=u_i\prod_a y_a^{m_{ia}},
\]
where \(u_i\) is a unit and \(m_{ia}\ge0\). For the equation
\(y_a=0\) of \(C\), at least one multiplicity \(m_{ia}\) is positive.
Pulling back logarithmic differentials gives
\[
 f^*\frac{dx_i}{x_i}
 =\frac{du_i}{u_i}+\sum_a m_{ia}\frac{dy_a}{y_a}.
\]
The holomorphic part contributes no residue, so
\[
 R_C=\sum_i m_{ia}A_i.
\]
Here the \(A_i\) denote the induced coefficients on \(\mathcal V\),
restricted to \(C\). If \(m_{ia}>0\), then \(f^*x_i\) vanishes on
\(C\), so the pulled-back relation \eqref{eq:square-free-residue} gives
\(P_i(A_i)|_C=0\). Thus every coefficient contributing to \(R_C\) is
semisimple. Integrability makes these coefficients commute, so their
linear combination is semisimple as well. The same polynomial relations
hold along the whole component and therefore give semisimplicity on
each of its stabilizer-character quotients.

Apply the inverse root-stack correspondence to \(\mathcal V\). It gives
a logarithmic Higgs bundle \((V,\theta_Y)\) with simultaneously split
flags, locally free graded quotients, and the residue properties just
proved. Every modification was an isomorphism over \(\mathcal W\).
Restricting the root-stack correspondence to this open subset therefore
recovers \((E,\theta)\) and its refined flags on \(X\setminus Z\).
Along a strict transform, the resulting flag terms are also the
saturated extensions of these identified generic fibers, since their
quotients are locally free.

Finally, the boundary contains the inverse image of \(Z\), and all new
boundary components lie over \(Z\). Thus its coarse support is exactly
the reduced union of the strict transform of \(D\) and \(\pi^{-1}(Z)\).
Discarding the temporary rational weights leaves the required bundle
and fixed flags.
\end{proof}

\subsection{Weights on the fixed flags and their coalescence}
\label{subsec:fixed-flag-weights}

We now assign weights to the flag system of
Theorem~\ref{thm:fixed-terminal-data} to recover the original parabolic
structure and construct its rational perturbations.

Write \(B=\bigcup_{a\in A}B_a\). In the lattice notation of
Section~\ref{sec:preliminaries}, the flag along \(B_a\) is a chain
\[
 V=L_{a,0}\supsetneq L_{a,1}\supsetneq\cdots
 \supsetneq L_{a,m_a}=V(-B_a).
\]
Assign a real number \(w_{a,q}\) to the quotient
\(L_{a,q-1}/L_{a,q}\), with
\[
 0\le w_{a,1}\le\cdots\le w_{a,m_a}<1.
\]
These numbers are the parabolic weights. Explicitly, they define the
decreasing, left-continuous filtration
\[
 F_{a,w}^{t}V=L_{a,k_a(t)},\qquad
 k_a(t)=\#\{q:w_{a,q}<t\},\qquad 0\le t\le1,
\]
extended to all real indices by
\(F_{a,w}^{t+1}V=F_{a,w}^{t}V(-B_a)\).
For \(\mathbf{b}\in[0,1)^A\), the associated parabolic lattice is
\[
 V_w^{\mathbf{b}}=\bigcap_{a\in A}F_{a,w}^{b_a}V,
\]
with integral periodic extension as in
\eqref{eq:periodic-lattice-notation}.

Let \(w^\infty\) be the weight vector on these fixed flags determined
by the original downstairs parabolic structure, as follows.
If \(B_a\) is the strict transform of \(D_i\)
and a refined quotient belongs to the original block
\(\Gr_i^\alpha(E_*)\), set \(w^\infty_{a,q}=\alpha\).
On every exceptional component over \(Z\), set \(w^\infty_{a,q}=0\) for
all \(q\). These weights are nondecreasing in the flag order: all the quotients inserted
inside one original block have the same weight. We call this merging of the refined filtration jumps \emph{coalescence}.
It recovers the original filtration on \(X\setminus Z\).

Later we choose strictly ordered rational weights
\(w^{(\nu)}\to w^\infty\) on these same flags. Strict weights distinguish
every refined quotient, so the resulting parabolic Higgs bundles are
graded semisimple by the residue property of the fixed model. 
Subsection~\ref{subsec:fixed-model-completion} completes the model by
arranging the lattice inclusion along the exceptional divisor.
Section~\ref{sec:fixed-model} then proves stability for sufficiently small
weight and polarization perturbations.

\subsection{Completion of the fixed logarithmic flag model}
\label{subsec:fixed-model-completion}

We make one final adjustment to obtain the inclusion
\(\pi^*E/\Tor\hookrightarrow V\) used in the proof of
Proposition~\ref{prop:downstairs-reflexive-degree-bridge}.
The principalization in
Proposition~\ref{prop:protected-total-boundary} makes the inverse
image of \(Z\) divisorial; denote its reduced support by
\(A_{\mathrm{exc}}\).

\begin{prop}\label{prop:fixed-model-correction}
Retain the fixed logarithmic flag models
\((V_j,\theta_{j,Y})\) of the original sheaf, or of its finitely many
stable summands \((E_{j,*},\theta_j)\), on \((Y,B)\).
There is a common integer \(N\ge0\) such that, with
\(L=\cO_Y(NA_{\mathrm{exc}})\), the canonical identifications away
from the exceptional divisor extend to injections
\[
 \pi^*E_j/\Tor\longrightarrow V_j\otimes L.
\]
Tensoring every flag term by \(L\), endowed with zero parabolic weights
and zero Higgs field, preserves the simultaneous local splittings and
the induced residue endomorphisms on the graded quotients.
The resulting models are unchanged over \(X\setminus Z\).
\end{prop}

\begin{proof}
Put \(\mathcal F_j=\pi_*V_j\). Properness makes this sheaf
coherent, and its identification with \(E_j\) outside \(Z\), together
with reflexivity of \(E_j\), gives
\(\mathcal F_j\subset\mathcal F_j^{**}=E_j\).
The quotients are supported on \(Z\), so one power
\(\mathcal I_Z^m\) satisfies
\(\mathcal I_Z^mE_j\subset\mathcal F_j\) for every \(j\).
The zero set of \(\mathcal I_Z\cO_Y\) is \(A_{\mathrm{exc}}\);
coherence and compactness give an integer \(N\) such that
\[
 \cO_Y(-NA_{\mathrm{exc}})\subset\mathcal I_Z^m\cO_Y.
\]
Pulling back and composing with the evaluation map
\(\pi^*\pi_*V_j\to V_j\), then quotienting by torsion, yields
\[
 (\pi^*E_j/\Tor)\otimes\cO_Y(-NA_{\mathrm{exc}})
 \longrightarrow V_j.
\]
It is injective because it is the canonical identification off the
exceptional divisor and its source is torsion-free.
Tensoring by \(L\) gives the required injection.

Tensoring all flag terms by the same line bundle preserves their
inclusions, local splittings, and graded residue endomorphisms.
Since \(L\) is canonically trivial away from the exceptional divisor,
the data there are unchanged.
\end{proof}

From now on, \(V_j\) and its flags denote these twisted objects.
This is the fixed logarithmic flag model used below.
For a single summand we omit \(j\); \(V_{\nu,*}\) and
\(V_{\infty,*}\) denote the parabolic structures at weights
\(w^{(\nu)}\) and \(w^\infty\), respectively.

\section{Stable rational approximations and Hermitian--Einstein metrics}\label{sec:fixed-model}

We use the fixed logarithmic flag model completed in
Section~\ref{sec:resolution-fixed-model}. We first establish stability
under small weight and Kähler perturbations and choose stable rational
stages. We then compare parabolic characteristic classes, construct
the metrics \(K_\nu\), \(H_{0,\nu}\), and the Hermitian--Einstein
solutions \(h_\nu\). For a rank-\(r\) summand, the Chern--Weil
identity we prove asserts that the parabolic Chern number
\[
 \left\langle
 \parc_2(V_{\nu,*})-\frac{r-1}{2r}\parc_1(V_{\nu,*})^2,
 \frac{[\omega_\nu]^{n-2}}{(n-2)!}
 \right\rangle
\]
equals the normalized quadratic integral of the Hitchin--Simpson
curvature of \(K_\nu\) in
\eqref{eq:literal-canonical-quadratic}. We then use this equality
to derive curvature energy bounds for \(h_\nu\).

\begin{prop}
\label{prop:uniform-fixed-model-stability}
Retain a fixed logarithmic flag model
\[
 \pi:Y\longrightarrow X,\qquad
 \omega_0=\pi^*\omega,\qquad
 \omega_\delta=\omega_0+\delta\vartheta,
\]
where \(\vartheta\) is a fixed Kähler form on \(Y\) and
\(0\le\delta\le1\). For \(\delta>0\), the form \(\omega_\delta\) is
Kähler and defines the polarization \([\omega_\delta]\).
Let \((V,\theta_Y)\) be one summand of the fixed logarithmic flag
model, of rank \(r\), whose parabolic structure \(V_{\infty,*}\) at
\(w^\infty\) agrees with an
\(\omega\)-stable downstairs summand \(E_*\) over \(X\setminus Z\).
For an admissible weight vector \(w=(w_{a,q})\) as in
Subsection~\ref{subsec:fixed-flag-weights}, denote the corresponding
upstairs parabolic Higgs bundle by \(V_{*,w}\); in particular,
\(V_{*,w^\infty}=V_{\infty,*}\).
For every saturated Higgs subsheaf \(0\ne F\subsetneq V\), transport
\(F\) over the biholomorphic locus and let \(S\subset E\) be its
saturated reflexive extension. Saturated intersections with the ambient
flags, equipped with the inherited weights, define the induced parabolic
Higgs subsheaves \(F_{*,w}\subset V_{*,w}\) and \(S_*\subset E_*\).
We use the parabolic slope of Definition~\ref{defn:parabolic-stability}.

Then \(S_*\) is nonzero and proper, and the coalesced parabolic subsheaf
\(F_{*,w^\infty}\) agrees with \(S_*\) over \(X\setminus Z\) under the
biholomorphic identification.
If \(r\ge2\), there are constants
\[
 \kappa>0,\qquad C\ge0,\qquad D\ge0
\]
depending only on the fixed logarithmic flag model such that, for every admissible weight
vector \(w\) on the fixed upstairs flags,
\[
 \mu_{\omega_\delta}(F_{*,w})-\mu_{\omega_\delta}(V_{*,w})
 \le-\kappa+\delta C+2e(w)D,
 \tag{4.1}\label{eq:uniform-stability-window}
\]
where
\[
 e(w)=\max_{a,q}|w_{a,q}-w^\infty_{a,q}|.
\]
Consequently \(V_{*,w}\) is \(\omega_\delta\)-stable whenever
\(\delta>0\) and \(e(w)\) are sufficiently small.
\end{prop}

\begin{proof}
A nonzero proper saturated subsheaf of a vector bundle has rank strictly
smaller than that of the ambient bundle.
Transport over the common open therefore gives a subsheaf of the same
rank, and saturated reflexive extension inside \(E\) is nonzero and
proper. Higgs invariance extends because the quotient is torsion-free and
\(\Omega_X^1(\log D)\) is locally free. Saturated intersection with a
one-divisor flag has torsion-free successive quotients after saturation.
The construction of the fixed logarithmic flag model identifies these
flags away from the exceptional image.

We first prove a uniform downstairs gap. Among saturated invariant
subsheaves \(S\subset E\), those satisfying
\(\mu_\omega(S_*)\ge\mu_\omega(E_*)-1\) form a bounded family.
After the finitely many possible flag ranks are separated, the inequality
gives an upper bound for the degrees of the pure quotients \(E/S\).
Lemma~4.3(3) of \cite{Tom11} gives boundedness of this family of quotients
of \(E\). Proposition~3.3(1) of the same reference gives boundedness
of the kernels \(S\), and its Corollary~5.2 gives finiteness of their Chern
classes.
Only finitely many first Chern classes and finitely many arrays of generic
flag ranks occur in this family. Hence only finitely many parabolic slopes
occur there.
Stability makes each of them strictly smaller than the ambient slope.
Subsheaves outside this family already have gap at least one. Thus there
is one \(\kappa>0\) such that
\[
 \mu_\omega(S_*)-\mu_\omega(E_*)\le-\kappa
\]
for every proper saturated invariant \(S\).

Let \(s=\rk F\). If \(M_{a,q}\) and \(m_{a,q}(F)\) are the generic ranks
of the ambient and induced flag quotients, put
\[
 A_F^\infty=
 \frac{c_1(F)}s-\frac{c_1(V)}r
 +\sum_{a,q}w^\infty_{a,q}
 \left(\frac{m_{a,q}(F)}s-\frac{M_{a,q}}r\right)[B_a].
\]
The degree-one comparison in the proof of
Proposition~\ref{prop:parabolic-class-bridge} applies to the ambient pair
and to \((F_{*,w^\infty},S_*)\), since the induced filtrations also agree on
\(X\setminus Z\). Together with the projection formula, it gives
\[
 \left\langle A_F^\infty,
 \frac{\omega_0^{n-1}}{(n-1)!}\right\rangle
 =
 \mu_\omega(S_*)-\mu_\omega(E_*)
 \le-\kappa.
\]

For \(1\le k\le n-1\), set
\[
 \eta_k=\omega_0^{n-1-k}\wedge\vartheta^k.
\]
There is a constant \(C_k^{\mathrm{ord}}\), independent of \(F\), such that
\[
 \left\langle
 \frac{c_1(F)}s-\frac{c_1(V)}r,\eta_k
 \right\rangle\le C_k^{\mathrm{ord}}.
 \tag{4.2}\label{eq:mixed-c1-bound}
\]
Indeed, on the subbundle locus let \(p_F\) be the orthogonal projection
for a fixed smooth metric on \(V\). The determinant Gauss--Codazzi formula
writes the first term as the contraction of the fixed ambient curvature
with \(p_F\), minus the squared norm of the second fundamental form.
Wedging with the strongly semipositive form \(\eta_k\) keeps the second
term nonpositive, while the first is bounded by a fixed multiple of
\(\rk F\). Capacity cutoffs around the codimension-two non-subbundle locus
remove the boundary term and give \eqref{eq:mixed-c1-bound}.

For each \(a\), both arrays
\[
 \left(\frac{m_{a,q}(F)}s\right)_q,
 \qquad
 \left(\frac{M_{a,q}}r\right)_q
\]
are probability vectors. Hence
\[
 \sum_q\left|
 \frac{m_{a,q}(F)}s-\frac{M_{a,q}}r
 \right|\le2.
\]
Since \(0\le w^\infty_{a,q}<1\) and
\([B_a]\eta_k\ge0\), the preceding inequality and
\eqref{eq:mixed-c1-bound} give a constant \(C_k^{\mathrm{par}}\)
such that
\[
 \langle A_F^\infty,\eta_k\rangle\le C_k^{\mathrm{par}}
\]
for every \(F\).

Expanding \(\omega_\delta^{n-1}\) and using
\(\delta^k\le\delta\) for \(0\le\delta\le1\), we obtain
\[
 \mu_{\omega_\delta}(F_{*,w^\infty})-\mu_{\omega_\delta}(V_{*,w^\infty})
 \le-\kappa+\delta C,
\]
where
\[
 C=\sum_{k=1}^{n-1}
 \binom{n-1}{k}
 \frac{\max\{C_k^{\mathrm{par}},0\}}{(n-1)!}.
\]
Finally put
\[
 d_a(\delta)=
 \left\langle[B_a],
 \frac{\omega_\delta^{n-1}}{(n-1)!}\right\rangle,
 \qquad
 D=\max_{0\le\delta\le1}\sum_a d_a(\delta).
\]
These quantities are nonnegative and \(D<\infty\). Changing the weights
alters the slope difference by at most
\[
 e(w)\sum_{a,q}
 \left|\frac{m_{a,q}(F)}s-\frac{M_{a,q}}r\right|d_a(\delta)
 \le2e(w)D.
\]
Together with the preceding polarization estimate, this proves
\eqref{eq:uniform-stability-window}. For example,
\[
 \delta<\frac{\kappa}{4(C+1)},\qquad
 e(w)<\frac{\kappa}{8(D+1)}
\]
make every proper slope difference negative.
\end{proof}

\begin{lem}
\label{lem:actual-kahler-smoothing}
Normalize a Kähler form \(\vartheta\) on \(Y\) by
\(\int_Y\vartheta^n=1\). For each component \(B_a\), choose a smooth
divisor metric and let \(s_a\le1\) be the squared norm of its canonical
section. Choose \(c_a\ge0\) so that
\[
 \ii\partial\bar\partial\log s_a\ge-c_a\vartheta
 \quad\text{on }Y\setminus B,
 \qquad C_B=\sum_ac_a.
\]
For \(0<\epsilon,\eta\le1\) and \(0<A\le1\), put
\[
 \Phi_{\epsilon,\eta}=\sum_a(s_a+\eta^2)^\epsilon.
\]
If
\[
 \delta>C_BA\epsilon,
\]
then
\[
 \omega_{\epsilon,A,\delta,\eta}
 =\pi^*\omega+\delta\vartheta
 +A\ii\partial\bar\partial\Phi_{\epsilon,\eta}
\]
is a smooth Kähler form and
\[
 [\omega_{\epsilon,A,\delta,\eta}]
 =\pi^*[\omega]+\delta[\vartheta].
\]
\end{lem}

\begin{proof}
For \(x=s_a>0\), put \(u=x+\eta^2\). We decompose the contribution
\(A\ii\partial\bar\partial u^\epsilon\) into a term involving
\(\ii\partial\bar\partial\log x\) and a semipositive term.
Using
\[
 \ii\partial\bar\partial x
 =x\,\ii\partial\bar\partial\log x
 +\frac{\ii\partial x\wedge\bar\partial x}{x},
\]
and writing
\[
 \mathfrak b(x):=
 A\epsilon(x+\eta^2)^{\epsilon-2}(\eta^2+\epsilon x)>0,
 \tag{4.3}\label{eq:amplitude-radial-coefficient}
\]
direct differentiation gives
\[
 \begin{aligned}
 A\ii\partial\bar\partial u^\epsilon
 &=A\epsilon u^{\epsilon-1}x\,
   \ii\partial\bar\partial\log x\\
 &\quad+\mathfrak b(x)
 \frac{\ii\partial x\wedge\bar\partial x}{x}.
 \end{aligned}
 \tag{4.4}\label{eq:amplitude-differentiation}
\]
The second term is semipositive because \(\mathfrak b(x)>0\). Since
\(0\le xu^{\epsilon-1}\le1\), the first line is bounded below by
\(-A\epsilon c_a\vartheta\). Summing over \(a\) yields
\[
 A\ii\partial\bar\partial\Phi_{\epsilon,\eta}
 \ge-C_BA\epsilon\vartheta.
\]
Thus
\[
 \omega_{\epsilon,A,\delta,\eta}
 \ge\pi^*\omega+(\delta-C_BA\epsilon)\vartheta>0.
\]
Smoothness follows from \(\eta>0\). The term \(A\ii\partial\bar\partial\Phi_{\epsilon,\eta}\) is globally
exact, which proves the asserted cohomology-class identity.
\end{proof}

\begin{cor}\label{cor:stable-rational-approximations}
Retain the fixed logarithmic flag models of
Proposition~\ref{prop:fixed-model-correction}. For each summand \(j\),
divisor component \(a\), and quotient label \(q\), let
\[
 w^\infty_{j,a,q}\in[0,1)
\]
be the normalized weight determined by the original parabolic
structure in Subsection~\ref{subsec:fixed-flag-weights}. These prescribed
weights are nondecreasing in
the flag order, so equal weights are allowed:
\[
 w^\infty_{j,a,q}\le w^\infty_{j,a,q+1}
 \qquad\text{for every pair of successive quotients.}
\]

For every positive integer \(\nu\), there are rational normalized weights
\[
 w^{(\nu)}_{j,a,q}\in\QQ\cap[0,1)
\]
and a positive integer \(M_\nu\) such that:

\begin{enumerate}
\item the weights along each fixed flag are strictly ordered in the flag
order, and hence distinguish all its nonzero successive quotients;
\item \(M_\nu w^{(\nu)}_{j,a,q}\) is integral for every \(j,a,q\);
\item
\[
 e_\nu=\max_{j,a,q}|w^{(\nu)}_{j,a,q}-w^\infty_{j,a,q}|
 \longrightarrow0.
\]
\end{enumerate}

Write \(V_{j,\nu,*}\) for the parabolic Higgs bundle defined by
\(w^{(\nu)}\) on the \(j\)-th fixed model. These bundles are locally
abelian and graded semisimple.

Choose the divisor norms and \(C_B\) from
Lemma~\ref{lem:actual-kahler-smoothing}. Fix a real number \(p>n\), an
integer
\[
 N>\max\{1,p-1\},
\]
and a constant \(C_\delta>C_B\). After increasing a fixed integer
\(\nu_0\), put
\[
 n_\nu=\nu+\nu_0,\qquad
 \epsilon_\nu=n_\nu^{-1},\qquad
 A_\nu=\epsilon_\nu^N,
\]
\[
 \mathfrak m_\nu=A_\nu\epsilon_\nu=n_\nu^{-(N+1)},\qquad
 \eta_\nu^2=\mathfrak m_\nu,\qquad
 \delta_\nu=C_\delta\mathfrak m_\nu,
\]
and define
\[
 \omega_\nu
 =\pi^*\omega_X+\delta_\nu\vartheta
 +A_\nu\ii\partial\bar\partial
   \sum_a(s_a+\eta_\nu^2)^{\epsilon_\nu}.
 \tag{4.5}\label{eq:kahler-schedule-form}
\]
Then \(\omega_\nu\) is a smooth Kähler form,
\[
 [\omega_\nu]=\pi^*[\omega_X]+\delta_\nu[\vartheta],
\]
and \(\omega_\nu\to\pi^*\omega_X\) smoothly on every compact subset of
\(Y\setminus B\). For all sufficiently large \(\nu\), every
\(V_{j,\nu,*}\) is \(\omega_\nu\)-stable.
\end{cor}

\begin{proof}
Fix one flag with \(N\) nonzero successive quotients, and write its
prescribed weights from \(w^\infty\) as
\[
 x^\infty=(x_1^\infty,\ldots,x_N^\infty),\qquad
 0\le x_1^\infty\le\cdots\le x_N^\infty<1.
\]
For \(0<t<1\), define
\[
 x_q(t)=(1-t)x_q^\infty+t\frac{q}{N+1}.
\]
For \(q<N\),
\[
 x_{q+1}(t)-x_q(t)
 =(1-t)(x_{q+1}^\infty-x_q^\infty)+\frac{t}{N+1}>0.
\]
Also \(x_1(t)>0\) and \(x_N(t)<1\). Hence \(x(t)\) lies in the open
strict chamber and converges to \(x^\infty\) as \(t\downarrow0\).

Rational vectors are dense in this chamber. For each \(\nu\), choose a rational vector within \(1/\nu\) of \(x^\infty\). Perform the choice for every fixed flag. There are finitely many weights at each stage, so the least common multiple of their reduced denominators gives \(M_\nu\). This proves the three weight assertions. Local abelianness and
semisimplicity of the quotient residues are unaffected because only the
numerical labels have changed.

The parameters satisfy
\[
 \delta_\nu-C_BA_\nu\epsilon_\nu
 =(C_\delta-C_B)\mathfrak m_\nu>0.
\]
Lemma~\ref{lem:actual-kahler-smoothing} therefore proves positivity and
the cohomology identity. If \(\mathcal C\Subset Y\setminus B\), every \(s_a\) has
a positive lower bound on \(\mathcal C\). Every derivative of positive order
of
\[
 A_\nu(s_a+\mathfrak m_\nu)^{\epsilon_\nu}
\]
contains the factor \(A_\nu\epsilon_\nu=\mathfrak m_\nu\), while the remaining
factors and their derivatives are uniformly bounded on \(\mathcal C\). Hence the
perturbation term converges to zero in every \(C^q(\mathcal C)\), proving the
claimed smooth convergence on compact sets. With these parameters,
the coefficient \(\mathfrak b(x)\) in
\eqref{eq:amplitude-radial-coefficient} is
\[
 \mathfrak b_\nu(x)
 =\mathfrak m_\nu(x+\mathfrak m_\nu)^{\epsilon_\nu-2}
  (\mathfrak m_\nu+\epsilon_\nu x).
\]
For a summand of rank at least two,
Proposition~\ref{prop:uniform-fixed-model-stability} gives constants
\(\kappa_j>0\) and \(C_j,D_j\ge0\), independent of the nonzero proper
saturated Higgs subsheaf \(F\subset V_j\). If \(F_{\nu,*}\) is its
induced parabolic structure at stage \(\nu\), then
\[
 \mu_{\omega_\nu}(F_{\nu,*})-\mu_{\omega_\nu}(V_{j,\nu,*})
 \le-\kappa_j+\delta_\nu C_j+2e_\nu D_j.
\]
There are finitely many summands. Taking the minimum of the positive
\(\kappa_j\)'s and the maxima of \(C_j,D_j\) gives a common inequality.
Its right side is negative for all sufficiently large \(\nu\), because
\(\delta_\nu,e_\nu\to0\).
\end{proof}

We next compare the parabolic characteristic classes of these stages
with those of the original sheaf.

\subsection{Comparison of parabolic Chern classes}

\begin{prop}
\label{prop:parabolic-class-bridge}
Let \((E_*,\theta)\) be the original rank-\(r\) parabolic Higgs sheaf,
and retain its fixed logarithmic flag model
\(\pi:(Y,B)\to(X,D)\) from Section~\ref{sec:resolution-fixed-model}.
With the weight assignment of Subsection~\ref{subsec:fixed-flag-weights},
let \(V_{\infty,*}\) be the upstairs parabolic structure at the
prescribed weights \(w^\infty\), and write
\(V_{\nu,*}\) for the structures at strictly ordered rational weights
\(w^{(\nu)}\to w^\infty\). Then
\[
 \begin{aligned}
 \pi_*\parc_k(V_{\infty,*})&=\parc_k(E_*)\qquad(k=1,2),\\
 \pi_*\bigl(\parc_1(V_{\infty,*})^2\bigr)&=\parc_1(E_*)^2.
 \end{aligned}
 \tag{4.6}\label{eq:low-degree-bridge}
\]
Consequently,
\[
 \begin{aligned}
 &\pi_*\left(\parc_2(V_{\infty,*})-\frac{r-1}{2r}\parc_1(V_{\infty,*})^2\right)\\
 &\qquad=\parc_2(E_*)-\frac{r-1}{2r}\parc_1(E_*)^2.
 \end{aligned}
\]

For the rational stages,
\[
 \begin{aligned}
 \pi_*\parc_1(V_{\nu,*})&\longrightarrow \parc_1(E_*),\\
 \pi_*\parc_2(V_{\nu,*})&\longrightarrow \parc_2(E_*),\\
 \pi_*\bigl(\parc_1(V_{\nu,*})^2\bigr)
 &\longrightarrow \parc_1(E_*)^2.
 \end{aligned}
\]
\end{prop}

\begin{proof}
Apply the weighted formula \eqref{eq:weighted-parabolic-chern} to the
reflexive parabolic lattices. For fixed finite chains, its components are
polynomials in the weights. Their coefficients depend on the ordinary
Chern characters of the lattices and the divisor classes; see also the
explicit SNC formulas in \cite[Section~4.1]{Taher09}. The definition uses
real weights directly, so choosing a common denominator for rational
weights does not change it.

By construction, \(\pi\) is biholomorphic over \(X\setminus Z\), where
\(\codim_X Z\ge3\). On this common open, the ambient sheaves and
their filtrations along each divisor agree. An exceptional weight
coordinate contributes neither a divisor class nor a different lattice
after restriction; its integration over the unit interval contributes the
same factor one to numerator and denominator. Hence the degree-two and
degree-four parabolic Chern-character components agree on the common open.

At a double crossing outside \(Z\), choose a common split frame and let
\(r_{i,k}^{a,b}\) be the number of its basis vectors with jumps \((a,b)\).
Then
\[
 \sum_b r_{i,k}^{a,b}=\rk\Gr_i^aE,\qquad
 \sum_a r_{i,k}^{a,b}=\rk\Gr_k^bE.
\]
Thus coalescing inserted steps recombines every crossing contribution into
the original coefficient; no extra term is introduced by the order in
which the two associated gradeds are formed.

The difference between the corresponding degree-two classes is
supported over \(Z\), and the differences between the degree-four classes
and between the first-class squares are supported there as well. Proper
pushforward preserves these support conditions. Cohomological purity in a
smooth manifold says that a class supported on a complex analytic set of
codimension at least three has no component in real degree two or four.
The pushed differences therefore vanish, which gives
\eqref{eq:low-degree-bridge}. Taking the stated linear combination proves the discriminant identity.

For the family \(V_{\nu,*}\), the weighted formula is polynomial on
each ordering chamber and extends continuously when adjacent weights
coalesce. The preceding rank identities show that the limiting
crossing coefficients are the coefficients of the limiting filtration.
Therefore, in \(H^{2k}(Y,\RR)\),
\[
 \parc_k(V_{\nu,*})\longrightarrow\parc_k(V_{\infty,*})
 \qquad(k=1,2).
\]
Cup product is continuous in the finite-dimensional cohomology ring, so
\[
 \parc_1(V_{\nu,*})^2\to \parc_1(V_{\infty,*})^2.
\]
Proper pushforward is linear and continuous. Applying it and then using
\eqref{eq:low-degree-bridge} proves the asserted convergence in the
stated cohomology spaces.
\end{proof}

\subsection{Construction of an adapted Hermitian metric}
\label{subsec:adapted-references}

On the fixed model \((Y,B)\), fix a rational
stage \(\nu\) from Corollary~\ref{cor:stable-rational-approximations} and one
summand \((V_{\nu,*},\theta_Y)\), of rank \(r\). Put \(U=Y\setminus B\).
We construct a Hermitian metric \(K_\nu\) whose growth realizes the
parabolic lattices. Its explicit local form will also be used in the
proof of absolute convergence of the curvature integral in
\eqref{eq:literal-canonical-quadratic}.

Choose a finite SNC atlas \(\{U_\lambda\}\) on which the fixed flags
split simultaneously.
On a chart \(U_\lambda\), write \(B\) as \(\prod_a z_a=0\) and choose
a compatible frame \(e_1,\ldots,e_r\). Let \(w_{a,i}^{(\nu)}\) be the
weight of the flag quotient to which \(e_i\) belongs along \(B_a\).
Thus \(w_{a,i}^{(\nu)}=w_{a,q}^{(\nu)}\) when the frame vector
\(e_i\) belongs to the \(q\)-th quotient.

We also require compatibility with the residue eigenspaces. This is
distinct from compatibility with the parabolic weights: one graded
bundle \(Q_{a,\alpha}=\Gr_a^\alpha(V_{\nu,*})\), indexed by a weight
\(\alpha\), may have several residue eigenvalues. The graded bundles
on the fixed model are locally free. If \(R_{a,\alpha}\) is the induced
residue, graded semisimplicity gives a holomorphic decomposition
\[
 Q_{a,\alpha}=\bigoplus_{\xi\in\Sigma_{a,\alpha}}Q_{a,\alpha,\xi},
 \qquad
 R_{a,\alpha}|_{Q_{a,\alpha,\xi}}=\xi\Id.
\]
At an SNC intersection, integrability of the Higgs field makes the
induced residues commute. Their spectral projectors therefore commute
and preserve the induced flags, so the eigenspace decompositions admit
a simultaneous refinement compatible with the flag splitting.
Choose the local frames to respect these decompositions on the graded
bundles. We refer to these as \emph{joint-spectral frames}.

Declare the frame vectors orthogonal and set
\[
 K_{\nu,\lambda}(e_i,e_j)
 =\delta_{ij}\prod_a|z_a|^{2w_{a,i}^{(\nu)}}.
 \tag{4.7}\label{eq:local-stage-reference}
\]
This is the local \emph{split monomial metric}: it is diagonal in the
chosen frame, and its powers of the boundary coordinates record the
parabolic weights. Extend the atlas across the interior of \(U\), using
smooth positive metrics there. Choose a smooth partition of unity
\(\{\rho_\lambda\}\) subordinate to this atlas and put
\[
 \widetilde K_\nu=\sum_\lambda\rho_\lambda K_{\nu,\lambda}.
 \tag{4.8}\label{eq:partitioned-stage-reference}
\]
The sum is taken as Hermitian forms on the same bundle. Interior terms
have compact support away from \(B\). This defines a smooth positive
Hermitian metric \(\widetilde K_\nu\) on \(V|_U\).

The transition maps preserve the parabolic lattices. On each smaller
chart, with \(\nu\) fixed, the local metrics are therefore uniformly
comparable with the same split model. Their positive weighted sum has
the same property.
Comparing a section with that model gives both directions of the growth
condition in Definition~\ref{defn:growth}, at every real multi-index and
every SNC intersection. Thus \(\widetilde K_\nu\) is adapted to
\(V_{\nu,*}\). If the frames, interior metrics, and partition are
fixed as the weights vary, \(w^{(\nu)}\to w^\infty\) also gives smooth
convergence of \(\widetilde K_\nu\) on compact subsets of \(U\).

The common denominator \(M_\nu\) also has a metric interpretation.
Let \(p_\lambda\) be the local root cover \(z_a=t_a^{M_\nu}\), and put
\(m_{a,i}=M_\nu w_{a,i}^{(\nu)}\in\mathbb Z\). On the divisor complement,
set
\[
 \widehat e_i=\left(\prod_a t_a^{-m_{a,i}}\right)p_\lambda^*e_i.
\]
Under Proposition~\ref{prop:analytic-root-correspondence}, these are
holomorphic local generators of the root-stack bundle corresponding to
\(V_{\nu,*}\); the displayed factors may have poles relative to the
ordinary pullback bundle \(p_\lambda^*V\). Each generator transforms by a
character of the finite root group, hence the term \emph{character frame}.
The monomial factors cancel the metric powers in
\eqref{eq:local-stage-reference}, giving
\[
 (p_\lambda^*K_{\nu,\lambda})(\widehat e_i,\widehat e_j)=\delta_{ij}.
\]
The transition matrices between these frames extend holomorphically and
invertibly across the root divisor. Thus the metric
\(\widetilde K_\nu\) also extends there as a smooth positive Hermitian
matrix.
This construction applies on SNC charts, including multiple intersections,
and the same \(M_\nu\) may be used on every boundary component.

We finally normalize the determinant. Let
\(\gamma_{\nu,a}=\sum_iw_{a,i}^{(\nu)}\), with the sum counted with
multiplicity. Choose smooth metrics \(k_{\det V}\) and \(k_a\) on
\(\det V\) and \(\cO_Y(B_a)\), respectively, and let \(\sigma_a\) be
the canonical divisor section. The metric
\[
 q_\nu^0=k_{\det V}\prod_a|\sigma_a|_{k_a}^{2\gamma_{\nu,a}}
\]
has the prescribed determinant exponents. Its curvature integral is the
parabolic degree. Solving the scalar Poisson equation on \(Y\)
multiplies \(q_\nu^0\) by a smooth positive factor and gives a metric
\(q_\nu\) satisfying
\[
 \ii\Lambda_{\omega_\nu}F_{q_\nu}=r c_\nu,\qquad
 c_\nu=\frac{2\pi\pardeg_{\omega_\nu}(V_{\nu,*})}
 {r\Vol_{\omega_\nu}(Y)}.
\]
Define
\[
 K_\nu=\left(\frac{q_\nu}{\det\widetilde K_\nu}\right)^{1/r}
          \widetilde K_\nu.
 \tag{4.9}\label{eq:normalized-stage-reference}
\]
The scalar ratio is smooth and positive on root charts. Consequently
\(K_\nu\) remains adapted and smooth on the root stack, and
\(\det K_\nu=q_\nu\). More precisely, choose fixed smaller charts
\(U_\lambda'\Subset U_\lambda\) covering \(B\). For every \(\nu\), there
is a constant \(C_\nu\ge1\) such that, for every \(\lambda\),
\[
 C_\nu^{-1}K_{\nu,\lambda}\le K_\nu\le C_\nu K_{\nu,\lambda}
 \qquad\text{on }U_\lambda'\setminus B.
\]
For the following arguments, we suppress \(\nu\) and write
\(K=K_\nu\) and \(q=q_\nu\).

\subsection{A metric with compactly supported error}
\label{subsec:compact-source-reference}

For a Hermitian metric \(H\) on \(V|_U\), its \emph{error in the
Hermitian--Einstein equation} is the endomorphism
\[
 \ii\Lambda_\omega\bigl(F_H+[\theta,\theta^{\dagger H}]\bigr)
 -c\Id_V,
 \qquad
 c=\frac{2\pi\pardeg_\omega(V_*)}{r\Vol_\omega(Y)}.
\]
We construct a metric \(H_0\) with the same determinant as \(K\),
comparable with \(K\), and solving the Hermitian--Einstein equation
near \(B\). Its error is then bounded and compactly supported in
\(U\), while the comparison preserves the parabolic growth.
The bounded error is used in the Moser estimate for the Dirichlet
solutions in Theorem~\ref{thm:fixed-stage-he-metrics}.
We first record the cutoff functions used in the construction and in
the subsequent integral identities.

\begin{lem}\label{lem:product-snc-cutoffs}
Let \(B\) be an SNC divisor on a compact Kähler manifold \(Y\), and put
\(U=Y\setminus B\). There are functions
\(\chi_\ell\in C_c^\infty(U)\), \(0\le\chi_\ell\le1\), such that
\(\chi_\ell\to1\) on compact subsets and
\[
 \|d\chi_\ell\|_{L^2(U)}\longrightarrow0.
\]
\end{lem}

\begin{proof}
Choose global divisor norms \(s_a\le1\) and put
\(T_a=1-\log s_a\). Take a fixed smooth function
\(\rho\) which is one on \((-\infty,1]\), zero on \([2,\infty)\), and
satisfies \(|\rho'|\le2\). Define
\[
 \chi_\ell=\prod_a
 \rho\left(\frac{\log T_a}{\ell}\right),
\]
using the global functions \(T_a\). On a divisor chart,
\[
 |d\log T_a|^2\le
 \frac{C\,|dz_a|^2}{|z_a|^2T_a^2}+C.
\]
The derivative of the \(a\)-th factor is supported where
\(e^\ell\le T_a\le e^{2\ell}\). Radial integration gives
\[
 \int
 \frac{|dz_a|^2}{\ell^2|z_a|^2T_a^2}
 \le\frac{C}{\ell^2}\int_{e^\ell}^{e^{2\ell}}\frac{dT}{T^2}
 \longrightarrow0.
\]
Finiteness of the atlas and the product rule prove the claimed decay of
the cutoff energy.
\end{proof}

For the energy estimates below, we use the following notation for the
Higgs--Dolbeault operator acting on an endomorphism \(s\):
\[
 \bar\partial_\theta s:=D''s
 =\bar\partial_{\End V}s+[\theta,s],
 \qquad [\theta,s]=\theta s-s\theta.
\]
Here \(D''=\bar\partial_V+\theta\) induces the displayed action on
\(\End V\). The two summands have types \((0,1)\) and \((1,0)\), so
\(|\bar\partial_\theta s|^2
=|\bar\partial_{\End V}s|^2+|[\theta,s]|^2\).
Thus \(\bar\partial_\theta s\in L^2\) requires both the ordinary
Dolbeault derivative and the Higgs commutator to belong to \(L^2\).

\begin{prop}\label{prop:compact-source-reference}
Fix a rational stage on a connected compact Kähler manifold
\((Y,\omega)\) with a finite SNC divisor \(B\), and put
\(U=Y\setminus B\). Assume that the stage is locally abelian and graded
semisimple. Let \((V_*,\theta)\) be the corresponding rank-\(r\)
logarithmic parabolic Higgs bundle, and let \(K\) be the Hermitian metric
defined by \eqref{eq:normalized-stage-reference} in
Subsection~\ref{subsec:adapted-references}, with the stage index omitted. Set
\[
 q=\det K,\qquad
 c=\frac{2\pi\pardeg_\omega(V_*)}
 {r\Vol_\omega(Y)}.
\]
There is a smooth Hermitian metric \(H_0\) on \(V|_U\) with the
following properties.

\begin{enumerate}
\item 
The determinant is unchanged:
\[
 \det H_0=q.
\]
There is an open neighborhood \(W_0\) of \(B\) on which \(H_0\)
solves the central Hermitian--Einstein equation:
\[
 \ii\Lambda_\omega
 \bigl(F_{H_0}+[\theta,\theta^{\dagger H_0}]\bigr)
 =c\Id_V
 \quad\text{on }W_0\setminus B.
\]
Consequently, its Hermitian--Einstein error
\[
 \Gamma_0=
 \ii\Lambda_\omega
 \bigl(F_{H_0}+[\theta,\theta^{\dagger H_0}]\bigr)-c\Id_V
\]
is smooth, bounded, and supported in a compact subset of \(U\).

\item 
For some \(C\ge1\),
\[
 C^{-1}K\le H_0\le CK\quad\text{on }U.
\]
In particular, \(H_0\) is adapted to \(V_*\).

\item 
The relative logarithm \(a=\log(K^{-1}H_0)\) satisfies
\[
 a\in L^\infty(U),\qquad
 D''a,\ D'_Ka\in L^2(U).
 \tag{4.10}\label{eq:compact-source-relative-energy}
\]
Here \(D'_K=\partial_K+\theta^{\dagger K}\) also acts on endomorphisms, and
the norms are computed using \(K\) and \(\omega\).

\end{enumerate}

All constants may depend on the fixed rational stage.
\end{prop}

\begin{proof}
We first construct a metric comparable with \(K\) that solves the
central equation near \(B\), then extend it over \(U\) by interpolation
away from \(B\). The estimates used in this construction will also give
the asserted finite relative energy.

The normalization in Subsection~\ref{subsec:adapted-references} gives
\(K\), adapted to \(V_*\), with
\[
 \ii\Lambda_\omega F_{\det K}=rc.
\]
Integrating this determinant equation fixes the value of \(c\) given in
the statement.

\smallskip
\noindent\textit{Step 1: a bounded scalar potential for the error of \(K\).}
Write
\[
 \Phi_K=
 \ii\Lambda_\omega
 \bigl(F_K+[\theta,\theta^{\dagger K}]\bigr)-c\Id_V.
\]
We first estimate its growth near \(B\). On a smaller chart in the finite
atlas from Subsection~\ref{subsec:adapted-references}, choose coordinates
\((z_1,\ldots,z_k,w_1,\ldots,w_{n-k})\) such that
\[
 B=\{z_1\cdots z_k=0\},\qquad r_a:=|z_a|\le1.
\]
Let \(m\) be a common denominator of the weights. Pass to the root chart
\[
 p:(t_1,\ldots,t_k,w)\longmapsto(t_1^m,\ldots,t_k^m,w),
 \qquad z_a=t_a^m.
\]
In the character frame constructed in
Subsection~\ref{subsec:adapted-references}, the metric
\(\widehat K=p^*K\) extends smoothly and positively across every
\(t_a=0\). The lifted Higgs field is logarithmic, so we can write
\[
 \widehat\theta=p^*\theta
 =\sum_{a=1}^k A_a\frac{dt_a}{t_a}
   +\sum_{\mu=1}^{n-k}B_\mu\,dw_\mu,
\]
with holomorphic matrix coefficients \(A_a,B_\mu\) on the root chart.
Below, \(\dagger\) denotes the adjoint with respect to \(\widehat K\).

We next explain why the residue \(R_a=A_a|_{t_a=0}\) is normal with
respect to the restriction of \(\widehat K\). Let \(\widehat V\) denote
the bundle extended in the character frame, and put
\(\widehat B_a=\{t_a=0\}\). There are two decompositions to distinguish.
First, the stabilizer \(\mu_m\) in the \(a\)-th root direction gives
\[
 \widehat V|_{\widehat B_a}=\bigoplus_\alpha\widehat Q_{a,\alpha}.
\]
Under the root correspondence, these character summands correspond to
the parabolic graded bundles along \(B_a\); thus \(\alpha\) indexes
parabolic weights, not residue eigenvalues. The residue preserves each
summand and corresponds there to \(m\) times the graded residue, since
\(p^*(dz_a/z_a)=m\,dt_a/t_a\). Second, graded semisimplicity gives
\[
 \widehat Q_{a,\alpha}=\bigoplus_\xi E_{a,\alpha,\xi},
 \qquad R_a|_{E_{a,\alpha,\xi}}=\xi\Id.
\]
These are holomorphic subbundles, by the spectral-projector argument in
Subsection~\ref{subsec:adapted-references}. 

The joint-spectral frames respect the residue eigenspaces within each graded bundle,
and the lifted local split metrics make their character frames
orthonormal, as shown in
Subsection~\ref{subsec:adapted-references}. Thus, if \(u,v\) belong to
distinct eigenspaces of \(R_a\), every local metric
\(\widehat K_\lambda\) entering the construction satisfies
\(\widehat K_\lambda(u,v)=0\).

On overlaps, these are the same intrinsic eigenspaces of the residue,
even though the chosen bases may differ. Expressing the local metrics
on a common root chart, their partition-of-unity sum therefore satisfies
\[
 \widehat{\widetilde K}(u,v)
 =\sum_\lambda(p^*\rho_\lambda)\widehat K_\lambda(u,v)=0.
\]
The final determinant normalization has the form
\(\widehat K=f\widehat{\widetilde K}\) with \(f>0\), so it also
preserves this orthogonality. On an eigenspace with eigenvalue \(\xi\),
we consequently have
\(R_a^\dagger=\bar\xi\Id\). Hence
\[
 [R_a,R_a^\dagger]=0.
\]

Integrability gives \([A_a,A_b]=[A_a,B_\mu]=0\). On \(t_a=0\), each
of these coefficients therefore preserves the eigenspaces of \(R_a\).
Since those eigenspaces are orthogonal, its adjoint preserves them as
well. It follows that
\[
 [A_a,A_b^\dagger]|_{t_a=0}=0,\qquad
 [A_a,B_\mu^\dagger]|_{t_a=0}=0.
\]
These commutators are smooth on the root chart. Taylor's estimate in
the two real coordinates of \(t_a\), uniformly on a smaller chart,
therefore yields
\[
 \begin{aligned}
 |[A_a,A_b^\dagger]|_{\widehat K}&\le C_\lambda |t_a|,\\
 |[A_a,B_\mu^\dagger]|_{\widehat K}&\le C_\lambda |t_a|.
 \end{aligned}
\]
The first estimate also applies when \(a=b\). Thus the coefficient that
would otherwise produce a \(|z_a|^{-2}\) singularity has a positive
order of vanishing.

We now express the estimates in the original coordinates. Put
\(\delta=1/m>0\). Since
\[
 dt_a=\frac1m t_a^{1-m}dz_a,\qquad
 \frac{dt_a}{t_a}=\frac1m\frac{dz_a}{z_a},\qquad
 |t_a|=r_a^\delta,
\]
smoothness of \(F_{\widehat K}\) gives
\[
 |F_K|_{K,\omega}
 \le C_\lambda\left(1+\sum_a r_a^{-2+2\delta}\right).
\]
Smoothness of \(\omega\) makes its norm
uniformly comparable with the coordinate norm on the smaller chart.

For the Higgs term, let \(C_K=[\theta,\theta^{\dagger K}]\), and use
subscripts to denote its coefficients in the ordinary coordinate
forms \(dz_a,dw_\mu\) and their conjugates. For example, on the root
chart the coefficient \((C_K)_{a\bar b}\) is represented by
\([A_a,A_b^\dagger]/(m^2z_a\bar z_b)\). The preceding vanishing
estimates give
\[
 \begin{aligned}
 |(C_K)_{a\bar b}|_K
   &\le C_\lambda r_a^{-1+\delta}r_b^{-1},\\
 |(C_K)_{a\bar\mu}|_K+|(C_K)_{\mu\bar a}|_K
   &\le C_\lambda r_a^{-1+\delta},\\
 |(C_K)_{\mu\bar\nu}|_K&\le C_\lambda.
 \end{aligned}
\]
In particular, the first line with \(a=b\) is bounded by
\(C_\lambda r_a^{-2+\delta}\).

Choose \(0<\epsilon_0\le\min\{\delta,1/2\}\). All the preceding
bounds can be combined into a single sum. Indeed, for distinct
\(a,b\), weighted Young's inequality gives
\[
 r_a^{-1+\epsilon_0}r_b^{-1}
 \le \frac{1-\epsilon_0}{2-\epsilon_0}r_a^{-2+\epsilon_0}
     +\frac1{2-\epsilon_0}r_b^{-2+\epsilon_0},
\]
and every remaining singular factor is bounded by a constant times
one of the \(r_a^{-2+\epsilon_0}\). After contraction with \(\omega\)
and enlargement of the constant to include \(c\Id_V\), we obtain
\[
 |\Phi_K|_K
 \le A_\lambda\left(1+\sum_a r_a^{-2+\epsilon_0}\right).
\]
The same \(\epsilon_0\) works on the finite atlas, since the root order
\(m\) is fixed.

Choose a compact divisor collar \(W\), with smooth nonempty outer
boundary on each component, contained in this atlas. Let
\(\{\zeta_\lambda\}\) be a smooth partition subordinate to the divisor
charts, augmented by a nonsingular chart when necessary, and define
\[
 F_{\mathrm{dom}}
 =\sum_\lambda\zeta_\lambda A_\lambda
 \left(1+\sum_a r_a^{-2+\epsilon_0}\right).
\]
On the nonsingular chart its coefficient is enlarged once to dominate
\(1+|\Phi_K|_K\). Thus
\[
 F_{\mathrm{dom}}\ge|\Phi_K|_K.
\]

We next bound the mass of \(F_{\mathrm{dom}}\) on small balls; this
will ensure that its Dirichlet Green potential is bounded. Fix a smooth
Riemannian metric \(g\) on \(Y\).
In one normal complex coordinate, the radial integral is
\[
 \int_0^s r^{-2+\epsilon_0}r\,dr
 =\frac{s^{\epsilon_0}}{\epsilon_0}.
\]
The same bound, up to a fixed factor, holds for a disk with any center:
\[
 \int_{|z-z_0|<s}|z|^{-2+\epsilon_0}\,dA(z)
 \le C s^{\epsilon_0}.
\]
To see this, if \(|z_0|\le2s\), enclose the disk in \(|z|<3s\); if
\(|z_0|>2s\), use \(|z|\ge |z_0|/2>s\) throughout the disk. The other
\(2n-2\) real coordinates contribute \(O(s^{2n-2})\). Coordinate
comparisons, smooth volume densities, and the finite partition are
uniformly bounded. Consequently,
there is a finite constant \(M\) such that, for all \(x\in W\) and
all sufficiently small \(s>0\),
\[
 \int_{W\cap B_s(x)}F_{\mathrm{dom}}\,dV_g
 \le M s^{2n-2+\epsilon_0}.
\]
Here \(B_s(x)\) is the \(g\)-ball of radius \(s\). In particular,
\(F_{\mathrm{dom}}\in L^1(W)\); enlarge \(M\), if necessary, so that
it also bounds its total mass.

Choose \(g\) to be a fixed constant rescaling of the Riemannian metric
underlying \(\omega\), so that
\[
 \mathcal L_D=d^*d=-\mathcal L_{\mathrm{sc}},\qquad
 \mathcal L_{\mathrm{sc}}
 =-2\ii\Lambda_\omega\bar\partial\partial.
\]
For \(N\ge1\), let \(F_N=\min(F_{\mathrm{dom}},N)\), interpreted as a
bounded weak source, and let \(W_N\) solve
\[
 \mathcal L_DW_N=F_N,\qquad W_N|_{\partial W}=0.
\]
Approximation by smooth increasing sources gives the same solution.
The Dirichlet Green kernel is nonnegative. Near its pole it is bounded
above by \(C\,d(x,y)^{2-2n}\), while it is bounded away from the pole.
Fix a sufficiently small radius \(s_0>0\) and divide the ball about the
pole into annuli with successive radii \(2^{-k}s_0\), \(k\ge0\).
The preceding ball-mass estimate gives
\[
 \begin{aligned}
 \int G(x,y)F_N(y)\,dV_g(y)
 &\le C M\sum_{k\ge0}
 (2^{-k}s_0)^{2-2n}
 (2^{-k}s_0)^{2n-2+\epsilon_0}+C\\
 &\le C_{\mathrm{Mor}}M.
 \end{aligned}
\]
Here \(C_{\mathrm{Mor}}\) depends on the fixed collar and background
metric, and is chosen for all the finitely many collar components.
Hence
\[
 0\le W_N\le C_{\mathrm{pot}}=C_{\mathrm{Mor}}M.
\]
The sources and solutions increase with \(N\). Their bounded limit
\(\overline W\) satisfies
\[
 \mathcal L_D\overline W=F_{\mathrm{dom}},
 \qquad
 0\le\overline W\le C_{\mathrm{pot}}
\]
distributionally. Since the source is smooth on \(W\setminus B\),
interior and boundary elliptic regularity make \(\overline W\) smooth
there.

\smallskip
\noindent\textit{Step 2: a central metric near the divisor.}
We use this bounded potential to control the Dirichlet solutions as their
inner boundaries approach \(B\). Exhaust \(W\setminus B\) by compact
smooth domains \(W_\ell\)
whose outer boundary is fixed and whose inner boundaries approach \(B\).
On every boundary component prescribe the metric \(K\), and solve the
determinant-fixed central Dirichlet problem
\[
 \det H_\ell=q,\qquad
 \ii\Lambda_\omega
 \bigl(F_{H_\ell}+[\theta,\theta^{\dagger H_\ell}]\bigr)
 =c\Id_V.
\]
Existence and uniqueness follow from
\cite[Propositions~2.1, 3.4 and Theorem~5.1]{ZZZ18}. Put
\[
 u_\ell=\log\frac{\Tr(K^{-1}H_\ell)}r,\qquad
 v_\ell=\log\frac{\Tr(H_\ell^{-1}K)}r.
\]
Both vanish on every outer and moving inner boundary because
\(H_\ell=K\) there. To estimate them in the interior, consider the
identity map in the two directions
\[
 I_\ell:(V,K,\theta)\longrightarrow(V,H_\ell,\theta),
 \qquad
 J_\ell:(V,H_\ell,\theta)\longrightarrow(V,K,\theta).
\]
Their squared norms in the respective Hom bundles are
\[
 |I_\ell|_{K,H_\ell}^{2}=\Tr(K^{-1}H_\ell),
 \qquad
 |J_\ell|_{H_\ell,K}^{2}=\Tr(H_\ell^{-1}K).
\]
Here the Hom norm is the Hilbert--Schmidt norm determined by the source
and target metrics. Both maps are holomorphic and commute with
\(\theta\); equivalently, they are holomorphic sections annihilated by
the induced Higgs field \(s\mapsto\theta s-s\theta\).

To apply the Bochner formula to these two sections, for a
metric \(G\), write
\[
 \Phi_G=\ii\Lambda_\omega
 \bigl(F_G+[\theta,\theta^{\dagger G}]\bigr)-c\Id_V,
\]
using the same constant \(c\) for both metrics. On
\(\Hom((V,G),(V,H))\), the contracted Higgs curvature acts on a section
\(s\) by
\[
 (c\Id_V+\Phi_H)s-s(c\Id_V+\Phi_G)
 =\Phi_Hs-s\Phi_G.
\]
Thus the common central term cancels. For a nowhere-zero holomorphic
Higgs morphism \(s\), the Chern Bochner formula and Cauchy--Schwarz give,
with our convention \(\mathcal L_D=2\ii\Lambda_\omega\bar\partial\partial\),
\[
 \mathcal L_D\log|s|_{G,H}^{2}
 \le
 2\frac{\operatorname{Re}\langle\Phi_Hs-s\Phi_G,s\rangle_{G,H}}
 {|s|_{G,H}^{2}}
 \le 2\bigl(|\Phi_H|_H+|\Phi_G|_G\bigr).
\]
Indeed, the derivative terms in the formula for \(\log|s|^2\) have
nonpositive sum. Since the induced Higgs field annihilates \(s\), its
commutator contributes a nonnegative squared norm to the contracted
curvature pairing, so replacing the Chern curvature by the Higgs
curvature preserves this upper bound. The last inequality uses the
endomorphism operator norms, which are bounded by the displayed
Hilbert--Schmidt norms; it does not require a prior comparison of
\(G\) and \(H\).

Apply this estimate first to \(I_\ell\) and then to \(J_\ell\).
The central equation says \(\Phi_{H_\ell}=0\), and the constant
normalization by \(r\) does not affect the Laplacian. Hence
\[
 \mathcal L_Du_\ell\le2|\Phi_K|_K\le2F_{\mathrm{dom}},\qquad
 \mathcal L_Dv_\ell\le2|\Phi_K|_K\le2F_{\mathrm{dom}}.
\]
The functions \(u_\ell-2\overline W\) and
\(v_\ell-2\overline W\) have nonpositive \(\mathcal L_D\) and are
nonpositive on every boundary component.
The maximum principle and the bound
\(0\le\overline W\le C_{\mathrm{pot}}\) therefore give
\[
 u_\ell,v_\ell\le2C_{\mathrm{pot}}
 \quad\text{on }W_\ell,
\]
uniformly in the moving inner boundary. The first trace bound controls
every eigenvalue of \(K^{-1}H_\ell\) from above, while the second
controls its reciprocal from above. Consequently,
\[
 C_{\mathrm{cmp}}^{-1}K\le H_\ell\le C_{\mathrm{cmp}}K,
 \qquad
 C_{\mathrm{cmp}}=2r\exp(2C_{\mathrm{pot}}).
\]

We next obtain a derivative estimate whose bound is independent of
\(\ell\), including as the inner boundary approaches \(B\). Put
\[
 s_\ell=\log(K^{-1}H_\ell),\qquad
 M=\log C_{\mathrm{cmp}}.
\]
The endomorphism \(s_\ell\) is \(K\)-self-adjoint, is trace free,
and vanishes on \(\partial W_\ell\). Its eigenvalues lie in
\([-M,M]\), so \(|s_\ell|_K\le\sqrt r\,M\).

We write out the Donaldson identity used here. For a self-adjoint
endomorphism \(s\), let \(\Psi_{\exp}(s)\) act on an
endomorphism-valued form by multiplying its component from the
\(x\)-eigenspace of \(s\) to the \(y\)-eigenspace by
\[
 \Psi_{\exp}(x,y)=\int_0^1 e^{t(y-x)}\,dt
 =\begin{cases}
 \dfrac{e^{y-x}-1}{y-x},&x\ne y,\\[4pt]
 1,&x=y.
 \end{cases}
\]
The boundary-value form of the identity
\cite[Proposition~2.6]{ZZZ18}, applied to \(K\) and \(H_\ell\), is
\[
 \int_{W_\ell}
 \bigl\langle\Psi_{\exp}(s_\ell)D''s_\ell,D''s_\ell
 \bigr\rangle_{K,\omega}\,dV_\omega
 =-\int_{W_\ell}\Tr(\Phi_Ks_\ell)\,dV_\omega.
\]
To see the origin of this formula, consider the path
\(K_t=K\exp(ts_\ell)\). Differentiating the contracted curvature
and integrating by parts gives
\[
 \frac{d}{dt}\int_{W_\ell}\Tr(\Phi_{K_t}s_\ell)\,dV_\omega
 =\int_{W_\ell}|D''s_\ell|_{K_t,\omega}^2\,dV_\omega.
\]
There is no boundary contribution because \(s_\ell=0\) on
\(\partial W_\ell\). On a component from an eigenvalue \(x\) to an eigenvalue
\(y\), the squared Hom norm for \(K_t\) is \(e^{t(y-x)}\) times
the one for \(K\). Integrating in \(t\) therefore gives the displayed
kernel, and the endpoint term vanishes because \(\Phi_{H_\ell}=0\).

For \(x,y\in[-M,M]\), the integral defining the kernel gives
\(\Psi_{\exp}(x,y)\ge e^{-2M}\). Consequently,
\[
 \begin{aligned}
 e^{-2M}\int_{W_\ell}|D''s_\ell|_{K,\omega}^2\,dV_\omega
 &\le -\int_{W_\ell}\Tr(\Phi_Ks_\ell)\,dV_\omega\\
 &\le \sqrt r\,M\int_W F_{\mathrm{dom}}\,dV_\omega.
 \end{aligned}
\]
Thus, with the finite constant
\(C_{\mathrm{en}}=e^{2M}\sqrt r\,M
\int_W F_{\mathrm{dom}}\,dV_\omega\),
\[
 \int_{W_\ell}|D''s_\ell|_{K,\omega}^2\,dV_\omega
 \le C_{\mathrm{en}}
 \qquad\text{for every }\ell.
\]
This bounds both \(\bar\partial_{\End V}s_\ell\) and
\([\theta,s_\ell]\) in \(L^2\), since their form types are
orthogonal.

Interior estimates and a diagonal subsequence give a limit
\(H_{\mathrm{col}}\), with convergence in \(C^\infty\) on compact
subsets of \(W'\setminus B\), where \(W'\Subset W\) is a smaller
open neighborhood of \(B\). It satisfies
\[
 \det H_{\mathrm{col}}=q,\qquad
 \ii\Lambda_\omega
 \bigl(F_{H_{\mathrm{col}}}
 +[\theta,\theta^{\dagger H_{\mathrm{col}}}]\bigr)
 =c\Id_V,
\]
and inherits the comparison
\[
 C_{\mathrm{cmp}}^{-1}K\le H_{\mathrm{col}}
 \le C_{\mathrm{cmp}}K.
\]
Its relative logarithm \(s=\log(K^{-1}H_{\mathrm{col}})\) is
\(K\)-self-adjoint and trace free, with eigenvalues in \([-M,M]\).
The comparison with the adapted \(K\) proves adaptedness
of \(H_{\mathrm{col}}\) at every real index.

We now pass the uniform derivative bound to this limit. Choose
relatively compact domains \(\Omega_j\) increasing to
\(W'\setminus B\). For each fixed \(j\), we have
\(\Omega_j\subset W_\ell\) for all sufficiently large \(\ell\),
and the local smooth convergence gives
\[
 \int_{\Omega_j}|D''s|_{K,\omega}^2\,dV_\omega
 =\lim_{\ell\to\infty}
 \int_{\Omega_j}|D''s_\ell|_{K,\omega}^2\,dV_\omega
 \le C_{\mathrm{en}}.
\]
The bound is independent of \(j\). Letting \(j\to\infty\) and
using monotone convergence therefore yields
\[
 \int_{W'\setminus B}|D''s|_{K,\omega}^2\,dV_\omega
 \le C_{\mathrm{en}}.
\]
The other derivative has the same norm because \(s\) is
\(K\)-self-adjoint. Indeed, compatibility of the Chern connection
with \(K\) and the adjoint rule for commutators give
\[
 (\bar\partial_{\End V}s)^{\dagger K}=\partial_Ks,
 \qquad
 [\theta,s]^{\dagger K}=-[\theta^{\dagger K},s].
\]
Since the two form types are orthogonal and taking adjoints preserves
norms, it follows that
\[
 |D'_Ks|_{K,\omega}^2
 =|\partial_Ks|_{K,\omega}^2
  +|[\theta^{\dagger K},s]|_{K,\omega}^2
 =|D''s|_{K,\omega}^2.
\]
We have therefore proved
\[
 D''s,\quad D'_Ks\in L^2(W'\setminus B,K,\omega).
\]

\smallskip
\noindent\textit{Step 3: construction of the global metric \(H_0\).}
Choose nested collars
\[
 W_1\Subset W_2\Subset W
\]
on which \(H_{\mathrm{col}}\) is defined. From any smooth positive
interior metric \(R\), form
\[
 G=\left(\frac q{\det R}\right)^{1/r}R.
\]
Then \(\det G=q\). On the compact annulus
\(W_2\setminus\overline W_1\), let
\[
 S=\log(H_{\mathrm{col}}^{-1}G).
\]
It is \(H_{\mathrm{col}}\)-self-adjoint and \(\Tr S=0\). For a cutoff
\(\chi\) which is zero near \(W_1\) and one near the exterior boundary
of \(W_2\), set
\[
 H_0=H_{\mathrm{col}}\exp(\chi S)
\]
on the annulus, and glue it to \(H_{\mathrm{col}}\) and \(G\) on the
two sides. This metric is smooth, positive, determinant preserving, and
adapted. It agrees with \(H_{\mathrm{col}}\) on a fixed
divisor collar, which we take as \(W_0\), so its Hermitian--Einstein error
vanishes there. On the compact complement of this collar, the metric
and its error are smooth. Thus \(\Gamma_0\) is bounded and compactly
supported in \(U\). The comparison with \(K\) and the relative
\(L^2\) derivative bounds for \(H_{\mathrm{col}}\), together with
smoothness on the compact complement, give the asserted global
comparison and finite relative energy for \(H_0\).

\end{proof}

\subsection{Hermitian--Einstein metrics at a rational stage}
\label{subsec:he-stages}

Fix a stable rational stage. We construct a Hermitian--Einstein metric
on \(U\) by solving Dirichlet problems on compact domains exhausting
\(U\), with boundary values given by the metric \(H_{0,\nu}\) from
Proposition~\ref{prop:compact-source-reference}. Its error in the
Hermitian--Einstein equation is bounded and compactly supported;
together with stability, this allows us to obtain estimates independent
of the exhaustion domain.

\begin{thm}
\label{thm:fixed-stage-he-metrics}
Fix a sufficiently late rational stage from
Corollary~\ref{cor:stable-rational-approximations}, and let
\((V_*,\theta)\) be one stable rank-\(r\) summand on \((Y,B)\). We use the smooth
Kähler form
\[
 \omega_\nu
 =\pi^*\omega_X+\delta_\nu\vartheta
 +A_\nu\ii\partial\bar\partial
   \sum_a(s_a+\eta_\nu^2)^{\epsilon_\nu}
\]
defined in \eqref{eq:kahler-schedule-form}. Put \(U=Y\setminus B\),
and let \(K_\nu\) and \(q_\nu=\det K_\nu\) be the metrics constructed in
Subsection~\ref{subsec:adapted-references}. Set
\[
 c_\nu=
 \frac{2\pi\pardeg_{\omega_\nu}(V_*)}
 {r\Vol_{\omega_\nu}(Y)}.
\]
There exists a smooth Hermitian metric \(h_\nu\) on \(V|_U\) such that
\[
 \det h_\nu=q_\nu,\qquad
 \ii\Lambda_{\omega_\nu}
 \bigl(F_{h_\nu}+[\theta,\theta^{\dagger h_\nu}]\bigr)
 =c_\nu\Id_V.
\]
Moreover, for some \(C_\nu\ge1\),
\[
 C_\nu^{-1}K_\nu\le h_\nu\le C_\nu K_\nu
 \quad\text{on }U,
\]
and
\[
 \bar\partial_\theta\log(K_\nu^{-1}h_\nu)
 \in L^2(U,K_\nu,\omega_\nu).
\]
In particular, \(h_\nu\) is adapted to \(V_*\):
\[
 P^{\mathbf{c}}(h_\nu)=V^{\mathbf{c}}
 \qquad\text{for every real local multi-index }\mathbf{c}.
\]
\end{thm}

\begin{proof}
Let \(H_{0,\nu}\) be the metric supplied by
Proposition~\ref{prop:compact-source-reference} for \(K_\nu\). It satisfies
\[
 \det H_{0,\nu}=q_\nu,\qquad
 C_\nu^{-1}K_\nu\le H_{0,\nu}\le C_\nu K_\nu.
\]
Its error
\[
 \Gamma_{0,\nu}:=
 \ii\Lambda_{\omega_\nu}
 \bigl(F_{H_{0,\nu}}+[\theta,\theta^{\dagger H_{0,\nu}}]\bigr)
 -c_\nu\Id_V
\]
is smooth, bounded, and compactly supported in \(U\). The relative energy
estimate
\eqref{eq:compact-source-relative-energy} gives
\[
 a_\nu=\log(K_\nu^{-1}H_{0,\nu})\in L^\infty,\qquad
 D''a_\nu,\ D'_{K_\nu}a_\nu\in L^2.
\]
Choose connected domains \(\Omega_{\nu,k}\Subset U\) with smooth
boundary, increasing to \(U\).
The Dirichlet existence and uniqueness theorem
\cite[Propositions~2.1, 3.4 and Theorem~5.1]{ZZZ18} gives
Hermitian metrics \(H_{\nu,k}\) satisfying
\[
 H_{\nu,k}=H_{0,\nu}\text{ on }\partial\Omega_{\nu,k},\qquad
 \det H_{\nu,k}=q_\nu,
\]
\[
 \ii\Lambda_{\omega_\nu}
 \bigl(F_{H_{\nu,k}}+[\theta,\theta^{\dagger H_{\nu,k}}]\bigr)
 =c_\nu\Id_V.
\]
Put
\[
 s_{\nu,k}=\log(H_{0,\nu}^{-1}H_{\nu,k}).
\]
It is \(H_{0,\nu}\)-self-adjoint, trace free, and zero on the boundary.
The stationary Donaldson identity is
\[
 \int_{\Omega_{\nu,k}}\Tr(\Gamma_{0,\nu}s_{\nu,k})\,dV_{\omega_\nu}
 +\int_{\Omega_{\nu,k}}
 \left\langle
 \Psi_{\exp}(s_{\nu,k})\bar\partial_\theta s_{\nu,k},
 \bar\partial_\theta s_{\nu,k}
 \right\rangle dV_{\omega_\nu}=0.
\]

We first prove that the \(L^2\) norms
\[
 \ell_{\nu,k}
 =\|s_{\nu,k}\|_{L^2(\Omega_{\nu,k},H_{0,\nu},\omega_\nu)}
\]
are bounded independently of \(k\).
Apply the Bochner inequalities for the identity map in both directions,
as in the proof of Proposition~\ref{prop:compact-source-reference}, to
\(H_{0,\nu}\) and \(H_{\nu,k}\). The resulting scalar inequalities
control the largest and smallest eigenvalues of \(s_{\nu,k}\), with
source bounded by \(\Gamma_{0,\nu}\). Extending \(|s_{\nu,k}|_{H_{0,\nu}}\) by zero and
applying the Sobolev inequality and Moser iteration for the fixed smooth
compact background gives constants \(A_{0,\nu},A_{1,\nu}\), independent
of \(k\), such that
\[
 \sup_{\Omega_{\nu,k}}|s_{\nu,k}|_{H_{0,\nu}}
 \le A_{0,\nu}+A_{1,\nu}\ell_{\nu,k}.
\]
Indeed, the zero boundary trace removes the boundary term in the weak
scalar inequalities, the \(L^\infty\) norm of \(\Gamma_{0,\nu}\) controls
their inhomogeneous terms, and Moser iteration gives the displayed
bound in terms of \(\ell_{\nu,k}\).

Suppose that a subsequence satisfies \(\ell_{\nu,k}\to\infty\), and set
\[
 u_{\nu,k}=\ell_{\nu,k}^{-1}s_{\nu,k}
\]
on \(\Omega_{\nu,k}\), extended by zero to \(U\). After discarding
finitely many terms, the preceding supremum estimate gives
\[
 \|u_{\nu,k}\|_{L^\infty(U,H_{0,\nu})}\le C_{\nu}^{\mathrm{sp}},
 \qquad
 \|u_{\nu,k}\|_{L^2(U,H_{0,\nu},\omega_\nu)}=1.
\]
Divide the stationary Donaldson identity by \(\ell_{\nu,k}\). In source
and target eigenspaces of \(u_{\nu,k}\), the positive kernel becomes
\[
 \mathcal K_{\ell_{\nu,k}}(x,y)
 =\ell_{\nu,k}\Psi_{\exp}(\ell_{\nu,k}x,\ell_{\nu,k}y)
 =\int_0^{\ell_{\nu,k}}e^{t(y-x)}\,dt.
\]
On the fixed square
\([ -C_{\nu}^{\mathrm{sp}},C_{\nu}^{\mathrm{sp}}]^2\), this kernel
dominates a positive constant. The first integral in the divided identity
is bounded in absolute value by
\[
 rC_{\nu}^{\mathrm{sp}}
 \|\Gamma_{0,\nu}\|_{L^1(U,\omega_\nu)}.
\]
Consequently the divided Donaldson identity gives
\[
 \sup_k
 \|\bar\partial_\theta u_{\nu,k}\|_{L^2(U,H_{0,\nu},\omega_\nu)}
 <\infty.
\]
Since \(u_{\nu,k}\) has zero Sobolev trace on
\(\partial\Omega_{\nu,k}\), its extension by zero creates no boundary
distribution.

Let
\[
 \mathcal C_1\Subset \mathcal C_2\Subset\cdots\Subset U
\]
be a smooth compact exhaustion. On each \(\mathcal C_a\), the metrics, the Higgs
field, and all background coefficients are smooth and bounded. The
identity
\[
 \bar\partial u_{\nu,k}
 =\bar\partial_\theta u_{\nu,k}-[\theta,u_{\nu,k}]
\]
together with the uniform \(L^\infty\) bound on \(u_{\nu,k}\) and the
\(L^2\) bound on \(\bar\partial_\theta u_{\nu,k}\) controls the
Dolbeault derivative.
Self-adjointness also controls the conjugate Chern derivative. Thus the
sequence is bounded in \(W^{1,2}(\mathcal C_a)\). After a diagonal extraction,
\(u_{\nu,k}\) converges strongly in \(L^2\), almost everywhere, and weakly
in \(W^{1,2}\) on every \(\mathcal C_a\). Its limit \(u_{\nu,\infty}\) is bounded,
\(H_{0,\nu}\)-self-adjoint, and trace free.

Since \(U\) has finite \(\omega_\nu\)-volume,
\(\Vol_{\omega_\nu}(U\setminus \mathcal C_a)\to0\). The uniform supremum bound gives
\[
 \sup_k\int_{U\setminus \mathcal C_a}|u_{\nu,k}|_{H_{0,\nu}}^2\,dV_{\omega_\nu}
 \le (C_{\nu}^{\mathrm{sp}})^2
     \Vol_{\omega_\nu}(U\setminus \mathcal C_a)
 \longrightarrow0.
\]
Fatou's lemma gives the same tail estimate for \(u_{\nu,\infty}\).
First choose \(a\) so that both tails are small, and then use strong
\(L^2\) convergence on \(\mathcal C_a\). This proves
\[
 u_{\nu,k}\longrightarrow u_{\nu,\infty}
 \quad\text{strongly in }L^2(U),\qquad
 \|u_{\nu,\infty}\|_{L^2(U)}=1.
\]
For every positive smooth function
\(\zeta\) on the fixed spectral square satisfying
\[
 \zeta(x,y)<\frac1{x-y}\qquad(y<x),
\]
the kernels \(\mathcal K_{\ell_{\nu,k}}\) dominate \(\zeta\) for all
large \(k\). Weak lower semicontinuity on \(\mathcal C_a\), followed by monotone
exhaustion and dominated convergence in the source term, gives
\[
 \int_U\Tr(\Gamma_{0,\nu} u_{\nu,\infty})\,dV_{\omega_\nu}
 +\int_U
 \left\langle
 \zeta(u_{\nu,\infty})\bar\partial_\theta u_{\nu,\infty},
 \bar\partial_\theta u_{\nu,\infty}
 \right\rangle dV_{\omega_\nu}
 \le0.
\]
Letting \(\zeta(x,y)\) increase without bound where \(y\ge x\)
forces the corresponding components of
\(\bar\partial_\theta u_{\nu,\infty}\) to vanish. Here \(x\) and \(y\)
are the source and target eigenvalues, respectively. Cyclicity of trace then
gives
\[
 \bar\partial\Tr(u_{\nu,\infty}^m)=0,
 \qquad 1\le m\le r.
\]
These traces are real, locally \(W^{1,2}\), and hence constant on the
connected manifold \(U\). Newton identities show that the characteristic
polynomial is constant almost everywhere. Since the endomorphism is
trace free and has norm one, it has distinct real eigenvalues
\[
 \lambda_1<\cdots<\lambda_t,\qquad t\ge2,
\]
with constant multiplicities. Write
\(u=u_{\nu,\infty}\) and let \(\Pi_i\) project onto its
\(\lambda_i\)-eigenspace. Each \(\Pi_i\) is a polynomial in \(u\),
\[
 \Pi_i=\prod_{j\ne i}
 \frac{u-\lambda_j\Id}{\lambda_i-\lambda_j},
\]
so it has the same local Sobolev regularity. For \(1\le\alpha<t\),
put \(p_\alpha=\sum_{i\le\alpha}\Pi_i\). Then
\[
 u_{\nu,\infty}
 =\lambda_t\Id-\sum_{\alpha=1}^{t-1}
  (\lambda_{\alpha+1}-\lambda_\alpha)p_\alpha.
\]
Differentiating \(up_\alpha=p_\alpha u\) gives, for \(i\ne j\),
\[
 \Pi_j(D''p_\alpha)\Pi_i
 =\frac{\mathbf1_{j\le\alpha}-\mathbf1_{i\le\alpha}}
        {\lambda_j-\lambda_i}\,
   \Pi_j(D''u)\Pi_i.
\]
The diagonal blocks vanish by \(p_\alpha^2=p_\alpha\).
We already know that \(\Pi_j(D''u)\Pi_i=0\) when \(j\ge i\).
Consequently \(D''p_\alpha\) has no component from
\(\im p_\alpha\) to its orthogonal complement. The \((0,1)\) and
\((1,0)\) components give weak holomorphicity and Higgs invariance,
respectively. The fixed nonzero spectral gaps and \(D''u\in L^2\)
also give finite energy:
\[
 (\Id-p_\alpha)\bar\partial p_\alpha=0,\qquad
 (\Id-p_\alpha)\theta p_\alpha=0,\qquad
 \bar\partial_\theta p_\alpha\in L^2(U).
\]
For a nonzero block \(\Pi_j(D''u)\Pi_i\), we have \(j<i\).
This block contributes to \(D''p_\alpha\) precisely for
\(j\le\alpha<i\). Its coefficient in the sum of projection energies is
\[
 \sum_{\alpha=j}^{i-1}
 \frac{\lambda_{\alpha+1}-\lambda_\alpha}
      {(\lambda_i-\lambda_j)^2}
 =\frac1{\lambda_i-\lambda_j}.
\]
This equals the limiting kernel coefficient \(1/(x-y)\), with
\(x=\lambda_i\) and \(y=\lambda_j\). Summing the orthogonal blocks
in the integral inequality for \(u_{\nu,\infty}\) therefore gives
\[
 \int_U\Tr(\Gamma_{0,\nu} u_{\nu,\infty})\,dV_{\omega_\nu}
 +\sum_{\alpha=1}^{t-1}
  (\lambda_{\alpha+1}-\lambda_\alpha)
  \int_U|\bar\partial_\theta p_\alpha|^2\,dV_{\omega_\nu}
 \le0.
\]

To apply parabolic stability, we must extend the generic images of
\(p_\alpha\) across \(B\) to coherent subsheaves of \(V\) on \(Y\).
We use the extension and degree conclusions of
Proposition~\ref{prop:power-log-interface} below, with \(h=H_{0,\nu}\),
and verify its metric hypotheses as follows.
Proposition~\ref{prop:compact-source-reference} gives adaptedness,
while \(\det H_{0,\nu}=q_\nu\) satisfies its determinant condition:
by Subsection~\ref{subsec:adapted-references}, the local coefficient of
\(q_\nu\) is a smooth strictly positive factor times
\(\prod_a|z_a|^{2\sum_iw_{a,i}^{(\nu)}}\), including at SNC crossings.
The contracted Hitchin--Simpson curvature is
\(\Gamma_{0,\nu}+c_\nu\Id_V\), which is integrable since
\(\Gamma_{0,\nu}\) is bounded and \(U\) has finite volume.

The one-divisor charts cover \(B\) away from its SNC crossings,
which form a closed analytic set of codimension at least two.
On a chart meeting only \(B_a=(z=0)\), take the local split
metric from \eqref{eq:local-stage-reference}:
\[
 g_a=\diag_i\bigl(|z|^{2w_{a,i}^{(\nu)}}\bigr).
\]
This is the model in Proposition~\ref{prop:power-log-interface} with
\(\beta_{a,i}=w_{a,i}^{(\nu)}\) and \(\kappa_{a,i}=0\), so no logarithmic
factor is needed. The local comparison of \(K_\nu\) with \(g_a\) from
Subsection~\ref{subsec:adapted-references}, followed by the comparison of
\(H_{0,\nu}\) with \(K_\nu\), gives
\(C^{-1}g_a\le H_{0,\nu}\le Cg_a\).
To check the relative derivative condition, put
\(b_a=\log(K_\nu^{-1}g_a)\). Let \(M_\nu\) be the common denominator
of the stage weights. Under the root map \(z=t^{M_\nu}\), both
\(K_\nu\) and \(g_a\) extend smoothly and positively in the same character
frame, so \(b_a\) is smooth there. Since
\(d\bar t=M_\nu^{-1}\bar t^{1-M_\nu}d\bar z\), on a smaller chart
\[
 |\bar\partial b_a|_{K_\nu,\omega_\nu}
 \le C\bigl(1+|z|^{-1+1/M_\nu}\bigr).
\]
The square of this bound is integrable against the normal area measure
\(|z|\,d|z|\,d\arg z\); hence \(\bar\partial b_a\in L^2\).
We also have \(\bar\partial a_\nu\in L^2\) by
\eqref{eq:compact-source-relative-energy}. The identity
\(H_{0,\nu}^{-1}g_a=e^{-a_\nu}e^{b_a}\), uniform comparison of the
three metrics, and differentiation of the matrix exponential and logarithm
give
\[
 \bigl|\bar\partial\log(H_{0,\nu}^{-1}g_a)\bigr|
 \le C\bigl(|\bar\partial a_\nu|+|\bar\partial b_a|\bigr).
\]
Here the norms use \(\omega_\nu\) and \(K_\nu\); comparison gives the
same conclusion with \(H_{0,\nu}\). Thus the relative derivative required
in Proposition~\ref{prop:power-log-interface} belongs to \(L^2\).

The projections \(p_\alpha\) are locally \(W^{1,2}\), have ranks strictly
between zero and \(r\), and satisfy the weak holomorphicity, Higgs
invariance, and finite-energy conditions proved above. The proposition
therefore extends their images to proper saturated reflexive invariant
parabolic subsheaves \(F_{\alpha,*}\subset V_*\) on \(Y\) and gives
\[
 \deg_{\an,H_{0,\nu}}(F_\alpha)
 =\pardeg_{\omega_\nu}(F_{\alpha,*}).
\]
Returning to the integral inequality for \(u_{\nu,\infty}\) and the
spectral projections \(p_\alpha\), substitution of these degree identities yields
\[
 2\pi\sum_{\alpha=1}^{t-1}
 (\lambda_{\alpha+1}-\lambda_\alpha)\rk(F_\alpha)
 \bigl[\mu(V_*)-\mu(F_{\alpha,*})\bigr]\le0.
\]
Every summand is positive by stability, a contradiction. Hence
\(\ell_{\nu,k}\) is uniformly bounded.

The supremum estimate for \(s_{\nu,k}\) now gives
\[
 \sup_{\Omega_{\nu,k}}|s_{\nu,k}|_{H_{0,\nu}}\le C_\nu.
\]

The interior estimate
\cite[Proposition~3.5]{ZZZ18}, followed by elliptic bootstrapping, now
gives a subsequential limit \(h_\nu\), with convergence smooth on compact
subsets. The determinant constraint and Hermitian--Einstein equation pass
to the limit, giving \(\det h_\nu=q_\nu\) and the central equation for
\(h_\nu\). By the uniform bound on \(s_{\nu,k}\) and the comparison
between \(H_{0,\nu}\) and \(K_\nu\), after enlarging \(C_\nu\),
\[
 C_\nu^{-1}K_\nu\le h_\nu\le C_\nu K_\nu\quad\text{on }U.
\]
Since \(K_\nu\) is adapted, so is \(h_\nu\).
Since \(\Gamma_{0,\nu}\) is compactly
supported and the logarithms are uniformly bounded, the Donaldson
identity also gives a domain-independent \(L^2\) bound for their
Higgs--Dolbeault derivatives. Weak lower semicontinuity gives
\(\bar\partial_\theta\log(H_{0,\nu}^{-1}h_\nu)\in L^2(U)\).
Combining this with the \(L^2\) estimate for \(D''a_\nu\) and
boundedness of the relative logarithms gives
\[
 \bar\partial_\theta\log(K_\nu^{-1}h_\nu)\in L^2(U).\qedhere
\]
\end{proof}

\subsection{Chern--Weil identities and curvature energy estimates}
\label{subsec:fixed-stage-chern-weil}

Fix one rational stage from Theorem~\ref{thm:fixed-stage-he-metrics}
and suppress the index \(\nu\). Let \(K\) be the Hermitian metric
constructed in Subsection~\ref{subsec:adapted-references} and normalized
by \eqref{eq:normalized-stage-reference}, and let \(h\) be the
Hermitian--Einstein metric from Theorem~\ref{thm:fixed-stage-he-metrics}.
We first prove the Chern--Weil identity
\eqref{eq:literal-canonical-quadratic} in
Proposition~\ref{prop:literal-reference-quadratic}: the parabolic Chern
number entering the Bogomolov--Gieseker inequality equals a normalized
quadratic integral of the Hitchin--Simpson curvature of \(K\).
We then use auxiliary
Dirichlet solutions with boundary value \(K\) to derive an upper bound
for the curvature energy of \(h\),
\[
 \mathcal E(h)=\int_U\left(
 |F_h+[\theta,\theta^{\dagger h}]|_{h,\omega}^2
 +2|\partial_h\theta|_{h,\omega}^2
 \right)dV_\omega.
\]
The exhaustion limits of these solutions are identified with \(h\) using
Theorem~\ref{thm:central-he-same-lattice}.
Theorem~\ref{thm:fixed-stage-energy-ledger} bounds this energy in terms
of parabolic Chern numbers. The corresponding estimate for the trace-free
curvature proves the Bogomolov--Gieseker inequality for \(V_*\). The bound
for \(\mathcal E(h)\) is also used in Corollary~\ref{cor:factor-energy} to prove
energy decay along the rational stages under the stated Chern-number
vanishing assumptions.

\begin{prop}
\label{prop:literal-reference-quadratic}
Let \(K\) be the metric defined by
\eqref{eq:normalized-stage-reference}. Put
\[
 \mathcal F_K
 =(\bar\partial+\theta+\partial_K+\theta^{\dagger K})^2,
 \qquad
 Q_K=\Tr(\mathcal F_K\wedge\mathcal F_K)
 -\frac1r\Tr(\mathcal F_K)\wedge\Tr(\mathcal F_K),
\]
and \(\Omega=\omega^{n-2}/(n-2)!\). Then
\(\mathcal F_K\in L^1(U,\omega,K)\), the form \(Q_K\wedge\Omega\)
is absolutely integrable, and
\[
 \int_UQ_K\wedge\Omega
 =8\pi^2
 \left\langle
 \parc_2(V_*)-\frac{r-1}{2r}\parc_1(V_*)^2,
 \frac{[\omega]^{n-2}}{(n-2)!}
 \right\rangle.
 \tag{4.11}\label{eq:literal-canonical-quadratic}
\]
\end{prop}

\begin{proof}
Put
\[
 C_K=[\theta,\theta^{\dagger K}],\qquad
 \Psi_K=F_K+C_K,\qquad B_K=\partial_K\theta.
\]
We first prove absolute integrability of the trace-free \((1,1)\)
quadratic term
\[
 Q_{11,K}=\Tr(\Psi_K^0\wedge\Psi_K^0).
\]

Choose an SNC chart as in Subsection~\ref{subsec:adapted-references}.
Write \(z_1,\ldots,z_m\) for its normal coordinates and
\(w_1,\ldots,w_s\) for its remaining coordinates, and set
\(r_a=|z_a|\). After a finite multi-root
substitution
\[
 z_a=t_a^{N_a},
\]
the metric \(K\) extends smoothly and positively in the character frame.
Its pulled-back Chern curvature is smooth. Since
\[
 dt_a=\frac1{N_a}t_a^{1-N_a}dz_a,
\]
there is a common \(\delta>0\), decreased once over the finite atlas,
such that a coefficient of \(F_K\) containing \(k_a\) occurrences of
\(dz_a\) or \(d\bar z_a\) satisfies
\[
 \left|(F_K)^0_{I\bar J}\right|_K
 \le C\prod_a r_a^{-k_a(1-\delta)}.
 \tag{4.12}\label{eq:literal-root-curvature}
\]
Trace-free projection has bounded fiberwise operator norm, so it
preserves this estimate.

We use the residue orthogonality proved in Step~1 of
Proposition~\ref{prop:compact-source-reference} to estimate the Higgs terms.
On the root chart, write
\[
 \widehat\theta
 =\sum_a A_a\frac{dt_a}{t_a}+\sum_\mu B_\mu\,dw_\mu,
 \qquad R_a=A_a|_{t_a=0},
\]
and let \(\nabla'\) be the Chern connection on
\(\End\widehat V\) induced by the smooth lifted metric \(\widehat K\).
The eigensubbundles of \(R_a\) are holomorphic and mutually orthogonal
for \(\widehat K|_{t_a=0}\). The Chern connection along this divisor
therefore preserves each eigensubbundle. The eigenvalues are constant:
the characteristic coefficients of the graded residue are holomorphic
on the compact divisor component. Since \(R_a\) acts by a constant
scalar on each eigensubbundle, we obtain
\[
 \left.\nabla'_v A_a\right|_{t_a=0}=0
 \qquad\text{for every vector }v\text{ tangent to }\{t_a=0\}.
\]
For the commutator \(C_K\), integrability gives
\([A_a,A_b]=[A_a,B_\mu]=0\). On \(t_a=0\), these coefficients
preserve the orthogonal eigensubbundles of \(R_a\), as do their
adjoints. Thus
\[
 [A_a,A_b^\dagger]|_{t_a=0}=0,\qquad
 [A_a,B_\mu^\dagger]|_{t_a=0}=0.
\]
Both commutators are smooth on the root chart, so their norms are
\(O(|t_a|)\), uniformly on a smaller chart, including its crossings.
Their coefficients in the original coordinate forms are
\[
 (C_K)_{a\bar b}
 =\frac{[A_a,A_b^\dagger]}{N_aN_bz_a\bar z_b},\qquad
 (C_K)_{a\bar\mu}
 =\frac{[A_a,B_\mu^\dagger]}{N_az_a},
\]
where the equalities are understood in the character frame after
pullback. Since \(|t_a|=r_a^{1/N_a}\), a common
\(\epsilon>0\), no larger than any \(1/N_a\), gives
\[
 |(C_K)^0_{a\bar b}|_K
 \le Cr_a^{-1+\epsilon}r_b^{-1},\qquad
 |(C_K)^0_{a\bar\mu}|_K
 \le Cr_a^{-1+\epsilon},
 \tag{4.13}\label{eq:literal-commutator-coefficients}
\]
together with the adjoint bounds obtained by interchanging the
holomorphic and antiholomorphic positions. For \(a=b\), the diagonal
coefficient \((C_K)_{a\bar a}\) has norm at most
\(Cr_a^{-2+\epsilon}\).
Coefficients involving only the \(w\)-directions are bounded.

Expand \(\Psi_K^0=F_K^0+C_K^0\) into ordinary coordinate monomials and
wedge two such monomials. For a fixed divisor coordinate \(z_a\), a nonzero wedge contains
zero, one, or two occurrences of \(dz_a,d\bar z_a\); two holomorphic or
two antiholomorphic occurrences make the wedge zero. With zero or one
normal occurrence, the radial density is bounded by
\(Cr_a^{1-k_a}dr_a\), where \(k_a\le1\), and is integrable.

Suppose there are two normal occurrences. If both occur in one
coefficient, the bounds in \eqref{eq:literal-root-curvature} and
\eqref{eq:literal-commutator-coefficients} improve the radial exponent
by at least \(\delta\) and \(\epsilon\), respectively. If the two covectors
occur in separate coefficients, one of them is holomorphic. The bound
for its coefficient improves the exponent by \(\delta\) if it comes
from \(F_K\), and by \(\epsilon\) if it comes from \(C_K\). Thus in either case
the radial density is bounded by
\[
 Cr_a^{-1+\eta_a}\,dr_a
\]
for some \(\eta_a>0\). Decrease the finitely many \(\eta_a\)'s to one
positive number valid for all chart monomials. Fubini's theorem proves
absolute integrability on the product chart. Fiberwise,
\[
 |\Tr(A\wedge B)|\le r\,|A|_K|B|_K,
\]
so summing over the coordinate terms and the finite atlas gives
\[
 Q_{11,K}\wedge\Omega\in L^1(U).
\]
We next estimate \(B_K=\partial_K\theta\). In root coordinates,
the coefficients of \(\widehat B=\partial_{\widehat K}\widehat\theta\)
are, for \(a\ne b\),
\[
 \widehat B_{ab}
 =\frac{\nabla'_a A_b}{t_b}-\frac{\nabla'_b A_a}{t_a},
 \qquad
 \widehat B_{a\mu}
 =\nabla'_a B_\mu-\frac{\nabla'_\mu A_a}{t_a}.
\]
Here \(\nabla'_a\) and \(\nabla'_\mu\) denote differentiation in
the \(t_a\)- and \(w_\mu\)-directions. The tangential derivative
identity above gives
\(\nabla'_b A_a=O(|t_a|)\) for \(b\ne a\) and
\(\nabla'_\mu A_a=O(|t_a|)\). The quotients by \(t_a\) are
therefore bounded. The other derivatives are smooth, so all
ordinary coordinate coefficients of \(\widehat B\) are bounded.
Returning to \(z_a=t_a^{N_a}\) contributes a factor
\(r_a^{-1+1/N_a}\) for each normal covector \(dt_a\).
Decreasing \(\epsilon\) over the finite atlas gives, for \(a\ne b\),
\[
 \begin{aligned}
 |(B_K)_{ab}|_K
 &\le Cr_a^{-1+\epsilon}r_b^{-1+\epsilon},\\
 |(B_K)_{a\mu}|_K+|(B_K)_{\mu a}|_K
 &\le Cr_a^{-1+\epsilon},\\
 |(B_K)_{\mu\nu}|_K&\le C.
 \end{aligned}
 \tag{4.14}\label{eq:literal-B-coefficients}
\]
The same bounds hold for \(B_K^\dagger\). Squaring these coefficients
against the smooth radial volume gives factors
\(r_a^{-1+2\epsilon}dr_a\), and hence
\[
 B_K,\ B_K^\dagger\in L^2(U).
\]
Step~1 of the proof of
Proposition~\ref{prop:compact-source-reference} gives
\(\Psi_K\in L^1(U)\). Since \(U\) has finite volume and
\(\mathcal F_K=\Psi_K+B_K+B_K^\dagger\), this also proves
\(\mathcal F_K\in L^1(U)\).

The trace-free Hitchin--Simpson curvature has the type decomposition
\[
 \mathcal F_K^0=B_K^0+\Psi_K^0+(B_K^\dagger)^0.
\]
After wedging with
\(\Omega=\omega^{n-2}/(n-2)!\), only the
\((1,1)(1,1)\) term and the two degree-even
\((2,0)(0,2)\) cross terms have top degree. Cyclicity of trace identifies
the two cross terms. Therefore
\[
 Q_K\wedge\Omega
 =
 \left[
 Q_{11,K}
 +2\Tr(B_K^0\wedge(B_K^\dagger)^0)
 \right]\wedge\Omega.
\]
The first summand is integrable by the estimate for \(Q_{11,K}\) above.
The absolute density of the second is bounded by \(C|B_K|^2dV_\omega\),
which is integrable because \(B_K\in L^2(U)\). Thus \(Q_K\wedge\Omega\) is absolutely
integrable. Charts with compact closure in \(U\) contribute only smooth bounded
terms.

It remains to calculate the Higgs transgression and its boundary flux.
Put
\[
 \Phi_K=\Tr(\theta\wedge\theta^{\dagger K}),\qquad
 P(T)=\Tr(T\wedge T)-\frac1r\Tr(T)\wedge\Tr(T).
\]
The global logarithmic one-form \(\Tr\theta\) is closed by
\cite[Theorem, p.~295]{Nog95}, which applies on the compact Kähler
manifold \(Y\). Hence
\(\Tr B_K=\partial\Tr\theta=0\), and similarly for its adjoint. A direct
graded expansion gives
\[
 \Tr(\mathcal F_K)=\Tr(F_K)
\]
and
\[
 P(\mathcal F_K)-P(F_K)
 =-2\partial\bar\partial\Phi_K.
 \tag{4.15}\label{eq:literal-higgs-transgression}
\]
The term \(P(\mathcal F_K)\wedge\Omega=Q_K\wedge\Omega\) is absolutely
integrable, as just proved. On every finite root chart the
pullback of \(K\) is smooth, so \(P(F_K)\wedge\Omega\) is the finite
ramified descent of a smooth top-degree form and is absolutely
integrable. Thus
\(\partial\bar\partial\Phi_K\wedge\Omega\) is absolutely integrable.

Let \(\chi_\ell\) be the product-SNC cutoffs of
Lemma~\ref{lem:product-snc-cutoffs}. Locally,
\[
 \bar\partial\Phi_K=-\Tr(\theta\wedge B_K^\dagger).
\]
A summand of \(\partial\chi_\ell\) associated with the face \(z_i=0\)
has the form
\[
 q_{i,\ell}
 \left(-\frac{dz_i}{z_i}+\partial\phi_i\right)
\]
times bounded cutoff factors for the other faces. The functions
\(q_{i,\ell}\) are uniformly bounded, their supports move into
\(z_i=0\), and \(q_{i,\ell}\to0\) pointwise on \(U\).

The coefficients of the Higgs field in the ordinary coordinate forms satisfy
\[
 \theta=\sum_aH_a\,dz_a+\sum_\mu H_\mu\,dw_\mu,\qquad
 |H_a|_K\le Cr_a^{-1},\quad |H_\mu|_K\le C.
\]
Consider the \(dz_i/z_i\) part of the cutoff derivative. If the Higgs
factor in \(-\Tr(\theta\wedge B_K^\dagger)\) also contributes
\(dz_i/z_i\), the coordinate wedge is zero. Otherwise the Higgs factor
is regular or contributes \(dz_b/z_b\) for \(b\ne i\). If \(S\) is the
set of normal antiholomorphic covectors in the \(B_K^\dagger\)
coefficient, then \(|S|\le2\), and
the bounds for \(\theta\) above and for \(B_K^\dagger\) in
\eqref{eq:literal-B-coefficients} bound the absolute radial density by
\[
 C\prod_c
 r_c^{\,1-\mathbf1_{c=i}
 -\mathbf1_{c=b}
 -\mathbf1_{c\in S}(1-\epsilon)}\,dr_c,
\]
with the \(b\)-term omitted for a regular Higgs covector. Every exponent
is greater than \(-1\): when a divisor coordinate occurs twice, the
exponent still contains the positive contribution \(\epsilon\);
elsewhere it is nonnegative.
The smooth \(\partial\phi_i\) part omits
\(\mathbf1_{c=i}\) and is bounded by the same or a better integrable
majorant. The tangential volume and coefficients of \(\Omega\) are
bounded.

This majorant is independent of \(\ell\),
while \(q_{i,\ell}\to0\) pointwise. Dominated convergence, followed by
the finite sums over faces, monomials, and charts, gives
\[
 \lim_{\ell\to\infty}
 \int_U\partial\chi_\ell\wedge
 \bar\partial\Phi_K\wedge\Omega=0.
\]
Compactly supported Stokes gives
\[
 \int_U\chi_\ell\partial\bar\partial\Phi_K\wedge\Omega
 =-\int_U\partial\chi_\ell\wedge
 \bar\partial\Phi_K\wedge\Omega.
\]
The left integrand is absolutely integrable and
\(\chi_\ell\to1\). Hence
\[
 \int_U\partial\bar\partial\Phi_K\wedge\Omega=0,
 \qquad
 \int_UP(\mathcal F_K)\wedge\Omega
 =\int_UP(F_K)\wedge\Omega.
\]

Finally, put \(G_K=\ii F_K/(2\pi)\). The first and second Chern
forms of \(K\) are
\[
 c_1(K)=\Tr(G_K),\qquad
 c_2(K)=\frac12\bigl(c_1(K)^2-\Tr(G_K\wedge G_K)\bigr).
\]
Hence
\[
 \begin{aligned}
 c_2(K)-\frac{r-1}{2r}c_1(K)^2
 &=-\frac12\left[
 \Tr(G_K\wedge G_K)
 -\frac1r\Tr(G_K)\wedge\Tr(G_K)\right],
 \end{aligned}
\]
and therefore
\[
 P(F_K)=8\pi^2
 \left(c_2(K)-\frac{r-1}{2r}c_1(K)^2\right).
\]
The Chern forms are smooth on the root stack. Their invariant descents
represent the parabolic Chern classes and the square of the first
parabolic Chern class by
\cite[Definition~5.2, Lemma~5.2, and Corollary~5.4]{BP17}. Consequently
\[
 \begin{aligned}
 &\int_U\left(c_2(K)-\frac{r-1}{2r}c_1(K)^2\right)\wedge\Omega\\
 &\qquad=\left\langle
 \parc_2(V_*)-\frac{r-1}{2r}\parc_1(V_*)^2,
 \frac{[\omega]^{n-2}}{(n-2)!}\right\rangle.
 \end{aligned}
\]
Combining this identity with the reduction from \(P(\mathcal F_K)\) to
\(P(F_K)\) and the ordinary Chern--Weil formula above proves
\eqref{eq:literal-canonical-quadratic}.
\end{proof}

\begin{lem}
\label{lem:connection-transgression}
Let \(H_t\), \(0\le t\le1\), be a smooth determinant-fixed path of
Hermitian metrics on a Higgs bundle over a complex manifold. Put
\[
 D''=\bar\partial_E+\theta,\qquad
 D'_t=\partial_{H_t}+\theta^{\dagger H_t},\qquad
 \mathcal F_t=(D''+D'_t)^2,\qquad
 a_t=H_t^{-1}\dot H_t,
\]
and
\[
 Q_t=\Tr(\mathcal F_t\wedge\mathcal F_t)
 -\frac1r\Tr(\mathcal F_t)\wedge\Tr(\mathcal F_t).
\]
Then
\[
 Q_1-Q_0
 =\bar\partial\left(
 2\int_0^1\Tr(D'_ta_t\wedge\mathcal F_t)\,dt
 \right).
\]
If the base has complex dimension \(n\ge2\) and carries a K\"ahler
form \(\omega\), put \(\Omega=\omega^{n-2}/(n-2)!\).
Then, for every compactly supported smooth scalar function \(\chi\),
\[
 \int\chi(Q_1-Q_0)\wedge\Omega
 =-2\int\bar\partial\chi\wedge
 \int_0^1\Tr(D'_ta_t\wedge\mathcal F_t)\,dt\wedge\Omega.
\]
\end{lem}

\begin{proof}
The identities \((D'')^2=(D'_t)^2=0\) and direct differentiation in a
holomorphic frame give
\[
 \dot{\mathcal F}_t=D''(D'_ta_t),\qquad D''\mathcal F_t=0.
\]
Graded cyclicity therefore gives
\[
 \frac{d}{dt}\Tr(\mathcal F_t\wedge\mathcal F_t)
 =2\bar\partial\Tr(D'_ta_t\wedge\mathcal F_t).
\]
Since the determinant is fixed, \(\Tr(a_t)=0\), and hence
\[
 \frac{d}{dt}\Tr(\mathcal F_t)
 =\bar\partial\Tr(D'_ta_t)=0.
\]
The trace-correction term in \(Q_t\) is therefore constant in \(t\).
Integrating in \(t\) gives the formula for \(Q_1-Q_0\). Multiplying this
formula by \(\chi\), wedging with the closed form \(\Omega\), and applying
Stokes' theorem gives the asserted integral identity.
\end{proof}

\begin{thm}
\label{thm:central-he-same-lattice}
Let \((V_*,\theta)\) be a stable regular locally abelian parabolic
logarithmic Higgs bundle on a connected compact Kähler pair \((Y,B)\),
and assume \(U=Y\setminus B\) is connected. Let \(h_0,h_1\) be smooth
positive metrics on \(V|_U\). Suppose:

\begin{enumerate}
\item both metrics are adapted to the same parabolic structure at
every real multi-index;
\item
\[
 \det h_0=\det h_1=q,
\]
where, in every simultaneous-split chart,
\[
 q=\overline q_\lambda
 \prod_a|z_{\lambda,a}|^{2\sum_i\beta_{\lambda,a,i}},
 \qquad \overline q_\lambda>0\text{ smooth};
\]
here \(\beta_{\lambda,a,i}\) is the parabolic weight of the \(i\)-th
frame vector along \(z_{\lambda,a}=0\), the sum is not reduced modulo
one, and \(\overline q_\lambda\) extends smoothly and strictly
positively across the whole chart;
\item for one real number \(c\),
\[
 \ii\Lambda_\omega
 \bigl(F_{h_k}+[\theta,\theta^{\dagger h_k}]\bigr)
 =c\Id_V,\qquad k=0,1.
\]
\end{enumerate}

Then \(h_0=h_1\).
\end{thm}

\begin{proof}
Equip each \(\mathcal O_Y(B_a)\) with a smooth Hermitian metric and
let \(s_a\) be the squared norm of its canonical section, scaled so
that \(s_a\le1\). Choose a compact neighborhood \(W\) of \(B\)
with smooth boundary contained in \(U\), and put
\[
 T_{\rm raw}=-\sum_a\log s_a.
\]
We construct a harmonic function with this logarithmic growth and zero
outer boundary value. The function
\(\ii\Lambda_\omega\bar\partial\partial T_{\rm raw}\) on
\(W\setminus B\) extends smoothly across \(B\): each
\(\bar\partial\partial\log s_a\) is the curvature of a smooth
divisor metric away from its zero set. Denote this smooth extension by
\(f\), and solve the scalar Dirichlet problem on \(W\):
\[
 \ii\Lambda_\omega\bar\partial\partial v=-f,
 \qquad v|_{\partial W}=-T_{\rm raw}|_{\partial W}.
\]
The solution \(v\) is smooth and bounded on \(W\), since the data
are smooth. Thus \(\tau=T_{\rm raw}+v\) is harmonic on \(W\setminus B\),
vanishes on \(\partial W\), and tends to infinity at \(B\).
Applying the maximum principle on truncations near \(B\) shows that
\(\tau\) is positive in the interior. Moreover,
\[
 |\tau-T_{\rm raw}|\le C.
\]

On each chart of the finite SNC atlas, let \(g_\lambda\) be the
common split monomial metric. There
\(T_{\rm raw}=-\sum_a\log|z_a|^2+O(1)\).
Adaptedness and the common determinant give, for every \(\delta>0\),
\[
 C_\delta^{-1}e^{-\delta T_{\rm raw}}g_\lambda
 \le h_k\le
 C_\delta e^{\delta T_{\rm raw}}g_\lambda,
 \qquad k=0,1.
\]
Indeed, the upper bounds follow by applying the growth condition, for
every positive tolerance, to the frame vectors. The determinant formula controls the
top exterior power. Applying the same bounds to the complementary
\((r-1)\)-fold wedges and using the cofactor formula controls the inverse
Gram matrix, giving the lower inequality.

Set
\[
 u_{01}=\log\frac{\Tr(h_0^{-1}h_1)}r,\qquad
 u_{10}=\log\frac{\Tr(h_1^{-1}h_0)}r.
\]
The common determinant makes both functions nonnegative. The comparison of \(\tau\) with \(T_{\rm raw}\) and the preceding
small-power bounds on \(h_0,h_1\) show that, for every
\(\epsilon>0\),
\[
 u_{01},u_{10}\le C_\epsilon+\epsilon\tau.
\]
The two Higgs-Hom Bochner formulas and equality of the central constants
give
\[
 -2\ii\Lambda_\omega\bar\partial\partial u_{01}\ge0,\qquad
 -2\ii\Lambda_\omega\bar\partial\partial u_{10}\ge0.
\]
If \(M_{01}\) is the supremum of \(u_{01}\) on the outer boundary, compare
\(u_{01}\) on \(\{\tau<R\}\) with
\[
 M_{01}+
 \left(\frac{\max\{C_\epsilon-M_{01},0\}}R+\epsilon\right)\tau.
\]
This harmonic function dominates \(u_{01}\) on both boundary components.
The maximum principle, followed by \(R\to\infty\) and then
\(\epsilon\downarrow0\), gives \(u_{01}\le M_{01}\). Applying the
same argument to \(u_{10}\) bounds both relative traces on
\(W\setminus B\). On the compact complement of the interior of \(W\),
both metrics are smooth and positive. We obtain
\[
 C^{-1}h_0\le h_1\le Ch_0
 \quad\text{on }U.
\]

Put \(s=\log(h_0^{-1}h_1)\). Use product-SNC cutoffs in the
determinant-fixed metric-variation identity. A Caccioppoli estimate first
bounds the relative Higgs--Dolbeault energy by
\(C\|d\chi_\ell\|_{L^2}^2\|s\|_\infty^2\). Letting
\(\ell\to\infty\), the common central source cancels and gives
\[
 \int_U
 \left\langle
 \Psi_{\exp}(s)\bar\partial_\theta s,
 \bar\partial_\theta s
 \right\rangle=0.
\]
The kernel is positive on the bounded spectral interval, so
\(\bar\partial_\theta s=0\). Thus \(s\) is parallel and commutes with \(\theta\).
Its spectral projections preserve every parabolic lattice. More than one
eigenvalue would split the stable parabolic Higgs bundle into nonzero
proper invariant summands whose rank-weighted slope average equals the
ambient slope, contradicting stability. Hence \(s=a\Id_V\). Equality of
determinants gives \(ra=0\), and therefore \(h_0=h_1\).
\end{proof}

\begin{lem}
\label{lem:k-dirichlet-comparison}
Let \(K\) be the metric defined by
\eqref{eq:normalized-stage-reference}.
Let \(\Omega\Subset U\) be connected
with nonempty smooth boundary, and let \(H_\Omega\) be the
determinant-fixed central solution with boundary value \(K\). Then
\[
 C^{-1}K\le H_\Omega\le CK
 \tag{4.16}\label{eq:k-dirichlet-uniform-comparison}
\]
and, for \(u_\Omega=\log(K^{-1}H_\Omega)\),
\[
 \|u_\Omega\|_{L^\infty(\Omega)}
 +\|D''u_\Omega\|_{L^2(\Omega)}
 \le C,
\]
where \(C\) depends on the fixed stage but not on \(\Omega\).
\end{lem}

\begin{proof}
Let \(h\) be the global Hermitian--Einstein metric from
Theorem~\ref{thm:fixed-stage-he-metrics}. Recall that
\(C_0^{-1}K\le h\le C_0K\) on \(U\) for some \(C_0\ge1\).
Both \(H_\Omega\) and \(h|_\Omega\) have the same determinant and solve
the central equation with the same constant. Their symmetric Donaldson distance satisfies
\[
 -2\ii\Lambda_\omega\bar\partial\partial
 \left(
 \Tr(h^{-1}H_\Omega)+\Tr(H_\Omega^{-1}h)-2r
 \right)\ge0.
\]
On \(\partial\Omega\), the metric \(H_\Omega=K\), so the comparison
between \(K\) and \(h\) bounds this distance there independently of
\(\Omega\). The maximum principle therefore bounds the
symmetric distance independently of \(\Omega\), which proves
\eqref{eq:k-dirichlet-uniform-comparison} and the asserted \(L^\infty\) bound on \(u_\Omega\).

Along \(K\exp(tu_\Omega)\), the Donaldson functional with fixed boundary
value on \(\Omega\) is convex. Its derivative vanishes at \(t=1\), where
the metric is the central solution \(H_\Omega\). Hence its value
at \(t=1\) is nonpositive:
\[
 \int_\Omega\Tr(\Gamma_Ku_\Omega)\,dV_\omega
 +\int_\Omega
 \left\langle
 \Psi_{\exp}(u_\Omega)D''u_\Omega,D''u_\Omega
 \right\rangle dV_\omega\le0.
\]
The divided-difference kernel has a positive minimum on the uniformly
bounded spectral square. Since \(\Gamma_K\in L^1(U)\),
\[
 m\|D''u_\Omega\|_{L^2}^2
 \le r\|u_\Omega\|_\infty\|\Gamma_K\|_{L^1(U)}.
\]
This proves the remaining assertion.
\end{proof}

\begin{thm}
\label{thm:fixed-stage-energy-ledger}
Retain one stage of Theorem~\ref{thm:fixed-stage-he-metrics}, suppress
\(\nu\), and write
\[
 H=[\omega],\qquad d=\pardeg_\omega(V_*),\qquad
 p=\left\langle\parch_2(V_*),
 \frac{H^{n-2}}{(n-2)!}\right\rangle,
\]
\[
 \begin{aligned}
 J&=\left\langle\parc_2(V_*)-\frac{r-1}{2r}\parc_1(V_*)^2,
 \frac{H^{n-2}}{(n-2)!}\right\rangle,\\
 c&=\frac{2\pi d}{r\Vol_\omega(Y)}.
 \end{aligned}
\]
Let \(h\) be the Hermitian--Einstein metric from that theorem, and set
\[
 \Psi_h=F_h+[\theta,\theta^{\dagger h}],\qquad
 B_h=\partial_h\theta.
\]
Write \(\Psi_h^0=\Psi_h-\frac1r(\Tr\Psi_h)\Id_V\) for the trace-free
part of \(\Psi_h\). Then
\[
 0\le
 \int_U\bigl(|\Psi_h^0|^2+2|B_h|^2\bigr)dV_\omega
 \le8\pi^2J.
 \tag{4.17}\label{eq:tracefree-energy-bound}
\]
Consequently \(J\ge0\), or equivalently,
\[
 \left\langle
 2r\parc_2(V_*)-(r-1)\parc_1(V_*)^2,
 \frac{H^{n-2}}{(n-2)!}\right\rangle\ge0.
 \tag{4.18}\label{eq:fixed-stage-discriminant}
\]
Moreover, the curvature energy is finite and satisfies
\[
 0\le
 \int_U\bigl(|\Psi_h|^2+2|B_h|^2\bigr)dV_\omega
 \le-8\pi^2p+\frac{4\pi^2d^2}{r\Vol_\omega(Y)}.
 \tag{4.19}\label{eq:one-sided-energy-ledger}
\]
\end{thm}

\begin{proof}
\begingroup\emergencystretch=1em
Let \(K\) be the metric defined by \eqref{eq:normalized-stage-reference},
and let \(h\) be the global Hermitian--Einstein metric from
Theorem~\ref{thm:fixed-stage-he-metrics}.
\par\endgroup

Let \(\sigma_B\) be the canonical section of
\(\cO_Y(B)\). Choose a smooth metric \(k_B\), rescaled so that
\[
 \rho=|\sigma_B|_{k_B}^2<e^{-2},
\]
and put \(\phi=-\log\rho\) on \(U\). The function \(\phi\) is proper and
unbounded. Since \(U\) is connected, regular values
\(\beta_k\to\infty\) can be chosen so that the connected component
\(\Omega_k\) of \(\{\phi<\beta_k\}\) containing a fixed base point is
relatively compact, has smooth nonempty boundary, increases, and exhausts
\(U\). Indeed, a path from the base point to any prescribed point has
compact image and finite \(\phi\)-maximum. Put
\[
 f_k=1-\frac{\phi}{\beta_k}
 \quad\text{on }\Omega_k.
\]
Then \(0\le f_k\le1\), \(f_k=0\) on the boundary, and \(f_k\to1\)
pointwise. Since
\(\partial\bar\partial\log\rho\) is the fixed smooth curvature form of
\(k_B\) on \(U\),
\[
 \|\partial\bar\partial f_k\|_{L^\infty(Y,\omega)}
 \le\frac{C_\rho}{\beta_k}.
\]

For a metric \(M\), let \(\mathcal F_M\) be its
Hitchin--Simpson curvature and set
\[
 P(T)=\Tr(T\wedge T)-\frac1r\Tr(T)\wedge\Tr(T),\qquad
 \Omega=\frac{\omega^{n-2}}{(n-2)!}.
\]
Proposition~\ref{prop:literal-reference-quadratic} gives absolute integrability
of \(P(\mathcal F_K)\wedge\Omega\) and the identity
\[
 \int_UP(\mathcal F_K)\wedge\Omega=8\pi^2J.
\]
Let \(H_k\) be the determinant-fixed central solution on
\(\Omega_k\) with boundary value
\[
 H_k=K\quad\text{on }\partial\Omega_k.
\]
Existence and uniqueness follow from
\cite[Propositions~2.1, 3.4 and Theorem~5.1]{ZZZ18}.
Put
\[
 u_k=\log(K^{-1}H_k),\qquad
 H_{t,k}=K\exp(tu_k),\qquad
 D'_{t,k}=\partial_{H_{t,k}}+\theta^{\dagger H_{t,k}},
\]
\[
 \mathcal F_{t,k}=(D''+D'_{t,k})^2,\qquad
 Q_{t,k}=P(\mathcal F_{t,k}).
\]
Lemma~\ref{lem:k-dirichlet-comparison}, applied with \(h\) as the
comparison metric, gives
\[
 \|u_k\|_{L^\infty(\Omega_k)}
 +\|D''u_k\|_{L^2(\Omega_k)}\le C
\]
with \(C\) independent of \(k\). In particular all path metrics are
uniformly equivalent to \(K\).

Define
\[
 I_k(t)=\int_{\Omega_k}f_kQ_{t,k}\wedge\Omega.
\]
Lemma~\ref{lem:connection-transgression} and the
\(D'_{t,k}\)-Bianchi identity give
\[
 I_k'(t)
 =-2\int_{\Omega_k}
 \partial\bar\partial f_k\wedge
 \Tr(u_k\mathcal F_{t,k})\wedge\Omega.
\]
The first boundary term vanishes because \(f_k=0\) on
\(\partial\Omega_k\), and the second vanishes because \(u_k=0\) there.
Differentiating once more and using
\(\partial_t\mathcal F_{t,k}=D''D'_{t,k}u_k\) gives
\[
 I_k''(t)
 =2\int_{\Omega_k}
 \partial\bar\partial f_k\wedge
 \Tr(D''u_k\wedge D'_{t,k}u_k)\wedge\Omega.
\]
Again the omitted boundary term contains \(u_k\) and vanishes. Since
\[
 \|\partial\bar\partial f_k\|_{L^\infty}
 \le\frac{C_\rho}{\beta_k},
\]
self-adjointness of \(u_k\), uniform comparison of the path metrics, and the uniform
\(L^2\) bound on \(D''u_k\) imply
\[
 |I_k''(t)|\le\frac C{\beta_k}.
\]
At \(t=0\), the same estimates and the integrability
\(\mathcal F_{0,k}=\mathcal F_K\in L^1(U)\) from
Proposition~\ref{prop:literal-reference-quadratic} give
\[
 |I_k'(0)|\le\frac C{\beta_k}.
\]
Taylor's formula therefore yields
\[
 |I_k(1)-I_k(0)|\le\frac C{\beta_k}.
\]
Absolute integrability of \(P(\mathcal F_K)\wedge\Omega\) and dominated
convergence give
\[
 I_k(0)\longrightarrow8\pi^2J,
 \qquad I_k(1)\longrightarrow8\pi^2J.
 \tag{4.20}\label{eq:k-weighted-number-limit}
\]

Centrality and the pointwise Kähler identity give
\[
 P(\mathcal F_{H_k})\wedge\Omega
 =
 \bigl(2|B_{H_k}|^2+|\Psi_{H_k}^0|^2\bigr)dV_\omega.
 \tag{4.21}\label{eq:dirichlet-positive-density}
\]
Here \(\Tr B_{H_k}=\partial\Tr\theta=0\), since \(\Tr\theta\) is
closed, as noted in the proof of
Proposition~\ref{prop:literal-reference-quadratic}.

The uniform comparison
\eqref{eq:k-dirichlet-uniform-comparison}, the central equation, and
\cite[Proposition~3.5]{ZZZ18} give a subsequence converging smoothly on
compact subsets to a central metric \(\widetilde h\). It has determinant
\(q\), and the comparison with the adapted \(K\) proves that it
is adapted to the rational parabolic structure. Since the stage is
stable and \(h\) is also adapted, with determinant \(q\) and the same
central constant, Theorem~\ref{thm:central-he-same-lattice} gives
\[
 \widetilde h=h.
\]

Let \(\mathcal C\Subset U\). Smooth convergence on \(\mathcal C\), uniform convergence
\(f_k\to1\), the nonnegative identity
\eqref{eq:dirichlet-positive-density}, and
\eqref{eq:k-weighted-number-limit} give
\[
 \int_{\mathcal C}\bigl(2|B_h|^2+|\Psi_h^0|^2\bigr)dV_\omega
 \le8\pi^2J.
\]
Exhausting \(U\) and using monotone convergence proves
\eqref{eq:tracefree-energy-bound}. Its nonnegative left-hand side gives
\(J\ge0\), and hence \eqref{eq:fixed-stage-discriminant}.

It remains to add the determinant contribution. Put
\[
 \alpha=\frac{\ii}{2\pi}\Tr F_h.
\]
The prescribed determinant metric makes \(\alpha\) extend as a smooth
Chern--Weil representative of \(\parc_1(V_*)\), and the trace of the central
equation gives
\[
 \Lambda_\omega\alpha=\frac{rc}{2\pi}.
\]
The scalar Kähler identity
\[
 \alpha^2\wedge\frac{\omega^{n-2}}{(n-2)!}
 =
 \bigl((\Lambda_\omega\alpha)^2-|\alpha|^2\bigr)dV_\omega
\]
therefore yields
\[
 \frac1r\int_U|\Tr\Psi_h|^2dV_\omega
 =rc^2\Vol_\omega(Y)
 -\frac{4\pi^2}{r}
 \left\langle \parc_1(V_*)^2,
 \frac{H^{n-2}}{(n-2)!}\right\rangle.
\]
Using
\[
 |\Psi_h|^2=|\Psi_h^0|^2+\frac1r|\Tr\Psi_h|^2,
 \qquad
 J=-p+\frac1{2r}
 \left\langle \parc_1(V_*)^2,\frac{H^{n-2}}{(n-2)!}\right\rangle
\]
in \eqref{eq:tracefree-energy-bound} gives
\eqref{eq:one-sided-energy-ledger}. This also proves finiteness.
\end{proof}

We now let the rational stage vary. Under the numerical vanishing
hypotheses of Theorem~\ref{thm:harmonic-main}, the preceding energy
estimate has the following consequence for each stable summand.

\begin{cor}\label{cor:factor-energy}
For one original stable summand \((E_*,\theta)\) of rank \(r\), let \(h_\nu\) be
the global Hermitian--Einstein metrics of
Theorem~\ref{thm:fixed-stage-he-metrics} on \(U=Y\setminus B\).
If
\[
 \pardeg_{\omega_X}(E_*)=0,\qquad
 \left\langle\parch_2(E_*),
 \frac{[\omega_X]^{n-2}}{(n-2)!}\right\rangle=0,
\]
then their curvature energies satisfy
\[
 \mathcal E_\nu:=
 \int_U\left(
 |F_{h_\nu}+[\theta,\theta^{\dagger h_\nu}]|^2
 +2|\partial_{h_\nu}\theta|^2
 \right)dV_{\omega_\nu}
 \longrightarrow0.
\]
The norms use \(h_\nu\) and \(\omega_\nu\).
\end{cor}

\begin{proof}
Put
\[
 \begin{aligned}
 d_\nu&=\pardeg_{\omega_\nu}(V_{\nu,*}),\\
 p_\nu&=\left\langle\parch_2(V_{\nu,*}),
 \frac{[\omega_\nu]^{n-2}}{(n-2)!}\right\rangle,\\
 \Vol_\nu&=\Vol_{\omega_\nu}(Y).
 \end{aligned}
\]
Theorem~\ref{thm:fixed-stage-energy-ledger} gives
\[
 0\le\mathcal E_\nu
 \le-8\pi^2p_\nu
 +\frac{4\pi^2d_\nu^2}{r\Vol_\nu}.
 \tag{4.22}\label{eq:energy-ledger}
\]
By Proposition~\ref{prop:parabolic-class-bridge}, continuity of
the parabolic classes and the projection formula give
\[
 d_\nu\longrightarrow\pardeg_{\omega_X}(E_*)=0,\qquad
 p_\nu\longrightarrow
 \left\langle\parch_2(E_*),
 \frac{[\omega_X]^{n-2}}{(n-2)!}\right\rangle=0.
\]
Since \([\omega_\nu]\to\pi^*[\omega_X]\) and \(\pi\) has degree one,
\[
 \Vol_\nu
 =\int_Y\frac{[\omega_\nu]^n}{n!}
 \longrightarrow
 \int_X\frac{\omega_X^n}{n!}>0.
\]
Substitution into \eqref{eq:energy-ledger} proves the claim.
\end{proof}

\section{The Bogomolov--Gieseker inequality}\label{sec:bg}

We now prove Theorem~\ref{thm:bg-main} by passing to the limit in the
inequality \eqref{eq:fixed-stage-discriminant}.

\begin{thm}[Theorem~\ref{thm:bg-main}]\label{thm:bg-conclusion}
Let \((X,\omega)\) be a connected compact Kähler manifold of complex
dimension \(n\ge2\), let \(D\subset X\) be an SNC divisor, and let
\((E_*,\theta)\) be an \(\omega\)-stable regular parabolic Higgs sheaf of
rank \(r\) on \((X,D)\). Then
\[
 \int_X\Delta_{\mathrm{par}}(E_*)
 \wedge\frac{\omega^{n-2}}{(n-2)!}\ge0.
 \tag{5.1}\label{eq:bg-conclusion}
\]
\end{thm}

\begin{proof}
Take the fixed logarithmic flag model
\(\pi:(Y,B)\to(X,D)\) constructed in
Theorem~\ref{thm:fixed-terminal-data} and completed by
Proposition~\ref{prop:fixed-model-correction}. Its underlying logarithmic
Higgs bundle \((V,\theta_Y)\) has simultaneously split flags with locally
free quotients and semisimple induced residues.
Let \(w^\infty\) be the weights determined by the original parabolic
structure \(E_*\), as in Subsection~\ref{subsec:fixed-flag-weights},
and let \(V_{\infty,*}\) be the corresponding parabolic structure on
the completed model. It agrees with \(E_*\) over the common
open set \(X\setminus Z\), where \(\operatorname{codim}_X Z\ge3\).

Corollary~\ref{cor:stable-rational-approximations} supplies strictly
ordered rational weights \(w^{(\nu)}\to w^\infty\) and smooth Kähler forms
\(\omega_\nu\) on \(Y\), with
\[
 H_\nu=[\omega_\nu]
 =\pi^*[\omega]+\delta_\nu[\vartheta],
 \qquad \delta_\nu\downarrow0.
\]
For every sufficiently large \(\nu\), the associated parabolic Higgs
bundle \(V_{\nu,*}\) is \(\omega_\nu\)-stable, locally abelian, and
graded-semisimple. Theorem~\ref{thm:fixed-stage-he-metrics} therefore
provides a Hermitian--Einstein metric, and
Theorem~\ref{thm:fixed-stage-energy-ledger} gives
\[
 \left\langle
 \Delta_{\mathrm{par}}(V_{\nu,*}),
 \frac{H_\nu^{n-2}}{(n-2)!}
 \right\rangle\ge0.
\]
Here we use the identity
\(\Delta_{\mathrm{par}}=2r\parc_2-(r-1)\parc_1^2\).

It remains to identify the limit of these nonnegative numbers.
By Proposition~\ref{prop:parabolic-class-bridge} and the continuity in
the weights proved there,
\[
 \begin{aligned}
 \Delta_{\mathrm{par}}(V_{\nu,*})
 &\longrightarrow\Delta_{\mathrm{par}}(V_{\infty,*})
 \quad\text{in }H^4(Y,\RR),\\
 \pi_*\Delta_{\mathrm{par}}(V_{\infty,*})
 &=\Delta_{\mathrm{par}}(E_*).
 \end{aligned}
\]
Since \(H_\nu\to\pi^*[\omega]\), continuity of cup products and the
projection formula give
\[
 \begin{aligned}
 \int_X\Delta_{\mathrm{par}}(E_*)
       \wedge\frac{\omega^{n-2}}{(n-2)!}
 &=\left\langle
 \Delta_{\mathrm{par}}(V_{\infty,*}),
 \frac{(\pi^*[\omega])^{n-2}}{(n-2)!}
 \right\rangle\\
 &=\lim_{\nu\to\infty}
 \left\langle
 \Delta_{\mathrm{par}}(V_{\nu,*}),
 \frac{H_\nu^{n-2}}{(n-2)!}
 \right\rangle
 \ge0.
 \end{aligned}
\]
This proves Theorem~\ref{thm:bg-main}.
\end{proof}

\section{Pluriharmonic metrics and parabolic prolongation}\label{sec:analytic-tools}

We now turn to Theorem~\ref{thm:harmonic-main}. For one stable summand,
let \(h_\nu\) be the Hermitian--Einstein metrics from
Theorem~\ref{thm:fixed-stage-he-metrics}. Their curvature energy,
estimated in Theorem~\ref{thm:fixed-stage-energy-ledger}, is
\[
 \mathcal E_\nu
 =\int_{Y\setminus B}\left(
 |F_{h_\nu}+[\theta,\theta^{\dagger h_\nu}]|^2
 +2|\partial_{h_\nu}\theta|^2
 \right)dV_{\omega_\nu},
\]
where the norms are computed using \(h_\nu\) and the K\"ahler form
\(\omega_\nu\) on the fixed divisor complement \(Y\setminus B\).
Under the vanishing of the parabolic degree and the
integrated parabolic second Chern character assumed in
Theorem~\ref{thm:harmonic-main}, Corollary~\ref{cor:factor-energy}
gives \(\mathcal E_\nu\to0\).
Section~\ref{sec:compactness} establishes the required
compactness and growth estimates, then combines them with the analytic
tools collected here to complete the proof in
Subsection~\ref{subsec:adapted-pluriharmonic-metrics}.

Lemma~\ref{lem:energy-closure} shows that a smooth limit of metrics whose
Hitchin--Simpson curvature energies tend to zero is pluriharmonic.
With the required growth control, Proposition~\ref{prop:acceptable-prolongation}
establishes acceptability of the Hermitian bundle in the sense of
Mochizuki \cite[Definition~2.7]{Moc03-ATHB-I} and identifies the
growth filtration of the limiting metric with the original parabolic
filtration at every real index on \(X\). In particular, the original
parabolic Higgs sheaf is a locally abelian parabolic Higgs bundle.
Proposition~\ref{prop:power-log-interface} extends the images of weak
Higgs-invariant projections to parabolic Higgs subsheaves and identifies
their analytic and parabolic degrees. This allows stability to be used
in the compactness argument of Section~\ref{sec:compactness}, as it was
already used in the proof of
Theorem~\ref{thm:fixed-stage-he-metrics}.

\begin{lem}\label{lem:energy-closure}
Let \(U\) be a complex manifold and let \(E\) be a finite-rank holomorphic vector bundle on \(U\) with a fixed smooth Higgs field \(\theta\). Let \(U_\nu\subset U\) be open sets such that every relatively compact subset of \(U\) lies in \(U_\nu\) for all sufficiently large \(\nu\). Let \(\omega_\nu\) be a Kähler metric and \(h_\nu\) a smooth positive Hermitian metric on \(E|_{U_\nu}\). Assume that
\[
 \omega_\nu\longrightarrow\omega,\qquad
 h_\nu\longrightarrow h_\infty
\]
in \(C^\infty\) on every compact subset, where \(\omega\) is a Kähler
form on \(U\) and \(h_\infty\) is a smooth positive Hermitian metric.
Put
\[
 \Psi_\nu=F_{h_\nu}+[\theta,\theta^{\dagger h_\nu}],
 \qquad
 Q_\nu=\partial_{h_\nu}\theta,
\]
and
\[
 \mathcal E_\nu
 =\int_{U_\nu}\bigl(|\Psi_\nu|^2+2|Q_\nu|^2\bigr)\,dV_{\omega_\nu}.
\]
If every \(\mathcal E_\nu\) is finite and \(\mathcal E_\nu\to0\), then
\[
 F_{h_\infty}+[\theta,\theta^{\dagger h_\infty}]=0,
 \qquad
 \partial_{h_\infty}\theta=0
\]
on \(U\). Equivalently, the Hitchin--Simpson connection determined by \(h_\infty\) is flat.
\end{lem}

\begin{proof}
Fix \(\mathcal C\Subset U\). For all sufficiently large \(\nu\), \(\mathcal C\subset U_\nu\). Smooth convergence of \(h_\nu\) gives smooth convergence of the Chern connections and curvatures on \(\mathcal C\). Since \(\theta\) is fixed and smooth, its adjoint and its covariant derivative depend smoothly on \(h_\nu\) and its derivatives. Thus
\[
 \Psi_\nu\longrightarrow
 \Psi_\infty:=F_{h_\infty}+[\theta,\theta^{\dagger h_\infty}],
 \qquad
 Q_\nu\longrightarrow Q_\infty:=\partial_{h_\infty}\theta
\]
uniformly on \(\mathcal C\). The tensor norms and volume forms also converge uniformly. Therefore
\[
 \begin{aligned}
 &\int_{\mathcal C}\bigl(|\Psi_\infty|^2+2|Q_\infty|^2\bigr)\,dV_\omega\\
 &\qquad=
 \lim_{\nu\to\infty}
 \int_{\mathcal C}\bigl(|\Psi_\nu|^2+2|Q_\nu|^2\bigr)\,dV_{\omega_\nu}
 \le\lim_{\nu\to\infty}\mathcal E_\nu=0.
 \end{aligned}
\]
The limiting integrand is continuous and nonnegative. Hence
\(\Psi_\infty=Q_\infty=0\) on \(\mathcal C\), and therefore on \(U\).

The curvature of
\[
 D_\infty=\bar\partial_E+\partial_{h_\infty}
 +\theta+\theta^{\dagger h_\infty}
\]
decomposes into the \((1,1)\)-term \(\Psi_\infty\), the \((2,0)\)-term \(Q_\infty\), and its Hermitian-adjoint \((0,2)\)-term. The pure Higgs terms vanish because \(\theta\wedge\theta=0\), and the holomorphicity terms vanish because \((E,\theta)\) is a holomorphic Higgs bundle. Thus \(D_\infty^2=0\).
\end{proof}

\begin{prop}
\label{prop:acceptable-prolongation}
Let \(X\) be a connected complex manifold, let
\(D=\bigcup_iD_i\) be an SNC divisor, and let
\((E_*,\theta)\) be a regular parabolic Higgs sheaf.
In particular, \(E\) is locally free on \(X\setminus D\). Let
\(Z\subset D\) be closed analytic of codimension at least two, containing
the locus where \(E\) is not locally free or its divisor flags do not
split simultaneously. Let \(H\) be a smooth
positive pluriharmonic metric on \(E|_{X\setminus D}\).

Assume that, on \(X\setminus Z\), the metric growth lattice agrees with
the original parabolic lattice at every real multi-index:
\[
 P^{\mathbf{c}}(H)|_{X\setminus Z}=E^{\mathbf{c}}|_{X\setminus Z}.
 \tag{6.1}\label{eq:growth-off-Z}
\]
Here \(P^{\mathbf c}(H)\) is the growth prolongation of
Definition~\ref{defn:growth}.
Then the following hold.

\begin{enumerate}
\item The Hermitian holomorphic bundle \((E|_{X\setminus D},H)\)
is acceptable in the sense of Mochizuki
\cite[Definition~2.7]{Moc03-ATHB-I} at every point of \(D\), including
every multiple intersection; i.e., on a polydisc with
\(D=\{z_1\cdots z_\ell=0\}\), its Chern curvature has bounded norm
with respect to \(H\) and
\[
 g_P=\sum_{i=1}^{\ell}
 \frac{\ii\,dz_i\wedge d\bar z_i}
 {|z_i|^2(-\log|z_i|^2)^2}
 +\sum_{\mu=\ell+1}^{n}
 \frac{\ii\,dw_\mu\wedge d\bar w_\mu}
 {(1-|w_\mu|^2)^2}.
\]

\item The growth prolongations form a locally abelian parabolic
logarithmic Higgs bundle on \(X\) and recover the original structure:
\[
 P^{\mathbf c}(H)=E^{\mathbf c}
 \qquad\text{for every real local multi-index }\mathbf c.
\]
The identifications respect the Higgs field, the filtration inclusions,
and the periodicity maps.
\end{enumerate}

\end{prop}

\begin{proof}
By Definition~\ref{defn:all-real}, the given Higgs field is logarithmic.
On \(X\setminus Z\), its degree-\(k\) characteristic-polynomial
coefficient is therefore a holomorphic section of
\(\Sym^k\Omega_X^1(\log D)\). Since this bundle is locally free and
\(\codim_X Z\ge2\), each coefficient extends uniquely across \(Z\).
Thus \((E|_{X\setminus D},\theta,H)\) is a tame harmonic bundle.
The spectral condition just verified depends only on the given Higgs
field, not on the metric or the growth hypothesis \eqref{eq:growth-off-Z}.

Mochizuki's estimates near divisor intersections bound the Higgs
commutator in the product Poincaré norm
\cite[Proposition~8.2 and Corollary~8.2]{Moc03-ATHB-I}. By harmonicity,
they also bound the Chern curvature. The estimates cover the
curvature coefficients in both normal and tangential directions at
crossings, proving acceptability with respect to \(g_P\).

The regular prolongation theorem
\cite[Proposition~2.50]{Moc06-KH2} now applies at \(\lambda=0\). Its
conclusions, expressed in the decreasing convention of
Definition~\ref{defn:growth}, give coherent locally free prolongations
forming a regular filtered sheaf. The Higgs field preserves these
prolongations logarithmically. Multiplication by \(z_i\) changes the
allowed exponent from \(c_i\) to \(c_i+1\), which proves
\(P^{\mathbf{c}+\mathbf{e}_i}(H)=P^{\mathbf{c}}(H)(-D_i)\) directly.

The compatible-filtration and multi-intersection results
\cite[Theorem~8.2 and Corollary~4.5]{Moc03-ATHB-I} give a local frame
simultaneously splitting all branch filtrations. Thus the growth
filtration is locally abelian.

To identify the growth filtration globally, let
\(j:X\setminus Z\hookrightarrow X\). Each \(P^{\mathbf{c}}(H)\) is
locally free by the preceding argument, and each \(E^{\mathbf{c}}\) is
reflexive by Definition~\ref{defn:all-real}. Since
\(\codim_X Z\ge2\), the extension property of reflexive sheaves and
\eqref{eq:growth-off-Z} give
\[
 \begin{aligned}
 P^{\mathbf{c}}(H)
 &\simeq j_*\bigl(P^{\mathbf{c}}(H)|_{X\setminus Z}\bigr)\\
 &=j_*\bigl(E^{\mathbf{c}}|_{X\setminus Z}\bigr)
 \simeq E^{\mathbf{c}}.
 \end{aligned}
\]
These canonical isomorphisms extend the identity on \(X\setminus Z\),
so they identify the two sheaves inside the common ambient sheaf
\(j_{D,*}(E|_{X\setminus D})\). Uniqueness of extension ensures
compatibility with all filtration inclusions, periodicity maps, and
Higgs maps. This proves the asserted equality at every real multi-index
and transfers the simultaneous local splitting to the original
parabolic structure.

\end{proof}

The next result allows parabolic stability to be applied to the weak
Higgs-invariant projections arising in the metric estimates. It extends
their images across the divisor and identifies their analytic degrees
with the parabolic degrees of the resulting subsheaves.

\begin{prop}
\label{prop:power-log-interface}
Let \((Y,\omega)\) be a connected compact Kähler manifold,
\(B=\bigcup_a B_a\ne\varnothing\) an SNC divisor, and
\((V_*,\theta)\) a locally abelian parabolic logarithmic Higgs bundle
of rank \(r\) on \((Y,B)\). Put \(V=V^{\mathbf{0}}\) and
\(U=Y\setminus B\).
Let \(h\) be a smooth positive Hermitian metric on \(V|_U\) satisfying
the following conditions.
\begin{enumerate}
\item The metric \(h\) is adapted to \(V_*\). On every SNC chart
\(W\), in a holomorphic frame \(e_1,\ldots,e_r\) splitting the
parabolic flags, write \(\beta_{a,i}\) for the weight of \(e_i\)
along \(z_a=0\) and \(e_{\det}=e_1\wedge\cdots\wedge e_r\). Then
\[
 (\det h)(e_{\det},e_{\det})
 =b_h\prod_a|z_a|^{2\sum_{i=1}^r\beta_{a,i}},
 \qquad b_h\in C^\infty(W,\mathbb R_{>0}).
\]

\item There is a closed analytic subset \(Z_B\subset B\), with
\(\codim_Y Z_B\ge2\), containing all intersections of the components
of \(B\), such that near each point of \(B_a\setminus Z_B\) a
flag-compatible holomorphic frame gives a model
\[
 g_a=\diag_i\left(
 |z|^{2\beta_{a,i}}\log(e/|z|)^{2\kappa_{a,i}}
 \right),\qquad \kappa_{a,i}\in\RR.
\]
Here \(B_a=\{z=0\}\), and on the punctured chart,
\[
 C_a^{-1}g_a\le h\le C_ag_a,\qquad
 \bar\partial\log(h^{-1}g_a)\in L^2
\]
for some \(C_a\ge1\).

\item The contracted Hitchin--Simpson curvature satisfies
\[
 \ii\Lambda_\omega
 \bigl(F_h+[\theta,\theta^{\dagger h}]\bigr)\in L^1(U).
\]
\end{enumerate}
All norms and integrals use \(h\), \(\omega\), and \(dV_\omega\).
Let \(p\in W^{1,2}_{\mathrm{loc}}(U,\End V)\) be an
\(h\)-orthogonal projection of constant rank \(0<\rk p<r\), satisfying
\[
 (\Id-p)\bar\partial p=0,\qquad
 (\Id-p)\theta p=0,\qquad
 \bar\partial_\theta p\in L^2(U).
\]
Then its generic image extends uniquely to a proper saturated reflexive
parabolic Higgs subsheaf \(F_*\subset V_*\), and
\[
 \deg_{\an,h}(F)=\pardeg_\omega(F_*),\qquad
 \deg_{\an,h}(V)=\pardeg_\omega(V_*).
\]
\end{prop}

\begin{proof}
The weak holomorphic projection theorem
\cite[Theorem~0.1.1]{Pop03} first defines the image sheaf on \(U\).
It is a holomorphic subbundle away from an analytic set of codimension
at least two. Put \(k=\rk p\). We must extend this rank-\(k\)
subsheaf across \(B\).

On a one-divisor chart as in condition~(2), let \(q\) be the
\(g_a\)-orthogonal projection onto the same generic image.
Put \(S=h^{-1}g_a\) and let \(R=pSp\) on \(\im p\). Uniform
comparison confines the spectra of \(S\) and \(R\) to one compact
subinterval of \((0,\infty)\), and pointwise linear algebra gives
\[
 q=R^{-1}pS.
\]
To differentiate this formula on the whole bundle, extend \(R\) to
\(\widetilde R=pSp+(\Id-p)\). This endomorphism equals \(R\) on
\(\im p\) and the identity on its orthogonal complement, so
\(q=\widetilde R^{-1}pS\), with \(\widetilde R^{\pm1}\) uniformly
bounded. The product rule and
\[
 \bar\partial\widetilde R^{-1}
 =-\widetilde R^{-1}(\bar\partial\widetilde R)\widetilde R^{-1}
\]
bound \(|\bar\partial q|\) by
\(C(|\bar\partial p|+|\bar\partial S|)\).
Writing \(S=e^b\), with \(b=\log(h^{-1}g_a)\), gives
\[
 \bar\partial S
 =\int_0^1 e^{(1-t)b}(\bar\partial b)e^{tb}\,dt.
\]
Since \(b\) is bounded, this formula gives
\(|\bar\partial S|\le C|\bar\partial b|\). Using
\(C_a^{-1}g_a\le h\le C_ag_a\) to compare the norms, we obtain
\[
 |\bar\partial q|_{g_a}
 \le C\bigl(
 |\bar\partial p|_h+
 |\bar\partial\log(h^{-1}g_a)|
 \bigr).
 \tag{6.2}\label{eq:projection-transfer-bound}
\]
Thus \(\bar\partial q\) has finite \(L^2\) norm for \(g_a\).
Fix a tangential parameter for which the normal disk avoids the
non-subbundle locus and the normal derivative has finite energy; this
holds for almost every parameter by Fubini. Write
\(\Phi:\Delta^*\to\operatorname{Gr}(k,r)\) for the resulting
holomorphic map. Its energy for the fixed Grassmannian metric is not
yet controlled, since \(g_a\) depends on \(z\).

Choose a nowhere-zero holomorphic Plücker lift
\(f=(f_I)\) on \(\Delta^*\), which is possible since holomorphic
line bundles on the punctured disk are trivial. Put
\[
 v_g=\log|f|_{\wedge^k g_a}^2,\qquad
 v_0=\log\sum_I|f_I|^2.
\]
For the Euclidean disk Laplacian \(\Delta_0\), the ordinary
Gauss--Codazzi formula for the induced determinant metric gives the
following estimate, with Euclidean area and form norms on the disk:
\[
 |\Delta_0v_g|
 \le C\left(|F_{g_a}|_{g_a}+|\bar\partial q|_{g_a}^2\right).
\]
The first term is integrable: the power factors have zero curvature on
\(\Delta^*\), and the logarithmic factors contribute at most
\(C/(|z|^2\log(e/|z|)^2)\). Hence \(\Delta_0v_g\in L^1\).
If \(M_g(t)\) is the circular average of \(v_g\) at radius
\(e^{-t}\), polar integration gives
\(\int_{t_0}^{\infty}|M_g''(t)|\,dt<\infty\), and therefore
\(|M_g(t)|\le C(1+t)\).

The finitely many power and logarithmic factors also give
\(|v_g-v_0|\le C(1+\log(1/|z|))\). Thus the circular average
\(M_0(t)\) of \(v_0\) is bounded above by \(C(1+t)\).
Since \(v_0\) is subharmonic, \(M_0\) is convex. Its derivative is
therefore bounded above, and
\(\int_{t_0}^{\infty}M_0''(t)\,dt<\infty\).
Up to a fixed normalization, this last integral is
\(\int_{0<|z|<e^{-t_0}}\Phi^*\omega_{\mathrm{FS}}\), where
\(\omega_{\mathrm{FS}}\) is the Fubini--Study form of the Plücker
embedding. Consequently \(\Phi\) has finite energy for a fixed
compact Kähler target. A holomorphic map from a Riemann surface to a
Kähler manifold is harmonic, so the removable-singularity theorem
\cite[Theorem~3.6]{SacksUhlenbeck1981} extends \(\Phi\)
holomorphically across the puncture.

The Plücker coordinate ratios are thus meromorphic on almost every
normal disk; they are already meromorphic in the tangential variables
off the divisor. Shiffman's separate-meromorphicity theorem, in the
form recalled in the final step of the proof of
\cite[Theorem~0.1.1]{Pop03}, extends them meromorphically across each
one-divisor chart. Meromorphic Hartogs extension across \(Z_B\) then
extends the generic subspace across all of \(B\). The Plücker
relations persist by the identity theorem.

Clear local meromorphic denominators and saturate the resulting coherent
image in \(V\), obtaining \(F\subset V\) with torsion-free quotient
\(Q=V/F\). Since \(V\) is locally free, the depth lemma applied to
\[
 0\longrightarrow F\longrightarrow V\longrightarrow Q\longrightarrow0
\]
shows that \(F\) satisfies Serre's condition \(S_2\). Being torsion-free,
it is therefore reflexive. The weak Higgs equation says that the induced
morphism \(F\to Q\otimes\Omega_Y^1(\log B)\) vanishes on the common
subbundle locus. Its target is torsion-free, so it vanishes everywhere,
and \(F\) is Higgs invariant. Saturated intersections with the ambient
parabolic flags, followed by reflexive extension and integral
periodicity, give \(F_*\subset V_*\) at every real multi-index.
Saturation determines \(F\) uniquely from its generic image, and the
induced parabolic structure is consequently unique.

On the common subbundle locus, Higgs Gauss--Codazzi gives
\[
 \ii\Tr\bigl(p\Lambda_\omega\Psi_h\bigr)
 -|\bar\partial_\theta p|^2
 =
 \ii\Lambda_\omega F_{\det(F,h)}.
\]
The left side is \(L^1\). Write \(\ell_h=\det(F,h)\), and choose
\(q_F=k_F\prod_a|\sigma_a|_{k_a}^{2\gamma_{F,a}}\), where
\(k_F\) is a smooth metric on \((\bigwedge^kF)^{**}\), the
\(k_a\) are smooth divisor metrics, and \(\gamma_{F,a}\) is the
sum of the induced parabolic weights, counted with multiplicity.
The curvature integral of \(q_F\) on \(U\) is the parabolic
first-Chern pairing. We show that replacing \(q_F\) by \(\ell_h\)
does not change this integral.

On a split SNC chart, let
\[
 d(e_i,e_j)=\delta_{ij}\prod_a|z_a|^{2\beta_{a,i}},
 \qquad R_B=\sum_a-\log|z_a|^2.
\]
Adaptedness bounds the norm of each frame vector with every
positive tolerance in the growth exponent. The determinant condition then bounds the inverse
Gram matrix, so, for every \(\epsilon>0\),
\[
 C_\epsilon^{-1}e^{-\epsilon R_B}d
 \le h\le C_\epsilon e^{\epsilon R_B}d.
\]
Write \(\ell_d=\det(F,d)\). In a local generator of
\((\bigwedge^kF)^{**}\), its coefficient is
\[
 \ell_d=\sum_I|f_I|^2
       \prod_a|z_a|^{2\sum_{i\in I}\beta_{a,i}},
\]
with holomorphic Plücker coefficients \(f_I\) having no common
divisorial factor. Its logarithm is plurisubharmonic and extends across
the divisor. At a generic point of \(B_a\), the induced determinant
filtration gives
\[
 \gamma_{F,a}
 =\min_{I:f_I\not\equiv0}
   \left(\operatorname{ord}_{B_a}(f_I)
          +\sum_{i\in I}\beta_{a,i}\right).
\]
This can be computed in a basis compatible with the induced filtration
over the local discrete valuation ring. It is also the coefficient of
\(\log|z_a|^2\) in the logarithm of the displayed norm. Removing
these divisorial terms
leaves a positive closed \((1,1)\)-current with no divisorial mass.
It has no mass on the codimension-at-least-two common-zero locus
either, by the support theorem for positive closed currents; see
\cite[Chapter~III, Corollary~2.4, Theorem~2.10 and
Section~8]{DemaillyBook}. Away from those analytic sets its coefficients
are smooth. Thus the contracted curvature of \(\ell_d\) on the
subbundle locus has an \(L^1\) extension, and
\(\log(\ell_d/q_F)\) satisfies the corresponding curvature equation
distributionally, with no additional term on the omitted sets.
This logarithm is locally \(L^2\): the displayed finite sum is bounded
above, while any nonzero summand gives a lower bound by a finite sum
of logarithms of absolute values of holomorphic functions and of
divisor coordinates, after the prescribed powers are removed.

Put \(v=\log(\ell_h/\ell_d)\). The metric comparison gives
\[
 |v|\le C_\epsilon+\epsilon R_B
 \quad\text{for every }\epsilon>0.
\]
Its contracted curvature equation has an \(L^1\) right side off the
divisor. It extends across the non-subbundle locus there: for a normal
cutoff of radius \(\delta\) along a stratum of complex codimension
\(q\ge2\), testing the scalar equation produces errors bounded by
\(C(\delta^{q-1}+\delta^{q-2})\|v\|_{L^2(A_\delta)}\), where
\(A_\delta\) is the shrinking cutoff region. These tend to zero by
absolute continuity of the \(L^2\) integral.

To pass across the divisor, take
\[
 \chi_T=\prod_a\zeta\left(\frac{-\log|z_a|^2}{T}\right),
\]
where \(\zeta=1\) on \(( -\infty,1]\), \(\zeta=0\) on
\([2,\infty)\), and \(0\le\zeta\le1\). For the scalar
Laplacian \(\mathcal L=d_\omega^*d\) on the chart,
\[
 \|d\chi_T\|_{L^1}+\|\mathcal L\chi_T\|_{L^1}=O(T^{-1}).
\]
Indeed, \(\partial\bar\partial\log|z_a|^2=0\) off the divisor,
and the remaining normal second-derivative term has integral
\(O(T^{-2}\int_{e^{-T}}^{e^{-T/2}}dr/r)=O(T^{-1})\);
the mixed terms are smaller. On the supports of the derivatives of
\(\chi_T\), every contributing divisor depth is at most \(2T\).
For a fixed compactly supported test function, the cutoff errors in
the distributional equation for \(v\) are therefore bounded by
\(C(C_\epsilon/T+\epsilon)\). First letting \(T\to\infty\)
and then \(\epsilon\to0\) removes them. The \(L^1\) source
passes to the limit by dominated convergence.

Consequently \(\phi_F=\log(\ell_h/q_F)\) is locally \(L^2\)
and satisfies its contracted curvature equation on all of \(Y\),
with no distribution supported on the divisor or the non-subbundle
locus. Testing this global equation against one shows that the
curvature integrals of \(\ell_h\) and \(q_F\) agree. This proves
the degree identity. For the ambient identity, the determinant
condition makes the ratio of \(\det h\) to the analogous model
\(q_V\) smooth and strictly positive on \(Y\); compact Stokes
therefore gives the same conclusion directly.
\end{proof}

\section[Compactness and parabolic growth]{Compactness and parabolic growth\texorpdfstring{\\}{ }of Hermitian--Einstein metrics}\label{sec:compactness}

This section establishes the compactness and boundary growth estimates
and uses them to prove Theorem~\ref{thm:harmonic-main}. Under its numerical
vanishing hypotheses, Corollary~\ref{cor:factor-energy} shows that the
Hitchin--Simpson curvature energies of the Hermitian--Einstein
metrics \(h_\nu\) tend to zero. Two issues remain: obtaining a smooth positive
limit of these metrics on the divisor complement, and controlling its
growth near the divisor so that the original parabolic structure is
preserved.

For each original stable summand \(E_*\),
the modification \(\pi:Y\to X\) identifies the divisor complements
\(U_Y=Y\setminus B\) and \(U_X=X\setminus D\), which we denote by \(U\).

We work with the global solutions from
Theorem~\ref{thm:fixed-stage-he-metrics}, normalized by multiplication
by positive constants.
The key step in Proposition~\ref{prop:moving-endpoints} is to use stability
to prove a two-sided comparison, uniform in \(\nu\) on every compact
subset, between these solutions and the auxiliary metrics \(k_\nu\)
constructed in Proposition~\ref{prop:determinant-selector-limit}.
The metrics \(k_\nu\) and the K\"ahler forms \(\omega_\nu\) converge smoothly
on compact subsets of the divisor complement. Interior elliptic
estimates then give uniform
bounds for all derivatives on smaller compact subsets, and a diagonal
subsequence converges smoothly. The lower comparison bound ensures
that the limiting metric is positive.

Corollary~\ref{cor:nonexceptional-growth} provides the additional
two-sided growth estimates near the divisor needed to recover the
original parabolic filtration. In the final subsection,
Subsection~\ref{subsec:adapted-pluriharmonic-metrics},
energy decay makes the limit pluriharmonic, while these growth estimates
and Proposition~\ref{prop:acceptable-prolongation} identify all its growth prolongations
with the original parabolic sheaf. Applying this argument to each
stable summand completes the proof of Theorem~\ref{thm:harmonic-main}.

\subsection{Comparison metrics and determinant normalization}
\label{subsec:moving-comparison-metrics}
Using the K\"ahler forms \(\omega_\nu\) from
Corollary~\ref{cor:stable-rational-approximations}, we construct the
auxiliary comparison metrics, with a determinant normalization chosen
to ensure smooth convergence on compact subsets. These metrics retain
the prescribed parabolic lattices at each rational stage.

\begin{setup}
\label{setup:literal-moving-schedule}
Fix the finite SNC coordinate atlas \(\{U_\lambda\}\) for \((Y,B)\)
and the local frames from Subsection~\ref{subsec:adapted-references},
independently of \(\nu\).
Retain the rational stages, parameters, and divisor norms from
Corollary~\ref{cor:stable-rational-approximations}, together with the
K\"ahler forms \(\omega_\nu\) defined by
\eqref{eq:kahler-schedule-form}.

For every component \(B_b\) of \(B\), write
\[
 s_b=|\sigma_b|_{k_b}^2\le1
\]
for the previously chosen squared norm of its canonical section.

On an atlas chart \(U_\lambda\), write \(G_{\lambda,\nu}\) for the coefficient matrix of \(\omega_\nu\) in its fixed coordinate coframe.
Let \(D_\lambda\subset\{1,\ldots,n\}\) index the coordinate
hypersurfaces defining \(B\) on \(U_\lambda\). Rescale the divisor
coordinates so that \(|z_{\lambda,i}|<e^{-1}\) for \(i\in D_\lambda\).
On \(U_\lambda\setminus B\), put
\[
 \rho_{\lambda,i}=|z_{\lambda,i}|^2,\qquad
 \tau_{\lambda,i}=1-\log\rho_{\lambda,i}
 \quad(i\in D_\lambda),
\]
\[
 h_{\lambda,i}=
 \begin{cases}
 (\sqrt{\rho_{\lambda,i}}\tau_{\lambda,i})^{-1},
      &i\in D_\lambda,\\
 1,&i\notin D_\lambda.
\end{cases}
\]
On an interior chart, \(D_\lambda=\varnothing\) and all
\(h_{\lambda,i}=1\).
If \(z_{\lambda,i}=0\) defines the component \(B_b\) on this chart, then
\[
 s_b|_{U_\lambda}
 =u_{\lambda,b}\rho_{\lambda,i}
\]
for a smooth positive function \(u_{\lambda,b}\) whose value and reciprocal,
together with all derivatives needed below, are bounded on the shrunken
chart.
Define the diagonal and mixed source terms by
\[
 \Xi_{\lambda,i,\nu}
 =h_{\lambda,i}^2(G_{\lambda,\nu}^{-1})^{i\bar i},
 \qquad
 \Xi_{\lambda,ik,\nu}
 =h_{\lambda,i}h_{\lambda,k}
  \left|(G_{\lambda,\nu}^{-1})^{i\bar k}\right|.
 \tag{7.1}\label{eq:exact-source-rows}
\]
Positivity gives
\[
 2\Xi_{\lambda,ik,\nu}
 \le \Xi_{\lambda,i,\nu}+\Xi_{\lambda,k,\nu}.
 \tag{7.2}\label{eq:exact-mixed-row}
\]
The terms in \eqref{eq:exact-source-rows} arise when curvature coefficients
bounded by \(C h_{\lambda,i}h_{\lambda,k}\) are contracted using
\(G_{\lambda,\nu}^{-1}\).
\end{setup}

When two refined flag weights coalesce, their difference may tend to
zero, so the corresponding factor \(|z|^{\beta_i-\beta_j}\) gives
no uniform rate of decay. To obtain estimates uniform in \(\nu\),
we insert logarithmic factors whose exponents depend on the fixed flag
positions.

Fix one summand and retain the SNC coordinate atlas and notation
from Setup~\ref{setup:literal-moving-schedule}.
On a divisor chart \(U_\lambda\), let
\(e_{\lambda,1},\ldots,e_{\lambda,r}\) be the holomorphic frame of
\(V|_{U_\lambda}\) chosen in Subsection~\ref{subsec:adapted-references}.
It simultaneously splits the fixed parabolic flags, and its induced
frames on the graded quotients respect the residue eigenspace
decompositions.
For \(a\in D_\lambda\), use the fixed flag
\(L_{a,0}\supsetneq\cdots\supsetneq L_{a,m_a}\) of
Subsection~\ref{subsec:fixed-flag-weights}, and set
\(p_{\lambda,a,i}=q\) when \(e_{\lambda,i}\) represents a vector
in \(L_{a,q-1}/L_{a,q}\). Thus these integer labels increase in
the flag order and do not change when weights coalesce. Put
\[
 \bar p_{\lambda,a}=\frac1r\sum_i p_{\lambda,a,i}.
\]
For the stage-\(\nu\) weights \(\beta_{\lambda,\nu,a,i}\), define
\[
 g_{\lambda,\nu}(e_{\lambda,i},e_{\lambda,k})
 =\delta_{ik}\prod_a|z_{\lambda,a}|^{2\beta_{\lambda,\nu,a,i}}
 \prod_a\tau_{\lambda,a}^{-4(p_{\lambda,a,i}-\bar p_{\lambda,a})}.
\]
Let \(q_\nu\) be the determinant metric constructed at stage \(\nu\)
in Subsection~\ref{subsec:adapted-references}. Its coefficient in the
frame \(e_{\lambda,\det}=e_{\lambda,1}\wedge\cdots\wedge
e_{\lambda,r}\) is
\[
 q_\nu(e_{\lambda,\det},e_{\lambda,\det})
 =b_{\lambda,\nu}\prod_a
 |z_{\lambda,a}|^{2\sum_i\beta_{\lambda,\nu,a,i}},
 \qquad b_{\lambda,\nu}>0,
\]
where \(b_{\lambda,\nu}\) extends smoothly and strictly positively across
the whole chart.
Complete the divisor charts to a finite cover by adding interior charts.
On each interior chart choose a smooth positive Hermitian metric
\(g_{\lambda,\nu}=g_\lambda\), independent of \(\nu\).
On every chart, including the interior charts, normalize the local
metric by the same formula:
\[
 \widetilde g_{\lambda,\nu}
 =\left(\frac{q_\nu}{\det g_{\lambda,\nu}}\right)^{1/r}
  g_{\lambda,\nu}.
\]
Let \(\{\chi_\lambda\}\) be a fixed subordinate partition of unity
for this finite cover, and put
\[
 g_\nu^{\rm raw}=\sum_\lambda\chi_\lambda\widetilde g_{\lambda,\nu},
 \qquad
 \widehat g_\nu
 =\left(\frac{q_\nu}{\det g_\nu^{\rm raw}}\right)^{1/r}
  g_\nu^{\rm raw}.
\]
To estimate the effect of gluing, compare the sum with each local
metric. On \(\operatorname{supp}\chi_\lambda\setminus B\),
let \(H_{\lambda,\nu}\) and \(H_\nu^{\rm raw}\) be the matrices of
\(\widetilde g_{\lambda,\nu}\) and \(g_\nu^{\rm raw}\), respectively,
in the same holomorphic frame, and set
\[
 S_{\lambda,\nu}=H_{\lambda,\nu}^{-1}H_\nu^{\rm raw}.
\]
This is a positive endomorphism, self-adjoint with respect to
\(\widetilde g_{\lambda,\nu}\).
Let \(D\) be the connection on \(\operatorname{End}(V)\) induced by
the Chern connection of \(\widetilde g_{\lambda,\nu}\).
The indices \(A,B\in\{1,\ldots,n\}\) range over all coordinate
directions, and \(D_A\) and \(D_{\bar B}\) denote covariant
differentiation along \(\partial/\partial z_{\lambda,A}\) and
\(\partial/\partial\bar z_{\lambda,B}\), respectively.
Write \(h_A=h_{\lambda,A}\), as defined in
Setup~\ref{setup:literal-moving-schedule}, and take the endomorphism
norms with respect to \(\widetilde g_{\lambda,\nu}\).

\begin{lem}
\label{lem:doubled-reference}
The metrics just constructed have the following properties, with
constants independent of \(\nu\) and of the chart.
\begin{enumerate}
\item \(\det\widehat g_\nu=q_\nu\), and \(\widehat g_\nu\) realizes
the parabolic lattices of stage \(\nu\) at every real multi-index.

\item On \((\operatorname{supp}\chi_\lambda\cap
\operatorname{supp}\chi_\mu)\setminus B\), the local metrics
\(\widetilde g_{\lambda,\nu}\) and \(\widetilde g_{\mu,\nu}\) are
uniformly comparable. On \(\operatorname{supp}\chi_\lambda\setminus B\),
\[
 \|S_{\lambda,\nu}^{\pm1}\|\le C,\qquad
 \|D_AS_{\lambda,\nu}\|+\|D_AS_{\lambda,\nu}^{-1}\|\le Ch_A,
\]
\[
 \|D_{\bar B}D_AS_{\lambda,\nu}\|
 +\|D_{\bar B}D_AS_{\lambda,\nu}^{-1}\|
 \le Ch_Ah_B.
\]

\item Put \(\Psi_{\widehat g_\nu}=F_{\widehat g_\nu}
+[\theta,\theta^{\dagger\widehat g_\nu}]\).
In the fixed coordinate coframe of each chart, its trace-free
endomorphism-valued coefficients satisfy
\[
 \left|
 (\Psi_{\widehat g_\nu}^{0})_{A\bar B}
 \right|_{\op,\widehat g_\nu}
 \le Ch_Ah_B.
 \tag{7.3}\label{eq:doubled-tensor-row}
\]
Here \(|\cdot|_{\op,\widehat g_\nu}\) is the endomorphism operator
norm induced by \(\widehat g_\nu\).
\end{enumerate}
\end{lem}

\begin{proof}
The centered exponents have zero sum:
\[
 \sum_i(p_{\lambda,a,i}-\bar p_{\lambda,a})=0.
\]
Thus the logarithmic factors cancel in the determinant.
All normalized local metrics have determinant \(q_\nu\). This common scalar
factor cancels from their relative endomorphisms and from the induced
connections on \(\End V\); their trace-free Chern curvatures are
also independent of it.

\noindent\textit{Comparison and Chern curvature.}
Fix a chart and suppress its index. Write \(r_a=|z_a|\) and
\(p_{a,i}=p_{\lambda,a,i}\). A matrix entry \(t_{ij}\) of a
flag-preserving holomorphic map, from frame line \(j\) to line
\(i\), has normalized size, up to bounded factors from scalar
normalization and changes of divisor coordinates,
\[
 |t_{ij}|\prod_a r_a^{\beta_{\nu,a,i}-\beta_{\nu,a,j}}
                   \tau_a^{-2(p_{a,i}-p_{a,j})}.
\]
If \(p_{a,i}<p_{a,j}\), preservation of the decreasing flag makes
\(t_{ij}\) divisible by \(z_a\). Since the finitely many weights
converge in \([0,1)\), there is \(\epsilon_*>0\), independent of
\(\nu\), such that
\[
 1+\beta_{\nu,a,i}-\beta_{\nu,a,j}\ge\epsilon_*.
\]
The resulting positive radial power absorbs every fixed logarithmic
power. If \(p_{a,i}>p_{a,j}\), the radial factor is at most one
and the logarithmic factor is at most \(\tau_a^{-2}\).
For equal labels both factors are one. Divisibility by several
distinct coordinate functions gives divisibility by their product,
so these estimates hold simultaneously at every crossing.

Apply this calculation to the holomorphic transition
\(T_{\mu\lambda}\) from the \(\lambda\)-frame to the
\(\mu\)-frame and its inverse, and let \(\nabla\) be the Hom
connection induced by the two local metrics. It gives uniform bounds for
\(T_{\mu\lambda}^{\pm1}\). For its first covariant derivative,
an entry divisible by \(z_a^k\) has effective radial exponent
\(\delta=k+\beta_{\nu,a,i}-\beta_{\nu,a,j}\ge0\).
The radial connection term has coefficient \(\delta/z_a\), while
the logarithmic connection term is bounded by \(C/(r_a\tau_a)\).
Use
\[
 \sup_{\delta\ge0}\delta r_a^\delta\le C\tau_a^{-1}
\]
when the logarithmic power is nonpositive; when it is positive,
\(\delta\ge\epsilon_*\) absorbs it. The remaining holomorphic
coefficients have bounded derivatives. For a coordinate change
\(z'_a=u_a z_a\),
\[
 \tau'_a=\tau_a-\log|u_a|^2,\qquad
 \partial\tau'_a=-\frac{dz_a}{z_a}-\partial\log|u_a|^2.
\]
The ratios of the corresponding logarithmic factors are bounded,
and their logarithmic derivatives satisfy the same \(Ch_A\)
bound. The scalar ratios remaining after cancellation of \(q_\nu\)
also contribute bounded terms. Thus
\(\|\nabla_A T_{\mu\lambda}\|\le Ch_A\).

On a divisor chart the trace-free curvature of the local diagonal
model comes from its logarithmic factors and satisfies the
coefficient bound \(Ch_Ah_B\), also after the coordinate changes
just described. On an interior chart the trace-free curvature
is that of a fixed smooth metric and is bounded. The central parts
of the Chern curvatures of the normalized local metrics are all
\(F_{q_\nu}/r\), so they cancel in
\[
 \nabla''\nabla'T_{\mu\lambda}
 =F_{\widetilde g_{\mu,\nu}}T_{\mu\lambda}
  -T_{\mu\lambda}F_{\widetilde g_{\lambda,\nu}}.
\]
This bounds the mixed second derivatives by \(Ch_Ah_B\).
For \(P_{\mu\lambda}=T_{\mu\lambda}^\dagger T_{\mu\lambda}\),
metric compatibility and holomorphicity give
\[
 D'_A P_{\mu\lambda}=T_{\mu\lambda}^\dagger
                         \nabla_A T_{\mu\lambda},
\]
\[
 D_{\bar B}D'_A P_{\mu\lambda}
 =(\nabla_B T_{\mu\lambda})^\dagger\nabla_A T_{\mu\lambda}
  +T_{\mu\lambda}^\dagger\nabla_{\bar B}\nabla_A T_{\mu\lambda}.
\]
Now \(S_{\lambda,\nu}=\sum_\mu\chi_\mu P_{\mu\lambda}\).
The product rule and differentiation of
\(S_{\lambda,\nu}^{-1}S_{\lambda,\nu}=\Id\) prove all the
assertions in \textup{(2)}. In particular, the connection change formula
\[
 F_{g_\nu^{\rm raw}}-F_{\widetilde g_{\lambda,\nu}}
 =\bar\partial\bigl(S_{\lambda,\nu}^{-1}
   D'_{\widetilde g_{\lambda,\nu}}S_{\lambda,\nu}\bigr)
\]
gives the same coefficient bound for \(F_{g_\nu^{\rm raw}}^0\).

\noindent\textit{Higgs curvature.}
Write
\[
 \theta=\sum_{a\in D_\lambda}A_a\frac{dz_a}{z_a}
        +\sum_{\mu\notin D_\lambda}B_\mu\,dz_\mu.
\]
The coefficient matrices are holomorphic and preserve every fixed
divisor flag. Let \(\Sigma_a\) be the diagonal matrix of the
eigenvalues of the residue induced on its graded quotients, with
bounded holomorphic extensions of its diagonal entries off the
divisor. It is a representative of the graded residue, not
necessarily the matrix of the residue on the entire fiber.
Within each flag quotient, \(A_a-\Sigma_a\) vanishes on
\(z_a=0\). Between different quotients the preceding entry
calculation applies. In the local metric it therefore gives
\[
 \|A_a-\Sigma_a\|\le C\tau_a^{-2},\qquad
 \|B_\mu\|\le C.
\]

We must also control the change of adjoint under gluing. For an
overlapping chart, use orthonormal coordinates for the two local
metrics, and denote the normalized transition by \(T\) and the
two diagonal graded residue matrices by \(\Sigma_a\) and
\(\Sigma'_a\). Put
\[
 E_a=\Sigma'_aT-T\Sigma_a,\qquad
 E'_a=(\Sigma'_a)^\dagger T-T\Sigma_a^\dagger.
\]
On a common flag quotient, the underlying holomorphic transition
intertwines the graded residues along the divisor. Thus the
corresponding entries of both differences vanish there before
normalization: the first holomorphically and the second smoothly
with size \(O(r_a)\). Between different quotients, the underlying
holomorphic entry is divisible by \(z_a\) when it maps a higher flag
label to a lower one, while in the other direction its normalized size is
bounded by \(C\tau_a^{-2}\). The preceding product estimates give
\[
 \|E_a\|+\|E'_a\|\le C\tau_a^{-2}.
\]
Here and below the bounds hold in operator norm, uniformly in the
other divisor coordinates. The identity
\[
 [\Sigma_a,T^\dagger T]=T^\dagger E_a-(E'_a)^\dagger T
\]
and its adjoint imply, after summing with the partition,
\[
 \|[\Sigma_a,S_{\lambda,\nu}]\|
 +\|[\Sigma_a^\dagger,S_{\lambda,\nu}]\|
 \le C\tau_a^{-2}.
\]
Since
\(A_a^{\dagger g_\nu^{\rm raw}}
 =S_{\lambda,\nu}^{-1}A_a^{\dagger\widetilde g_{\lambda,\nu}}
   S_{\lambda,\nu}\), this gives
\[
 \|[A_a,A_a^{\dagger g_\nu^{\rm raw}}]\|
 \le C\tau_a^{-2}.
\]
The regular coefficients remain bounded under the uniform metric
comparison.

For the mixed terms use Higgs integrability. If two endomorphisms
\(A,B\) commute, cyclicity of trace gives
\[
 \|[A,B^\dagger]\|_{\rm HS}^2
 =\Tr\bigl([A,A^\dagger][B,B^\dagger]\bigr)
 \le\|[A,A^\dagger]\|_{\rm HS}
      \|[B,B^\dagger]\|_{\rm HS}.
\]
Here the adjoints and Hilbert--Schmidt norms use the same
Hermitian metric. All Higgs coefficient matrices commute because
\(\theta\wedge\theta=0\). Applying this identity with the
metric \(g_\nu^{\rm raw}\) gives
\[
 \|[A_a,A_b^\dagger]\|\le C(\tau_a\tau_b)^{-1},\qquad
 \|[A_a,B_\mu^\dagger]\|\le C\tau_a^{-1},\qquad
 \|[B_\mu,B_\kappa^\dagger]\|\le C.
\]
Dividing by the corresponding factors \(z_a\) and \(\bar z_b\)
therefore bounds every ordinary-coordinate coefficient of
\([\theta,\theta^{\dagger g_\nu^{\rm raw}}]\) by \(Ch_Ah_B\).
Together with the Chern curvature estimate this proves the same
bound for \(\Psi_{g_\nu^{\rm raw}}^0\). The final determinant
normalization changes only central Chern curvature and leaves Higgs
adjoints unchanged, proving \eqref{eq:doubled-tensor-row}.

Every fixed power of \(\tau_a\) grows more slowly than any positive
power of \(|z_a|^{-1}\). Hence the local metrics have the same growth
lattices as their split monomial factors, at every real index.
Uniform comparison preserves these lattices under the partition-of-unity
sum and determinant normalization. This proves the growth assertion.
\end{proof}

We next fix the multiplicative constants in the determinant metrics
used in Lemma~\ref{lem:doubled-reference}, and use them to normalize
the comparison metrics \(\widehat g_\nu\). Both sequences will then
converge smoothly on compact subsets of \(U\). The local formula for
the limiting determinant metric on \(X\) will be used in the degree proof for \(k_\infty\)
in Proposition~\ref{prop:downstairs-reflexive-degree-bridge}.

\begin{prop}
\label{prop:determinant-selector-limit}
Retain the fixed model \(\pi:Y\to X\) and the rational stages of
Setup~\ref{setup:literal-moving-schedule}, and identify
\(U=Y\setminus B\simeq X\setminus D\). Write \(\widetilde D_i\)
and \(A_a\) for the strict-transform and exceptional components of
\(B\), respectively, and put
\[
 L_X=(\bigwedge^rE)^{**},\qquad L_Y=\det V,\qquad
 L=L_Y|_U\simeq L_X|_U.
\]
Let \(\gamma_{i,\nu}\) and \(\eta_{a,\nu}\) be the sums, counted
with multiplicity, of the stage-\(\nu\) parabolic weights along
\(\widetilde D_i\) and \(A_a\), and let \(\gamma_{i,\infty}\) be
the corresponding sums of the original weights along \(D_i\).

Let \(q_\nu\) be the stage-\(\nu\) determinant metric from
Lemma~\ref{lem:doubled-reference}. There exist positive constants
\(a_\nu\) and a Hermitian metric \(Q_{X,\infty}\) on \(L\) such that
the rescaled metrics
\[
 Q_{Y,\nu}:=a_\nu q_\nu
\]
satisfy the following properties:

\begin{enumerate}
\item In a local holomorphic frame of \(L_Y\), the coefficient of
\(Q_{Y,\nu}\) is a smooth strictly positive function on the chart
multiplied by the factors \(|z_i|^{2\gamma_{i,\nu}}\) and
\(|u_a|^{2\eta_{a,\nu}}\), where \(z_i=0\) and \(u_a=0\) define
the components \(\widetilde D_i\) and \(A_a\) meeting the chart.
Moreover, for real constants \(c_{\det,\nu}\),
\[
 \ii\Lambda_{\omega_\nu}F_{Q_{Y,\nu}}=2\pi c_{\det,\nu}.
\]

\item For every \(\mathcal C\Subset U\),
\[
 Q_{Y,\nu}\longrightarrow Q_{X,\infty}
 \quad\text{in }C^\infty(\mathcal C).
 \tag{7.4}\label{eq:selector-common-open-limit}
\]
On every SNC coordinate chart of \(X\), in a local holomorphic
frame \(\tau_X\) of \(L_X\),
\[
 Q_{X,\infty}(\tau_X,\tau_X)
 =b_\infty\prod_i|z_i|^{2\gamma_{i,\infty}},
\]
where \(z_i=0\) are the local components of \(D\) and
\(b_\infty\) extends smoothly and strictly positively across the
whole chart, including the modification center \(Z\subset D\).
The constants and the limiting metric satisfy
\[
 c_{\det,\nu}\longrightarrow\lambda_{\det,\infty},\qquad
 \ii\Lambda_\omega F_{Q_{X,\infty}}
 =2\pi\lambda_{\det,\infty}.
\]
\item Let \(\widehat g_\nu\) be the metrics of
Lemma~\ref{lem:doubled-reference}, and define
\[
 k_\nu=
 \left(\frac{Q_{Y,\nu}}{\det\widehat g_\nu}\right)^{1/r}\widehat g_\nu
 =a_\nu^{1/r}\widehat g_\nu.
\]
These metrics realize the parabolic lattices of stage \(\nu\) at
every real multi-index and satisfy \(\det k_\nu=Q_{Y,\nu}\).
There is a smooth positive Hermitian metric \(k_\infty\) on \(V|_U\)
such that
\[
 k_\nu\longrightarrow k_\infty
 \quad\text{in }C^\infty_{\mathrm{loc}}(U),\qquad
 \det k_\infty=Q_{X,\infty}.
\]
For each \(\nu\), \(k_\nu\) and \(\widehat g_\nu\) have the same
Chern connection, Higgs adjoint, and induced norms on \(\End V\).
\end{enumerate}

\end{prop}

\begin{proof}
We construct metrics with the stated limit, then show that their
ratios to \(q_\nu\) are constant on \(Y\). Finally, we prove the
convergence of the resulting comparison metrics \(k_\nu\).

\noindent\textbf{Step 1. The determinant lines and the metrics on
\(X\).}
Since \(E\) is reflexive on the smooth manifold \(X\), \(L_X\) is a
line bundle. We first compare its pullback with \(L_Y\) to determine
the exceptional corrections needed in constructing \(Q_{Y,\nu}\).
By Theorem~\ref{thm:fixed-terminal-data}, the model is unchanged
outside \(Z=\operatorname{Sing}(E)\cup Z_{\mathrm{flag}}\subset D\).
The resulting identification of the determinant lines gives a
meromorphic section of \(L_Y\otimes(\pi^*L_X)^{-1}\) whose divisor
is supported on the exceptional components \(A_a\). Hence, for
uniquely determined integers \(t_a\),
\[
 L_Y\simeq\pi^*L_X\otimes
 \cO_Y\left(\sum_at_aA_a\right).
 \tag{7.5}\label{eq:determinant-line-discrepancy}
\]
Write
\[
 \pi^*D_i=\widetilde D_i+\sum_am_{ai}A_a,
\]
and set
\[
 \rho_{a,\nu}=\eta_{a,\nu}-\sum_im_{ai}\gamma_{i,\nu},
 \qquad e_{a,\nu}=t_a+\rho_{a,\nu}.
 \tag{7.6}\label{eq:selector-exceptional-discrepancy}
\]
These sequences are bounded because the weights converge.

Choose smooth positive metrics \(k_X\) on \(L_X\) and \(k_i\) on
\(\cO_X(D_i)\). With \(\sigma_i\) denoting the canonical section, put
\[
 q_{X,\mathrm{raw},\nu}
 =k_X\prod_i|\sigma_i|_{k_i}^{2\gamma_{i,\nu}}.
\]
The curvature forms
\[
 \alpha_{X,\mathrm{raw},\nu}
 =\frac{\ii}{2\pi}F_{q_{X,\mathrm{raw},\nu}}
\]
extend smoothly across \(X\), and these extended forms converge
smoothly because the weights converge. Define the scalar operator
\[
 P_Xf=\frac{\ii}{2\pi}\Lambda_\omega\bar\partial\partial f.
\]
There is a unique mean-zero solution of
\[
 P_Xf_{X,\nu}
 =\lambda_{X,\nu}
 -\Lambda_\omega\alpha_{X,\mathrm{raw},\nu},
\]
where \(\lambda_{X,\nu}\) is the \(\omega\)-volume average of
\(\Lambda_\omega\alpha_{X,\mathrm{raw},\nu}\). Elliptic estimates
for this fixed operator on compact \(X\) give
\[
 f_{X,\nu}\longrightarrow f_{X,\infty}
 \quad\text{in }C^\infty(X).
\]
Define
\[
 Q_{X,\nu}=e^{f_{X,\nu}}q_{X,\mathrm{raw},\nu},\qquad
 Q_{X,\infty}=e^{f_{X,\infty}}k_X
       \prod_i|\sigma_i|_{k_i}^{2\gamma_{i,\infty}}.
\]
These metrics converge smoothly on compact subsets of \(U\).
Since \(f_{X,\infty}\) is smooth on all of \(X\), this also proves
the asserted local form and curvature equation for
\(Q_{X,\infty}\), with
\(\lambda_{\det,\infty}=\lim_\nu\lambda_{X,\nu}\).

\noindent\textbf{Step 2. Adjust the equation to \(\omega_\nu\).}
The metric \(Q_{X,\nu}\) solves the scalar equation for \(\omega\),
but we need the equation for \(\omega_\nu\). Put
\[
 V_\nu=\int_Y\frac{\omega_\nu^n}{n!},\qquad
 c_{\mathrm{pull},\nu}
 =\frac1{V_\nu}\int_Y\pi^*\alpha_{X,\nu}
       \wedge\frac{\omega_\nu^{n-1}}{(n-1)!},
 \qquad
 \alpha_{X,\nu}=\frac{\ii}{2\pi}F_{Q_{X,\nu}}.
\]
The class identity
\(
[\omega_\nu]=\pi^*[\omega]+\delta_\nu[\vartheta]
\)
and the projection formula show that
\[
 V_\nu\longrightarrow\int_X\frac{\omega^n}{n!},
 \qquad
 c_{\mathrm{pull},\nu}\longrightarrow\lambda_{\det,\infty}.
 \tag{7.7}\label{eq:selector-pull-average-limit}
\]
Indeed, all other terms in the class expansion contain a positive
power of \(\delta_\nu\), with bounded coefficients.

Fix \(x_0\in U\) to normalize the scalar corrections. Define
\[
 P_\nu f=\frac{\ii}{2\pi}
             \Lambda_{\omega_\nu}\bar\partial\partial f,
\]
and solve
\[
 P_\nu r_\nu
 =\zeta_\nu:=c_{\mathrm{pull},\nu}
  -\Lambda_{\omega_\nu}\pi^*\alpha_{X,\nu},
 \qquad r_\nu(x_0)=0.
 \tag{7.8}\label{eq:selector-pull-poisson}
\]
The right-hand side has integral zero by the definition of
\(c_{\mathrm{pull},\nu}\). Smooth convergence of \(\alpha_{X,\nu}\)
on \(X\) gives
\[
 -C\pi^*\omega\le\pi^*\alpha_{X,\nu}\le C\pi^*\omega.
\]
Since \(\omega_\nu\ge\pi^*\omega\), we have
\(\Lambda_{\omega_\nu}\pi^*\omega\le n\), and hence
\(|\zeta_\nu|\le C\) uniformly. Compact convergence of the forms
and \eqref{eq:selector-pull-average-limit} also give
\(\zeta_\nu\to0\) in \(C^\infty_{\mathrm{loc}}(U)\).

To obtain \(L^1\) convergence, choose a neighborhood \(N\) of \(D\)
in \(X\) with smooth boundary and arbitrarily small \(\omega\)-volume.
The complement \(\pi^{-1}(X\setminus N)\) is compact in \(U\), so
compact convergence and convergence of the total volumes imply
\[
 \Vol_{\omega_\nu}(\pi^{-1}N)
 =V_\nu-\int_{\pi^{-1}(X\setminus N)}dV_{\omega_\nu}
 \longrightarrow\Vol_\omega(N).
\]
On this neighborhood \(\zeta_\nu\) is uniformly bounded; on its
complement it converges uniformly to zero. Letting \(\Vol_\omega(N)\)
tend to zero gives
\[
 \|\zeta_\nu\|_{L^1(Y,\omega_\nu)}\longrightarrow0.
\]

Let \(\mathcal G_\nu(x,y)\) be the Green kernel of \(P_\nu\),
with mean zero in \(y\) relative to \(dV_{\omega_\nu}\).
The operator \(P_\nu\) is a fixed positive
multiple of the scalar Laplacian in Proposition~\ref{prop:uniform-green}.
Thus \eqref{eq:uniform-green-kernel-l1} applies to \(\mathcal G_\nu\)
up to a fixed factor, and the heat-kernel bounds transfer by a
fixed rescaling of time. In particular, for each \(d>0\),
\[
 |\mathcal G_\nu(x,y)|\le C_d
 \quad\text{if }d_\omega(\pi(x),\pi(y))\ge d,
\]
uniformly in \(\nu\). To see this, integrate the heat kernel in
time after subtracting \(V_\nu^{-1}\). For \(0<t\le1\), the exponential factor in
\eqref{eq:stage-heat-kernel-bound} makes the integral finite when the
images are separated. For \(t\ge1\), the heat-kernel bound and the
uniform lower bound for the first positive eigenvalue proved in
Proposition~\ref{prop:uniform-green} give exponential decay.

Fix \(\mathcal C\Subset U\), and choose \(N\) so that its closure
is disjoint from \(\pi(\mathcal C\cup\{x_0\})\). The normalization
at \(x_0\) gives
\[
 r_\nu(x)=\int_Y
 \bigl(\mathcal G_\nu(x,y)-\mathcal G_\nu(x_0,y)\bigr)
 \zeta_\nu(y)\,dV_{\omega_\nu}(y).
\]
Split this integral over \(\pi^{-1}N\) and its complement.
The preceding pointwise bound controls the first part, and the
uniform \(L^1\) bound for the kernel controls the second. Therefore
\[
 \sup_{\mathcal C}|r_\nu|
 \le C_{\mathcal C,N}\|\zeta_\nu\|_{L^1(Y,\omega_\nu)}
    +C\|\zeta_\nu\|_{L^\infty(Y\setminus\pi^{-1}N)}
 \longrightarrow0.
\]
On a slightly larger compact subset of \(U\), the operators
\(P_\nu\) are uniformly elliptic with smoothly convergent
coefficients, and \(\zeta_\nu\to0\) smoothly. Interior Schauder
estimates applied to \eqref{eq:selector-pull-poisson} now give
\[
 r_\nu\longrightarrow0\quad\text{in }C^\infty(\mathcal C).
 \tag{7.9}\label{eq:selector-pull-csmooth}
\]
Thus
\[
 Q_{\mathrm{pull},\nu}=e^{r_\nu}Q_{X,\nu}
\]
satisfies
\(\ii\Lambda_{\omega_\nu}F_{Q_{\mathrm{pull},\nu}}
=2\pi c_{\mathrm{pull},\nu}\) and converges to \(Q_{X,\infty}\)
smoothly on compact subsets of \(U\).

\noindent\textbf{Step 3. Prescribe the exceptional exponents.}
We next construct the metric factors on
\(\cO_Y(t_aA_a)\) required by
\eqref{eq:determinant-line-discrepancy}. Let
\[
 c_a=\codim_{\CC}\overline{\pi(A_a)}\ge3,\qquad
 d_{a,\nu}=\int_{A_a}\frac{\omega_\nu^{n-1}}{(n-1)!}.
\]
Expanding the Kähler class on \(A_a\), a term with \(\ell\) factors of
\(\delta_\nu\vartheta\) contains
\((\pi^*\omega)^{n-1-\ell}\). Since the rank of
\(d(\pi|_{A_a})\) is at most \(n-c_a\), this term vanishes when
\(\ell<c_a-1\). Hence
\[
 0\le d_{a,\nu}\le C_a\delta_\nu^{c_a-1}\longrightarrow0.
 \tag{7.10}\label{eq:selector-exceptional-mass}
\]

Choose a smooth metric \(k_a\) on \(\cO_Y(A_a)\), with canonical
section \(\sigma_a\), and put \(\alpha_a=\ii F_{k_a}/(2\pi)\).
Let \(f_{a,\nu}\) be the mean-zero solution of
\[
 P_\nu f_{a,\nu}
 =\frac{d_{a,\nu}}{V_\nu}-\Lambda_{\omega_\nu}\alpha_a.
\]
The right-hand side has integral zero because \([\alpha_a]=[A_a]\).
On \(Y\setminus A_a\), set
\[
 g_{a,\nu}=f_{a,\nu}+\log|\sigma_a|_{k_a}^2,\qquad
 \widehat g_{a,\nu}=g_{a,\nu}-g_{a,\nu}(x_0).
\]
The Poincar\'e--Lelong formula gives, as an identity of distributions,
\[
 (P_\nu g_{a,\nu})\,dV_{\omega_\nu}
 =\frac{d_{a,\nu}}{V_\nu}\,dV_{\omega_\nu}
  -[A_a]\wedge\frac{\omega_\nu^{n-1}}{(n-1)!}.
\]
Here \([A_a]\) denotes the current of integration over \(A_a\).
Since \(\mathcal G_\nu\) has mean zero, its Green representation is
\[
 \widehat g_{a,\nu}(x)
 =-\int_{A_a}
 \bigl(\mathcal G_\nu(x,y)-\mathcal G_\nu(x_0,y)\bigr)
 \frac{\omega_\nu^{n-1}(y)}{(n-1)!}.
\]
For \(\mathcal C\Subset U\), the sets
\(\pi(\mathcal C\cup\{x_0\})\) and \(\pi(A_a)\) have positive
\(\omega\)-distance. The kernel bound from Step~2 and
\eqref{eq:selector-exceptional-mass} therefore give
\[
 \sup_{\mathcal C}|\widehat g_{a,\nu}|
 \le C_{\mathcal C}d_{a,\nu}\longrightarrow0.
\]
Away from \(A_a\), the same distributional identity reduces to
\[
 P_\nu\widehat g_{a,\nu}=\frac{d_{a,\nu}}{V_\nu}.
\]
Since \(V_\nu\) is bounded below, the right-hand side tends to zero.
The interior Schauder argument from Step~2 gives
\[
 \widehat g_{a,\nu}\longrightarrow0
 \quad\text{in }C^\infty(\mathcal C).
 \tag{7.11}\label{eq:selector-exceptional-limit}
\]

On \(\cO_Y(t_aA_a)\), define
\[
 Q_{a,\nu}
 =\exp\bigl(e_{a,\nu}f_{a,\nu}
        -e_{a,\nu}g_{a,\nu}(x_0)\bigr)
  k_a^{t_a}|\sigma_a|_{k_a}^{2\rho_{a,\nu}}.
\]
In a holomorphic frame near \(A_a=\{z=0\}\), its coefficient is a
smooth positive factor times \(|z|^{2\rho_{a,\nu}}\).
The equation for \(f_{a,\nu}\)
and the identity \(e_{a,\nu}=t_a+\rho_{a,\nu}\) give
\[
 \frac{\ii}{2\pi}\Lambda_{\omega_\nu}F_{Q_{a,\nu}}
 =e_{a,\nu}\frac{d_{a,\nu}}{V_\nu}.
\]
On \(Y\setminus A_a\), the nonvanishing section \(\sigma_a^{t_a}\)
trivializes this line bundle, and
\[
 Q_{a,\nu}(\sigma_a^{t_a},\sigma_a^{t_a})
 =\exp(e_{a,\nu}\widehat g_{a,\nu}).
\]

Using \eqref{eq:determinant-line-discrepancy}, set
\[
 Q_{Y,\nu}
 =Q_{\mathrm{pull},\nu}\otimes\bigotimes_aQ_{a,\nu}.
\]
Its local exponent along \(A_a\) is
\(\sum_i m_{ai}\gamma_{i,\nu}+\rho_{a,\nu}=\eta_{a,\nu}\),
and its strict-transform exponents remain \(\gamma_{i,\nu}\).
All remaining factors are smooth and strictly positive across \(B\).
Its contracted curvature is constant, with
\[
 c_{\det,\nu}
 =c_{\mathrm{pull},\nu}
  +\sum_ae_{a,\nu}\frac{d_{a,\nu}}{V_\nu}.
\]
Since the \(e_{a,\nu}\) are bounded,
\eqref{eq:selector-pull-average-limit} and
\eqref{eq:selector-exceptional-mass} give
\(c_{\det,\nu}\to\lambda_{\det,\infty}\).
Under the fixed identification of the determinant lines on \(U\),
\[
 \frac{Q_{Y,\nu}}{Q_{X,\nu}}
 =\exp\left(r_\nu+\sum_a e_{a,\nu}\widehat g_{a,\nu}\right)
 \longrightarrow1
 \quad\text{in }C^\infty_{\mathrm{loc}}(U)
\]
by \eqref{eq:selector-pull-csmooth} and
\eqref{eq:selector-exceptional-limit}. Together with Step~1, this
proves \eqref{eq:selector-common-open-limit} and the boundary formula
for \(Q_{X,\infty}\).

\noindent\textbf{Step 4. Comparison with \(q_\nu\).}
The two metrics \(Q_{Y,\nu}\) and \(q_\nu\) have the same
local monomial factors along every component of \(B\). Consequently
\[
 \ell_\nu=\log\frac{Q_{Y,\nu}}{q_\nu}
\]
extends to a smooth real function on \(Y\). Both metrics have
constant contracted curvature, so \(P_\nu\ell_\nu\) is constant.
Stokes' theorem gives \(\int_Y P_\nu\ell_\nu\,dV_{\omega_\nu}=0\);
hence \(P_\nu\ell_\nu=0\). On compact connected \(Y\), the maximum
principle implies that \(\ell_\nu\) is constant. Thus
\[
 Q_{Y,\nu}=a_\nu q_\nu,\qquad a_\nu=e^{\ell_\nu}>0.
\]

\medskip
\noindent\textbf{Step 5. Convergence of the comparison metrics.}
Since \(\det\widehat g_\nu=q_\nu\), the metric in \textup{(3)} is
\(k_\nu=a_\nu^{1/r}\widehat g_\nu\). Multiplication by the positive
constant \(a_\nu^{1/r}\) preserves every growth
lattice, the Chern connection, the Higgs adjoint, and all endomorphism
norms. The determinant of \(k_\nu\) is \(Q_{Y,\nu}\).

To prove convergence, use the local metrics \(g_{\lambda,\nu}\) and
\(\widetilde g_{\lambda,\nu}\) from Lemma~\ref{lem:doubled-reference}.
Replacing \(q_\nu\) by \(Q_{Y,\nu}=a_\nu q_\nu\) in the local
normalization gives
\[
 a_\nu^{1/r}\widetilde g_{\lambda,\nu}
 =\left(\frac{Q_{Y,\nu}}{\det g_{\lambda,\nu}}\right)^{1/r}
   g_{\lambda,\nu}.
\]
On compact subsets of each chart outside \(B\), the monomial factors
in \(g_{\lambda,\nu}\) converge smoothly as the weights converge,
while the frames and logarithmic factors are fixed. The interior
metrics are independent of \(\nu\). Thus the local metrics on the
right converge smoothly to positive metrics, by \textup{(2)}.
Their sum with the fixed partition of unity also has a smooth positive
limit. Normalizing the determinant of this sum to \(Q_{Y,\nu}\)
gives exactly \(a_\nu^{1/r}\widehat g_\nu=k_\nu\), and preserves
smooth convergence and positivity on compact subsets of \(U\).
This proves \(k_\nu\to k_\infty\); taking determinants gives
\(\det k_\infty=Q_{X,\infty}\).
\end{proof}

\subsection{Uniform analytic estimates for the K\"ahler metrics}
\label{subsec:uniform-kahler-estimates}
For the K\"ahler metrics \(\omega_\nu\), let
\(\mathcal L_\nu=d_{\omega_\nu}^{*}d\) denote the Laplacian on
functions on \(Y\). We first verify the volume and entropy bounds
needed to obtain Sobolev, heat-kernel, and Green-operator estimates
with constants independent of \(\nu\). These estimates control the
Poisson equations used for determinant normalization in
Proposition~\ref{prop:determinant-selector-limit} and for the barriers
that give bundle metric comparisons in
Proposition~\ref{prop:source-carrier}.

\begin{prop}
\label{prop:amplitude-gpss-ledger}
Let \(\omega_\nu\) be the K\"ahler metrics defined by
\eqref{eq:kahler-schedule-form}, with the parameters chosen in
Corollary~\ref{cor:stable-rational-approximations}. Retain its exponent
\(p>n\) and the normalization \(\int_Y\vartheta^n=1\), and put
\[
 V_\nu=\int_Y\frac{\omega_\nu^n}{n!},\qquad
 I_\nu=\int_Y\omega_\nu\wedge\vartheta^{n-1},\qquad
 \varrho_\nu=\frac{\omega_\nu^n}{n!V_\nu\vartheta^n}.
\]
There are constants \(V_\pm,I_\pm>0\) and \(K<\infty\), independent
of \(\nu\), such that
\[
 V_-\le V_\nu\le V_+,\qquad I_-\le I_\nu\le I_+,
\]
and
\[
 \int_Y|\log\varrho_\nu|^p\varrho_\nu\,\vartheta^n\le K.
 \tag{7.12}\label{eq:amplitude-entropy}
\]
The continuous function
\[
 \gamma=\frac{(\pi^*\omega)^n}{n!V_+\vartheta^n}
\]
is independent of \(\nu\) and satisfies
\[
 \varrho_\nu\ge\gamma,\qquad
 \dim_{\mathrm H}\{\gamma=0\}\le2n-2<2n-1.
 \tag{7.13}\label{eq:density-zero-dimension}
\]
There are an exponent \(q=q(n,p)>1\) and a constant \(C_S>0\),
both independent of \(\nu\), such that every real function
\(v\in W^{1,2}(Y,\omega_\nu)\) satisfies
\[
 \left(\int_Y|v|^{2q}\,dV_{\omega_\nu}\right)^{1/q}
 \le C_S\int_Y\bigl(|dv|_{\omega_\nu}^2+|v|^2\bigr)
       \,dV_{\omega_\nu}.
\]
Here \(dV_{\omega_\nu}=\omega_\nu^n/n!\).
Moreover,
\[
 \frac{\omega_\nu^n}{n!}
 \rightharpoonup
 \frac{(\pi^*\omega)^n}{n!}
 \tag{7.14}\label{eq:amplitude-weak-volume}
\]
weakly as measures on \(Y\).
\end{prop}

\begin{proof}
Lemma~\ref{lem:actual-kahler-smoothing}, applied to these parameters,
gives the cohomology identity
\([\omega_\nu]=\pi^*[\omega]+\delta_\nu[\vartheta]\) and
\[
 \omega_\nu\ge\pi^*\omega+
 (C_\delta-C_B)\mathfrak m_\nu\vartheta
 \ge\pi^*\omega.
 \tag{7.15}\label{eq:amplitude-positive-floor}
\]

Write
\[
 J_k=\int_Y(\pi^*\omega)^{n-k}\wedge\vartheta^k.
\]
Every \(J_k\) is nonnegative, \(J_n=1\), and
\[
 J_0=\int_X\omega^n>0
\]
because \(\pi\) is a degree-one modification. Hence
\[
 V_\nu=\frac1{n!}\sum_{k=0}^n\binom nk\delta_\nu^kJ_k,
 \qquad
 I_\nu=J_{n-1}+\delta_\nu.
\]
Moreover \(J_{n-1}>0\): on the dense open set where
\(\pi\) is biholomorphic, \(\pi^*\omega\) is positive definite, so
\(\pi^*\omega\wedge\vartheta^{n-1}\) is strictly positive there.
Since \(\delta_\nu>0\) and \(\delta_\nu\to0\), these identities give the uniform
volume and intersection bounds.

The prescribed parameters also give a uniform upper bound for the
metrics. Write \(h=\eta_\nu^2=A_\nu\epsilon_\nu\), with
\(0<h,\epsilon_\nu\le1\). The coefficient \(\mathfrak b_\nu\) from
\eqref{eq:amplitude-radial-coefficient} satisfies
\[
 \mathfrak b_\nu(s)
 =h(s+h)^{\epsilon_\nu-2}(h+\epsilon_\nu s)
 \le h(s+h)^{\epsilon_\nu-1}
 \le h^{\epsilon_\nu}\le1.
\]
The other coefficient in \eqref{eq:amplitude-differentiation}
satisfies \(h\,s(s+h)^{\epsilon_\nu-1}\le h\), since \(s\le1\).
For a squared norm \(s\) of a divisor section,
\(|\partial s|_\vartheta^2\le Cs\), while the smooth extension
of \(\ii\partial\bar\partial\log s\) is bounded.
Thus \eqref{eq:amplitude-differentiation}, summed over the finite
divisor components, gives \(\omega_\nu\le C\vartheta\).
Together with \(V_\nu\ge V_-\), this yields
\[
 0<\varrho_\nu\le C'.
\]
The function \(x|\log x|^p\), extended by zero at \(x=0\), is bounded
on \([0,C']\). Since \(\int_Y\vartheta^n=1\), the entropy bound
\eqref{eq:amplitude-entropy} follows.

Equation \eqref{eq:amplitude-positive-floor} implies
\(\omega_\nu\ge\pi^*\omega\). Monotonicity of determinants and
\(V_\nu\le V_+\) give
\[
 \varrho_\nu\ge\frac{(\pi^*\omega)^n}{n!V_+\vartheta^n}
 =\gamma.
\]
The function \(\gamma\) is continuous. Its zero set is the critical
locus of the modification, a proper complex analytic subset of complex
codimension at least one. Thus its real Hausdorff dimension is at most
\(2n-2\), proving \eqref{eq:density-zero-dimension}.

In the notation of \cite{GPSS23}, the entropy and lower-density bounds
give \(\omega_\nu\in W(n,p,\infty,K,\gamma)\) for every \(\nu\).
Theorem~2.1 of that paper gives a common exponent \(q=q(n,p)>1\)
and a normalized Sobolev inequality for functions with their averages
subtracted. The bounds for \(V_\nu\) and \(I_\nu\), together with
H\"older's inequality for the averages, give the stated Sobolev
inequality for arbitrary \(v\).

Finally, the densities of \(\omega_\nu^n/n!\) with respect to
\(\vartheta^n\) are uniformly bounded. They converge pointwise on
\(U\) to the density of \((\pi^*\omega)^n/n!\), and \(B\) has
\(\vartheta^n\)-measure zero. Dominated convergence proves
\eqref{eq:amplitude-weak-volume}.
\end{proof}

\begin{samepage}
\begin{cor}
\label{cor:pullback-gaussian}
Retain the K\"ahler metrics \(\omega_\nu\)
of Proposition~\ref{prop:amplitude-gpss-ledger}. Let
\(\mathcal L_\nu=d_{\omega_\nu}^{*}d\) be the nonnegative Laplacian on
functions, and let \(\mathsf H_\nu(t;x,y)\) be its heat kernel with
respect to \(dV_{\omega_\nu}\). There are constants \(\sigma>1\) and
\(C_H<\infty\), independent of \(\nu\), such that
\[
 \mathsf H_\nu(t;x,y)
 \le C_Ht^{-\sigma}
 \exp\left(
 -\frac{d_\omega(\pi(x),\pi(y))^2}{10t}
 \right)
 \tag{7.16}\label{eq:stage-heat-kernel-bound}
\]
for all \(x,y\in Y\) and \(0<t\le1\).
\end{cor}
\end{samepage}

\begin{proof}
Proposition~\ref{prop:amplitude-gpss-ledger} places the metrics \(\omega_\nu\)
in a common \(W(n,p,\infty,K,\gamma)\) class and supplies
uniform positive lower and finite upper bounds for \(V_\nu\) and
\(I_\nu\). Let \(q>1\) be its Sobolev exponent.
The Gaussian estimate in the geodesic
distance of \(\omega_\nu\), given by
\cite[Theorem~2.2]{GPSS23}, supplies an exponent
\[
 \sigma=\frac{q(n,p)}{q(n,p)-1}>1
\]
and a uniform constant such that
\[
 \mathsf H_\nu(t;x,y)
 \le C_Ht^{-\sigma}
 \exp\left(
 -\frac{d_{\omega_\nu}(x,y)^2}{10t}
 \right),
 \qquad 0<t\le1,
\]
where \(d_{\omega_\nu}\) is the geodesic distance of \(\omega_\nu\).
The common volume and intersection
bounds absorb the two time regimes in the cited estimate into the single
constant \(C_H\). The cited heat kernel is continuous for every smooth
compact metric, so the estimate holds for all \(x,y\in Y\).

By \eqref{eq:amplitude-positive-floor}, \(\omega_\nu\ge\pi^*\omega\).
Consequently every piecewise smooth curve \(\gamma\) from \(x\) to \(y\)
satisfies
\[
 \operatorname{Length}_{\omega_\nu}(\gamma)
 \ge
 \operatorname{Length}_{\omega}(\pi\circ\gamma).
\]
Taking the infimum over all such curves proves
\[
 d_{\omega_\nu}(x,y)
 \ge d_\omega(\pi(x),\pi(y)).
\]
Substituting this distance comparison into the preceding Gaussian estimate proves
\eqref{eq:stage-heat-kernel-bound}.
\end{proof}

\begin{prop}
\label{prop:uniform-green}
Retain the metrics \(\omega_\nu\) of
Proposition~\ref{prop:amplitude-gpss-ledger}, and put
\(\mathcal L_\nu=d_{\omega_\nu}^{*}d\) and
\(d\mu_\nu=dV_{\omega_\nu}=\omega_\nu^n/n!\).
Let \(G_\nu\) be the
inverse of \(\mathcal L_\nu\) on mean-zero functions. There is a constant
\(C_G\), independent of \(\nu\), such that
\[
 \|G_\nu F\|_{L^\infty(Y)}\le C_G\|F\|_{L^\infty(Y)}
 \tag{7.17}\label{eq:uniform-green-bound}
\]
whenever \(\int_YF\,d\mu_\nu=0\). Let \(G_\nu(x,y)\) denote the normalized
mean-zero Green kernel. After enlarging \(C_G\) if necessary, we also have
\[
 \sup_{\nu}\sup_{x\in Y}
 \int_Y|G_\nu(x,y)|\,d\mu_\nu(y)\le C_G.
 \tag{7.18}\label{eq:uniform-green-kernel-l1}
\]
\end{prop}

\begin{proof}
The Sobolev estimate in
Proposition~\ref{prop:amplitude-gpss-ledger}, written using
\(d\mu_\nu=dV_{\omega_\nu}\), is
\[
 \|v\|_{L^{2q}(d\mu_\nu)}^2
 \le C\bigl(\|dv\|_{L^2(d\mu_\nu)}^2+\|v\|_{L^2(d\mu_\nu)}^2\bigr)
 \tag{7.19}\label{eq:uniform-sobolev}
\]
for \(q=q(n,p)>1\).

We next prove a uniform Poincaré inequality.
Otherwise there are mean-zero \(v_\nu\) with \(\|v_\nu\|_2=1\) and
\(\|dv_\nu\|_2\to0\). On each compact subset of \(U\), compact metric
convergence and Rellich compactness give a constant local limit. The
constants agree on overlapping compact sets because \(U\) is connected.
Weak volume convergence
\eqref{eq:amplitude-weak-volume} shows that sufficiently deep divisor
collars have uniformly small volume. For every such collar \(E\),
Hölder's inequality gives
\[
 \int_E|v_\nu|^2\,d\mu_\nu
 \le\mu_\nu(E)^{1-1/q}
 \|v_\nu\|_{L^{2q}(d\mu_\nu)}^2.
\]
Since \(q>1\), equation \eqref{eq:uniform-sobolev} prevents
\(L^2\)-mass from escaping into those collars. Thus the common local
constant has norm one and mean zero, a contradiction. Hence the first
positive eigenvalue of every \(\mathcal L_\nu\) is bounded below by one
positive constant.

Let \(P_{\nu,t}\) be the heat semigroup and \(\Pi_\nu\) projection onto the
constants. The Markov property gives
\(\|P_{\nu,t}\|_{\infty\to\infty}\le1\). The heat-kernel bound
\eqref{eq:stage-heat-kernel-bound} from
Corollary~\ref{cor:pullback-gaussian}, evaluated at \(t=1/2\), together
with the uniform spectral gap and the semigroup property, gives
\[
 \|P_{\nu,t}-\Pi_\nu\|_{\infty\to\infty}
 \le Ce^{-\lambda t}\qquad(t\ge1).
\]
For mean-zero \(F\),
\[
 G_\nu F=\int_0^\infty P_{\nu,t}F\,dt.
\]
Splitting the integral at \(t=1\) proves
\eqref{eq:uniform-green-bound}. To prove
\eqref{eq:uniform-green-kernel-l1}, fix \(\nu\) and \(x\in Y\), and put
\[
 \psi(y)=\operatorname{sgn}G_\nu(x,y),\qquad
 \bar\psi=\frac{\int_Y\psi\,d\mu_\nu}{\int_Yd\mu_\nu}.
\]
Then \(\psi-\bar\psi\) has mean zero and
\(\|\psi-\bar\psi\|_{L^\infty}\le2\). Since
\(\int_YG_\nu(x,y)\,d\mu_\nu(y)=0\), the kernel representation and
\eqref{eq:uniform-green-bound} give
\[
 \int_Y|G_\nu(x,y)|\,d\mu_\nu(y)
 =G_\nu(\psi-\bar\psi)(x)\le2C_G.
\]
Enlarging \(C_G\) proves \eqref{eq:uniform-green-kernel-l1}.
\end{proof}

\subsection{Curvature estimates and scalar barriers}
\label{subsec:curvature-barriers}
We first obtain bounds for the Hermitian--Einstein error of the
comparison metrics. We then construct scalar barriers for comparing
them with Hermitian--Einstein metrics. Retain \(k_\nu\) and its positive limit
\(k_\infty\) from Proposition~\ref{prop:determinant-selector-limit},
on the common open set \(U=Y\setminus B=X\setminus D\). Put
\[
 c_\nu=\frac{2\pi c_{\det,\nu}}r,\qquad
 c_\infty=\frac{2\pi\lambda_{\det,\infty}}r,
\]
and define
\[
 \Gamma_\nu=
 \ii\Lambda_{\omega_\nu}
 \bigl(F_{k_\nu}+[\theta,\theta^{\dagger k_\nu}]\bigr)-c_\nu\Id_V,
 \qquad
 \mathfrak e_\nu=|\Gamma_\nu|_{\op,k_\nu}.
 \tag{7.20}\label{eq:defined-moving-source}
\]
The nonnegative scalar function \(\mathfrak e_\nu\)
is the operator norm of the Hermitian--Einstein error of \(k_\nu\).
Unless otherwise indicated, fiber
norms are Hilbert--Schmidt norms; the subscript \(\op\) denotes the
operator norm.

The convergence of \(k_\nu\) and \(c_\nu\) supplied by
Proposition~\ref{prop:determinant-selector-limit}, together with
\(\omega_\nu\to\omega\) in \(C^\infty_{\mathrm{loc}}(U)\) from
\eqref{eq:kahler-schedule-form}, gives
\[
 \Gamma_\nu\longrightarrow\Gamma_\infty
 \quad\text{in }C^\infty_{\mathrm{loc}}(U),\qquad
 \Gamma_\infty:=
 \ii\Lambda_\omega
 \bigl(F_{k_\infty}+[\theta,\theta^{\dagger k_\infty}]\bigr)
 -c_\infty\Id.
\]
\begin{prop}
\label{prop:local-contraction-estimates}
With the preceding metrics and the notation of
Setup~\ref{setup:literal-moving-schedule}, there is a constant \(C\),
independent of \(\nu\), such that on each chart \(U_\lambda\),
\[
 \mathfrak e_\nu\le
 C\left(\sum_i\Xi_{\lambda,i,\nu}
       +\sum_{i\ne k}\Xi_{\lambda,ik,\nu}\right).
 \tag{7.21}\label{eq:actual-source-row-bound}
\]
Moreover,
\[
 \sup_\nu\int_U\mathfrak e_\nu\,dV_{\omega_\nu}<\infty,
\]
and
\[
 \lim_{t\downarrow0}\sup_\nu
 \int_{U\cap\bigcup_b\{s_b<t\}}
 \mathfrak e_\nu\,dV_{\omega_\nu}=0.
\]
\end{prop}

\begin{proof}
\noindent\textbf{Step 1. The pointwise error bound.}
Put
\[
 Z_\nu=
 \left(F_{\widehat g_\nu}
       +[\theta,\theta^{\dagger\widehat g_\nu}]\right)^0.
\]
Lemma~\ref{lem:doubled-reference} gives
\(\lvert(Z_\nu)_{i\bar k}\rvert_{\op,\widehat g_\nu}
 \le Ch_{\lambda,i}h_{\lambda,k}\).
By Proposition~\ref{prop:determinant-selector-limit}\textup{(3)},
\(Z_\nu=(F_{k_\nu}+[\theta,\theta^{\dagger k_\nu}])^0\), and the
same coefficient bound holds in the \(k_\nu\)-norm.
The determinant equation identifies the error with its contraction:
\[
 \ii\Lambda_{\omega_\nu}F_{\det k_\nu}=rc_\nu,\qquad
 \Gamma_\nu=\ii\Lambda_{\omega_\nu}Z_\nu.
\]
Contracting the coefficient bound with \(G_{\lambda,\nu}^{-1}\)
therefore proves \eqref{eq:actual-source-row-bound}.

\medskip
\noindent\textbf{Step 2. Cofactor bounds for the volume densities.}
Work on a shrunken chart and suppress \(\lambda\) in its local indices.
By a fixed finite refinement, we may assume that its closure maps
under \(\pi\) into a coordinate chart of \(X\). Let \(J\) be the
holomorphic Jacobian matrix of \(\pi\) in these coordinates.
The coefficient matrix of \(\pi^*\omega\) is uniformly comparable
with \(J^*J\).

Use the coefficient \(\mathfrak b_\nu\) computed in the proof of
Corollary~\ref{cor:stable-rational-approximations}, and put
\[
 d_{i,\nu}=
 \begin{cases}
 \mathfrak b_\nu(\rho_i),&i\in D_\lambda,\\
 0,&i\notin D_\lambda.
 \end{cases}
\]
We claim that the matrix \(G_\nu=G_{\lambda,\nu}\) satisfies
\[
 C^{-1}\mathcal G_\nu\le G_\nu\le C\mathcal G_\nu,\qquad
 \mathcal G_\nu:=
 J^*J+\diag(\delta_\nu+d_{1,\nu},\ldots,\delta_\nu+d_{n,\nu}),
 \tag{7.22}\label{eq:chart-actual-gram}
\]
with \(C\) independent of \(\nu\).
To check this, write \(s_b=u_i|z_i|^2\) for a component meeting the
chart, where \(u_i\) is smooth and strictly positive.
Formula \eqref{eq:amplitude-differentiation} writes its positive
contribution to \(\omega_\nu\) as
\[
 u_i\mathfrak b_\nu(s_b)\,
 \ii\alpha_i\wedge\overline{\alpha_i},
 \qquad \alpha_i=dz_i+z_i\partial\log u_i.
\]
The functions \(u_i,u_i^{-1}\) and the derivatives of \(u_i\) up to
order two are uniformly bounded. The explicit radial formula gives
\[
 u_i\mathfrak b_\nu(u_i\rho_i)\asymp d_{i,\nu},
 \qquad
 \rho_i u_i\mathfrak b_\nu(u_i\rho_i)\le\mathfrak m_\nu.
\]
Thus replacing \(\alpha_i\) by \(dz_i\) in either direction costs only
a fixed multiple of \(\mathfrak m_\nu\vartheta\), besides a fixed
factor in the positive normal term. The divisor-curvature terms and
the contributions of components not meeting the chart are also
\(O(\mathfrak m_\nu\vartheta)\). Finally,
Lemma~\ref{lem:actual-kahler-smoothing} gives
\[
 \omega_\nu\ge\pi^*\omega+
 (C_\delta-C_B)\mathfrak m_\nu\vartheta,
 \qquad \delta_\nu=C_\delta\mathfrak m_\nu.
\]
Thus each \(O(\mathfrak m_\nu\vartheta)\) error is bounded by a fixed
multiple of \(\omega_\nu\).
Applying \(\lvert v+w\rvert^2\le2\lvert v\rvert^2+2\lvert w\rvert^2\)
to \(\alpha_i=dz_i+z_i\partial\log u_i\), and then to
\(dz_i=\alpha_i-z_i\partial\log u_i\), proves
\eqref{eq:chart-actual-gram}.

Let \(\Delta_\nu=\det\mathcal G_\nu\), and let
\(\Delta_{i,\nu}\) be the determinant obtained by deleting the
\(i\)-th row and column of \(\mathcal G_\nu\).
For use here and in the barrier construction below, we record their expansion.
For \(S\subset\{1,\ldots,n\}\), put
\[
 M_S=\sum_{\substack{A\subset\{1,\ldots,n\}\\|A|=|S|}}
       |\det J_{A,S}|^2,\qquad M_\varnothing=1.
\]
For positive \(r_1,\ldots,r_n\), Cauchy--Binet gives
\[
 \begin{aligned}
 \Delta(r)&:=\det\bigl(J^*J+\diag(r_1,\ldots,r_n)\bigr)
       =\sum_{S\subset\{1,\ldots,n\}}M_S\prod_{j\notin S}r_j,\\
 \Delta_\nu&=\Delta(r_\nu),\qquad
 \Delta_{i,\nu}=(\partial_{r_i}\Delta)(r_\nu),\qquad
 r_{i,\nu}=\delta_\nu+d_{i,\nu}.
 \end{aligned}
 \tag{7.23}\label{eq:cb-polynomial}
\]
All minors of \(J\) are bounded on the fixed chart, \(\delta_\nu\) is
bounded, and \(0\le d_{i,\nu}\le1\). Hence
\(\Delta_{i,\nu}\le C\) uniformly.
Taking determinants and inverses in \eqref{eq:chart-actual-gram}, and
using the cofactor formula, gives
\[
 \det G_\nu\le C\Delta_\nu,\qquad
 (G_\nu^{-1})^{i\bar i}\le C\frac{\Delta_{i,\nu}}{\Delta_\nu}.
\]
Moreover, positivity implies
\(\lvert(G_\nu^{-1})^{i\bar k}\rvert
 \le\sqrt{(G_\nu^{-1})^{i\bar i}(G_\nu^{-1})^{k\bar k}}\).
Multiplying by the volume density cancels the determinant denominator.
With \(d\lambda\) denoting Euclidean volume in the chart, we obtain
\[
 \begin{aligned}
 \Xi_{i,\nu}\,dV_{\omega_\nu}
 &\le C h_i^2\Delta_{i,\nu}\,d\lambda,\\
 \Xi_{ik,\nu}\,dV_{\omega_\nu}
 &\le C h_ih_k\sqrt{\Delta_{i,\nu}\Delta_{k,\nu}}\,d\lambda.
 \end{aligned}
\]

\medskip
\noindent\textbf{Step 3. Integrability near the divisor.}
By Step~2, it suffices to integrate \(h_ih_k\) against Euclidean
volume with the weight \(1+\sum_a\log\tau_a\).
In a divisor coordinate, put \(s=|z_a|^2\) and let \(0<R<1\)
be its squared radius. A coordinate occurring twice in \(h_ih_k\)
requires
\[
 \int_0^R
 \frac{1+\log(1-\log s)}{s(1-\log s)^2}\,ds<\infty.
\]
For a coordinate occurring once, the denominator is instead
\(\sqrt{s}(1-\log s)\); for a coordinate absent from \(h_ih_k\),
there is no denominator. These integrals are also finite.
Fubini's theorem gives, for all \(i,k\),
\[
 \sup_\nu\int_{U_\lambda\setminus B}
 \Xi_{\lambda,ik,\nu}
 \left(1+\sum_{a\in D_\lambda}\log\tau_{\lambda,a}\right)
 dV_{\omega_\nu}\le C.
\]
On a shrunken chart, the norms of the divisor components meeting it
satisfy \(s_b=u_{\lambda,b}\rho_{\lambda,i}\), with bounded positive
units and bounded reciprocals; the other divisor norms are bounded
away from zero. The pointwise estimate
\eqref{eq:actual-source-row-bound}, the weighted bound just proved,
and the finite atlas therefore give
\[
 \sup_\nu\int_U\mathfrak e_\nu
 \left(1+\sum_b\log(1-\log s_b)\right)dV_{\omega_\nu}\le C.
\]
On \(\bigcup_b\{s_b<t\}\), for \(0<t<1\), the weight is at least
\(1+\log(1-\log t)\). Dividing the uniform bound by this quantity,
which tends to infinity as \(t\downarrow0\), proves the asserted
uniform boundary limit as well as the global \(L^1\) bound.
\end{proof}

We next construct scalar barrier functions to compare Hermitian--Einstein
metrics with \(k_\nu\). To explain their role, let \(H_\nu^{\mathrm D}\)
be a determinant-fixed central Hermitian--Einstein metric on a smooth
domain \(\Omega_\nu\Subset U\), with central constant \(c_\nu\) and
boundary value \(k_\nu\), as in Lemma~\ref{lem:carrier-response}.
Consider the nonnegative function
\[
 v_\nu=\log\frac{
 \Tr(k_\nu^{-1}H_\nu^{\mathrm D})
 +\Tr((H_\nu^{\mathrm D})^{-1}k_\nu)}{2r}.
\]
Applying the Bochner formula to the identity map in both directions
between \((V,k_\nu,\theta)\) and \((V,H_\nu^{\mathrm D},\theta)\),
as in the proof of Proposition~\ref{prop:compact-source-reference},
gives
\[
 \cL_\nu v_\nu\le2\mathfrak e_\nu
 \quad\text{on }\Omega_\nu,\qquad
 v_\nu=0\quad\text{on }\partial\Omega_\nu.
\]

We therefore seek a nonnegative function \(b_\nu\) satisfying
\(\cL_\nu b_\nu\ge2\mathfrak e_\nu\) outside a fixed compact
domain \(W_0\Subset U\). Whenever \(W_0\subset\Omega_\nu\), put
\(M_\nu=\sup_{W_0}v_\nu\). The maximum principle applied to
\(v_\nu-b_\nu\) gives
\[
 v_\nu\le M_\nu+b_\nu
 \quad\text{on }\Omega_\nu.
\]
Both traces in the definition of \(v_\nu\) are at most \(2re^{v_\nu}\),
so this gives the two-sided metric estimate
\[
 \frac{e^{-M_\nu-b_\nu}}{2r}\,k_\nu
 \le H_\nu^{\mathrm D}
 \le 2r\,e^{M_\nu+b_\nu}k_\nu
 \quad\text{on }\Omega_\nu.
\]
Thus a bound for \(v_\nu\) on the fixed compact set \(W_0\),
together with the barrier, controls the comparison of the two metrics
throughout the Dirichlet domain. The following proposition constructs
these barriers and bounds
\(\int_U\mathfrak e_\nu(1+b_\nu)\,dV_{\omega_\nu}\)
uniformly in \(\nu\). This integral controls the source pairing in the
normalized Donaldson identity; the boundary growth of the barriers
controls the parabolic growth of the limiting metric.

\begin{prop}
\label{prop:source-carrier}
Let \(k_\nu\) be the comparison metrics of
Proposition~\ref{prop:determinant-selector-limit}, let
\(\mathfrak e_\nu\) be defined by \eqref{eq:defined-moving-source}, and put
\(\cL_\nu=d_{\omega_\nu}^*d\).
There are connected compact domains with smooth boundary
\[
 W_0\Subset\operatorname{int}(W_1)\Subset U
\]
and nonnegative smooth functions \(b_\nu\) on \(U\) such that, for all
sufficiently large \(\nu\),
\[
 \cL_\nu b_\nu\ge2\mathfrak e_\nu
 \quad\text{on }U\setminus W_0,
 \tag{7.24}\label{eq:carrier-barrier}
\]
and
\[
 \sup_\nu\int_U\mathfrak e_\nu(1+b_\nu)\,dV_{\omega_\nu}<\infty.
 \tag{7.25}\label{eq:carrier-first-moment}
\]

The functions \(b_\nu\) are uniformly bounded on every compact subset
of \(U\). On each shrunken SNC chart \(U_\lambda\), let
\(E_\lambda\subset D_\lambda\) index the exceptional divisor components.
There is a constant \(C_\lambda>1\), independent of \(\nu\), such that
\[
 b_\nu\le C_\lambda\left(
 1+\sum_{a\in D_\lambda}\log\tau_{\lambda,a}
 +\sum_{a\in E_\lambda}
   \log\frac{C_\lambda}{\delta_\nu^2+\rho_{\lambda,a}}
 \right).
 \tag{7.26}\label{eq:carrier-growth}
\]
Thus, for fixed \(\nu\), the boundary growth is at most log--log.
On charts whose closures avoid the exceptional locus, this log--log
bound is uniform in \(\nu\).
\end{prop}

\begin{proof}
\noindent\textbf{Step 1. A scalar function dominating the error up to a constant.}
The matrix comparison \eqref{eq:chart-actual-gram} and the bounds
\(0\le\mathfrak b_\nu\le1\) give uniform constants such that
\[
 c\delta_\nu\vartheta\le\omega_\nu\le C\vartheta,\qquad
 \pi^*\omega\le C\omega_\nu.
\]
Put \(R_\nu=\Lambda_{\omega_\nu}\vartheta\); in particular,
\(R_\nu\le C/\delta_\nu\). The role of the auxiliary functions below
is to control this contraction, which can grow near the exceptional locus.

Choose an effective exceptional divisor
\[
 E_{\mathrm{exc}}=\sum_{\ell=1}^m a_\ell A_\ell,\qquad a_\ell>0,
\]
with support equal to the exceptional locus and with
\(\cO_Y(-E_{\mathrm{exc}})\) relatively ample over \(X\).
To justify this choice, recall that the modification \(\pi:Y\to X\)
from Theorem~\ref{thm:fixed-terminal-data} is projective. Choose a
\(\pi\)-ample line bundle \(L_0\), and put
\[
 \mathcal F=\pi_*L_0,\qquad M=\mathcal F^{**}.
\]
Since \(\pi\) is proper and bimeromorphic, \(\mathcal F\) is a coherent
torsion-free sheaf of rank one. Thus \(M\) is a rank-one reflexive
sheaf, and hence a line bundle because \(X\) is smooth. Set
\[
 L_1=L_0\otimes\pi^*M^{-1}.
\]
Tensoring with the pullback of a line bundle on \(X\) preserves relative
ampleness: on any open set where \(M\) is trivial, \(L_1\) and \(L_0\)
are isomorphic over its inverse image. In particular, \(L_1\) is
\(\pi\)-ample.

The evaluation map \(\pi^*\mathcal F\to L_0\) and the map
\(\pi^*\mathcal F\to\pi^*M\) induced by \(\mathcal F\hookrightarrow M\)
are isomorphisms wherever \(\pi\) is an isomorphism. Comparing these
maps gives a meromorphic section of \(L_1\) which is holomorphic and
nowhere zero off the exceptional locus. Its divisor therefore has the form
\[
 G=\sum_{\ell=1}^m b_\ell A_\ell,\qquad b_\ell\in\mathbb Z,
 \qquad L_1\simeq\cO_Y(G).
\]
For every complete
curve \(C\) contracted by \(\pi\), relative ampleness gives
\(G\cdot C=\deg(L_1|_C)>0\); in particular, \(G\) is \(\pi\)-nef.
The negativity lemma for projective bimeromorphic morphisms of normal
complex analytic spaces says that a \(\pi\)-exceptional, \(\pi\)-nef
Cartier divisor has nonpositive coefficients
\cite[Section~11]{Fujino22}. It follows that \(G\le0\). Hence
\(E_{\mathrm{exc}}=-G\) is effective and
\(\cO_Y(-E_{\mathrm{exc}})\simeq L_1\) is \(\pi\)-ample.

It remains to check that no exceptional component is omitted. If some
\(A_\ell\) were not contained in \(\Supp E_{\mathrm{exc}}\), choose a
general point of \(A_\ell\) outside this support. Since
\(\operatorname{codim}_X\pi(A_\ell)\ge2\), the fiber of
\(A_\ell\to\pi(A_\ell)\) through this point is positive-dimensional
and projective, so it contains a complete curve \(C\) through the point.
The curve \(C\) is not contained in \(\Supp E_{\mathrm{exc}}\), and
effectivity therefore gives \(E_{\mathrm{exc}}\cdot C\ge0\). On the
other hand,
\[
 E_{\mathrm{exc}}\cdot C=-\deg(L_1|_C)<0,
\]
a contradiction. Thus every coefficient \(a_\ell=-b_\ell\) is a
positive integer, as required.

Choose a smooth metric on \(\cO_Y(-E_{\mathrm{exc}})\) with curvature
positive on the vertical tangent spaces. Choose smooth Hermitian
metrics on the divisor line bundles \(\cO_Y(A_\ell)\), and let
\(r_\ell\le1\) be the squared norms of their canonical sections, chosen so that
\[
 \rho_E=\prod_\ell r_\ell^{a_\ell}
\]
is the squared norm of the canonical section of
\(\cO_Y(E_{\mathrm{exc}})\) for the dual metric.
Its logarithmic Hessian extends smoothly and satisfies
\[
 \Theta_E:=\ii\partial\bar\partial\log\rho_E
 \ge c\vartheta-C\pi^*\omega.
\]
Indeed, its positivity on the kernel of \(d\pi\), followed by adding
a sufficiently large multiple of \(\pi^*\omega\), gives a positive
form on compact \(Y\).

Let \(A_E=\sum_\ell a_\ell\), and choose fixed numerator constants so
that
\[
 \psi_{0,\nu}=\log\frac{C_0}{\delta_\nu^{2A_E}+\rho_E},
 \qquad
 \psi_{\ell,\nu}=\log\frac{C_\ell}{\delta_\nu^2+r_\ell}
\]
are nonnegative. For a smooth squared divisor norm \(r\) and \(q>0\),
\[
 \ii\partial\bar\partial\log(q+r)
 =\frac{r}{q+r}\ii\partial\bar\partial\log r
 +\frac{qr}{(q+r)^2}
   \ii\partial\log r\wedge\bar\partial\log r.
\]
Since \(\cL_\nu=-2\ii\Lambda_{\omega_\nu}\partial\bar\partial\),
the curvature bound for \(\Theta_E\) implies
\[
 \cL_\nu\psi_{0,\nu}
 \ge c\frac{\rho_E}{\delta_\nu^{2A_E}+\rho_E}R_\nu-C.
\]
Each function \(\psi_{\ell,\nu}\) satisfies \(\cL_\nu\psi_{\ell,\nu}\ge-CR_\nu\).
On \(r_\ell\le\delta_\nu^2\), its positive normal term gives the
stronger estimate
\[
 \cL_\nu\psi_{\ell,\nu}\ge c\delta_\nu^{-2}-CR_\nu.
\]
To see the last assertion, write \(r_\ell=u|z|^2\) with \(u\) smooth
and positive. Then \(|\partial r_\ell|_{\omega_\nu}^2/r_\ell\)
has a uniform positive lower bound near \(z=0\), since
\(\omega_\nu\le C\vartheta\), while
\(\delta_\nu^2/(\delta_\nu^2+r_\ell)^2\ge1/(4\delta_\nu^2)\).

If all \(r_\ell\ge\delta_\nu^2\), the fraction in the bound for
\(\psi_{0,\nu}\) is at least \(1/2\). A sufficiently large fixed
coefficient \(c_0\) therefore makes
\[
 \Psi_\nu=c_0\psi_{0,\nu}+\sum_\ell\psi_{\ell,\nu}
\]
satisfy \(\cL_\nu\Psi_\nu\ge R_\nu-C\) on this region.
If some \(r_\ell\le\delta_\nu^2\), the term
\(c\delta_\nu^{-2}\) dominates all errors of size \(CR_\nu\),
because \(R_\nu\le C/\delta_\nu\). Thus the same inequality holds
on \(U\) for all sufficiently large stages.

Now put
\[
 \beta=\sum_b\log(1-\log s_b).
\]
Writing \(T_b=1-\log s_b\), direct differentiation gives
\[
 \cL_\nu\log T_b
 =\frac{2|\partial\log s_b|_{\omega_\nu}^2}{T_b^2}
 +\frac{2\Lambda_{\omega_\nu}
    (\ii\partial\bar\partial\log s_b)}{T_b}.
\]
On a chart, \(s_b=u_{\lambda,b}|z_{\lambda,a}|^2\). The first
term controls \(h_{\lambda,a}^2(G_{\lambda,\nu}^{-1})^{a\bar a}\)
up to an error \(CR_\nu\); the smooth curvature term has the same
error bound. The nondivisor terms are bounded by \(CR_\nu\), and
\[
 2\Xi_{\lambda,ik,\nu}
 \le\Xi_{\lambda,i,\nu}+\Xi_{\lambda,k,\nu}.
\]
Consequently \eqref{eq:actual-source-row-bound} gives
\(\cL_\nu\beta\ge c\mathfrak e_\nu-CR_\nu\).
First choosing \(C_\beta\) large, and then \(C_\Psi\) large, yields
a nonnegative function
\[
 \Phi_\nu=C_\beta\beta+C_\Psi\Psi_\nu,\qquad
 \cL_\nu\Phi_\nu\ge4\mathfrak e_\nu-C_{\mathrm{err}},
\]
where all constants are independent of \(\nu\).

\medskip
\noindent\textbf{Step 2. Removing the constant error.}
Choose \(W_0\Subset\operatorname{int}W_1\Subset U\) as in the
statement, and a nonzero nonnegative smooth function \(\chi\)
supported in \(\operatorname{int}W_0\).
Compact convergence of \(\omega_\nu\) gives
\(\int_Y\chi\,dV_{\omega_\nu}\ge c_\chi>0\), and the volumes of
\(Y\) are uniformly bounded. Hence
\[
 F_\nu=C_{\mathrm{err}}-
 \frac{C_{\mathrm{err}}\Vol_{\omega_\nu}(Y)}
      {\int_Y\chi\,dV_{\omega_\nu}}\chi
\]
is uniformly bounded and has mean zero. By
Proposition~\ref{prop:uniform-green}, its mean-zero solution
\(\gamma_\nu\) satisfies
\[
 \cL_\nu\gamma_\nu=F_\nu,\qquad
 \sup_\nu\|\gamma_\nu\|_{L^\infty(Y)}<\infty.
\]
For one sufficiently large constant \(C_*\), set
\[
 b_\nu=\frac12(C_*+\Phi_\nu+\gamma_\nu)\ge0.
\]
Outside \(W_0\), the correction cancels \(C_{\mathrm{err}}\);
this proves \eqref{eq:carrier-barrier}.

\medskip
\noindent\textbf{Step 3. Growth of the barrier.}
On a shrunken chart, the squared norms of the canonical divisor sections
are comparable to the squared moduli of their defining coordinates.
The norms for components not
meeting the chart are bounded below. Moreover,
\[
 \log\frac{C_0}{\delta_\nu^{2A_E}+\rho_E}
 \le C\left(1+\sum_{a\in E_\lambda}
       \log\frac{C_\lambda}{\delta_\nu^2+\rho_{\lambda,a}}\right).
\]
Indeed, put \(x_\ell=-\log r_\ell\) and
\(d_\nu=|\log\delta_\nu|\). Up to a bounded constant, the left side
is at most
\[
 \min\left(2A_Ed_\nu,\sum_\ell a_\ell x_\ell\right)
 \le A_E\sum_\ell\min(2d_\nu,x_\ell).
\]
Each term in this sum differs by a bounded amount from
\(\log(C_\ell/(\delta_\nu^2+r_\ell))\).
The definitions of \(\Phi_\nu\) and \(b_\nu\) now give
\eqref{eq:carrier-growth}. They also give uniform boundedness on
compact subsets of \(U\), where all divisor norms are bounded below.
For fixed \(\nu\), the exceptional logarithms are bounded; on
charts avoiding the exceptional locus their bounds are uniform in
\(\nu\). This proves the growth assertions.

\medskip
\noindent\textbf{Step 4. Integral estimates involving the barriers.}
We prove \eqref{eq:carrier-first-moment}.
Let \(\mathcal H_{\lambda,\nu}\ge1\) denote the expression in
parentheses on the right of \eqref{eq:carrier-growth}.
It suffices to bound the integrals of
\(\Xi_{\lambda,i,\nu}\mathcal H_{\lambda,\nu}\):
the mixed terms are bounded by the diagonal terms, and
\eqref{eq:actual-source-row-bound} then bounds the source.

Let \(J_\lambda\) be the holomorphic Jacobian of \(\pi\) and let
\(\Delta_{\lambda,i,\nu}\) be the \(i\)-th diagonal cofactor of the
comparison matrix in \eqref{eq:chart-actual-gram}.
For \(i\in E_\lambda\), the restriction of \(\pi\) to
\(z_{\lambda,i}=0\) has rank at most \(n-2\).
Every \((n-1)\)-minor obtained by deleting column \(i\) of
\(J_\lambda\) is therefore divisible by \(z_{\lambda,i}\).
The cofactor expansion \eqref{eq:cb-polynomial} consequently gives
\[
 \Delta_{\lambda,i,\nu}\le
 \begin{cases}
 C,&i\notin E_\lambda,\\
 C\left(\rho_{\lambda,i}
       +\displaystyle\sum_{k\ne i}
         (\delta_\nu+d_{\lambda,k,\nu})\right),&i\in E_\lambda.
 \end{cases}
 \tag{7.27}\label{eq:exceptional-cofactor-bound}
\]
Here the terms containing no factor \(\delta_\nu+d_{\lambda,k,\nu}\),
\(k\ne i\), are squares of the indicated minors. Every other term
contains at least one such factor.
Step~2 of Proposition~\ref{prop:local-contraction-estimates} gives
\[
 \Xi_{\lambda,i,\nu}\,dV_{\omega_\nu}
 \le C h_{\lambda,i}^2\Delta_{\lambda,i,\nu}\,d\lambda.
\]

The coefficients \(\mathfrak b_\nu\) satisfy
\[
 0\le\mathfrak b_\nu(s)\le1,\qquad
 \int_0^R\mathfrak b_\nu(s)\,ds
 =\mathfrak m_\nu R(R+\mathfrak m_\nu)^{\epsilon_\nu-1}
 \le C\delta_\nu.
 \tag{7.28}\label{eq:literal-spike-mass}
\]
This follows by differentiating
\(\mathfrak m_\nu s(s+\mathfrak m_\nu)^{\epsilon_\nu-1}\).
Recall that \(d_{\lambda,k,\nu}=\mathfrak b_\nu(\rho_{\lambda,k})\)
for \(k\in D_\lambda\), and is zero otherwise.

For the one-variable integrals, put
\[
 \tau(s)=1-\log s,\qquad
 L_\delta(s)=1+\log\frac{C}{\delta^2+s},\qquad 0<s<R<1.
\]
Uniformly for \(0<\delta\le1\), we have
\[
 \begin{aligned}
 &\int_0^R L_\delta(s)\,ds\le C,\\
 &\delta\int_0^R\frac{L_\delta(s)}{s\tau(s)^2}\,ds\le C,\\
 &\int_0^R d(s)L_\delta(s)\,ds\le C
 \quad\text{if }0\le d\le C_d,\quad \int_0^R d(s)\,ds\le C_d\delta.
 \end{aligned}
 \tag{7.29}\label{eq:exceptional-integral-bounds}
\]
The first bound follows from
\(L_\delta(s)\le C(1+|\log s|)\).
For the second, set \(t=\tau(s)\) and
\(D_\delta=1+|\log\delta|\). Since
\(L_\delta(s)\le C\min(2D_\delta,t)\), its left side is at most
\[
 C\delta\left(\int_1^{2D_\delta}\frac{dt}{t}+
       2D_\delta\int_{2D_\delta}^\infty\frac{dt}{t^2}\right)
 \le C\delta(1+\log D_\delta)\le C.
\]
The last bound follows from
\(L_\delta\le CD_\delta\) and the assumed integral bound on \(d\).
The log--log factors are controlled by
\[
 \int_0^R\frac{(1+\log\tau(s))}{s\tau(s)^2}\,ds<\infty,
 \qquad
 \int_0^R(1+\log\tau(s))\,ds<\infty.
\]
Also,
\[
 \mathcal H_{\lambda,\nu}
 \le C\left(1+\sum_{a\in E_\lambda}
           L_{\delta_\nu}(\rho_{\lambda,a})
          +\sum_{a\in D_\lambda}\log\tau_{\lambda,a}\right).
\]

If \(i\notin E_\lambda\), the bound
\(\Delta_{\lambda,i,\nu}\le C\) suffices: any exceptional logarithm
occurs in a coordinate different from \(i\), so Fubini's theorem and
the preceding integrals apply.
If \(i\in E_\lambda\), use the second cofactor bound in
\eqref{eq:exceptional-cofactor-bound}. Its
\(\rho_{\lambda,i}\)-term cancels the factor
\(1/\rho_{\lambda,i}\) in \(h_{\lambda,i}^2\).
A term containing \(\delta_\nu\) is controlled by the second
bound in \eqref{eq:exceptional-integral-bounds}.
For a term containing \(d_{\lambda,k,\nu}\), \(k\ne i\), there are
two relevant possibilities. If the exceptional logarithm occurs
in coordinate \(i\), first integrate coordinate \(k\), using
\eqref{eq:literal-spike-mass}, and then use that same second bound.
If it occurs in coordinate \(k\), use the third bound in
\eqref{eq:exceptional-integral-bounds}; the integral of
\(1/(s\tau(s)^2)\) in coordinate \(i\) is finite.
All other placements separate into the first bound and the
log--log integrals above. Thus
\[
 \sup_\nu\int_{U_\lambda\setminus B}
 \Xi_{\lambda,i,\nu}\mathcal H_{\lambda,\nu}
 \,dV_{\omega_\nu}<\infty.
\]
Summing over the finite atlas proves \eqref{eq:carrier-first-moment}.
\end{proof}

The limiting comparison metric consequently has integrable contracted
curvature:
\[
 \Phi_\infty:=
 \ii\Lambda_\omega
 \bigl(F_{k_\infty}+[\theta,\theta^{\dagger k_\infty}]\bigr)
 \in L^1(U,k_\infty,\omega).
 \tag{7.30}\label{eq:downstairs-contracted-source}
\]
Indeed, Proposition~\ref{prop:local-contraction-estimates} bounds
\(\int_U|\Gamma_\nu|_{\op,k_\nu}\,dV_{\omega_\nu}\) uniformly.
On a compact exhaustion, the metrics, sources, and volume forms
converge smoothly. Passing to the limit first on each compact set and
then exhausting \(U\) gives
\[
 \int_U|\Gamma_\infty|_{\op,k_\infty}\,dV_\omega
 \le \liminf_{\nu\to\infty}
       \int_U\mathfrak e_\nu\,dV_{\omega_\nu}<\infty.
\]
The finite volume of \(X\), the finite constant \(c_\infty\), and
equivalence of operator and Hilbert--Schmidt norms imply the claim.

\begin{lem}
\label{lem:carrier-response}
Use the comparison metrics \(k_\nu\) from
Proposition~\ref{prop:determinant-selector-limit}, the error functions
\(\mathfrak e_\nu\) from \eqref{eq:defined-moving-source}, and the
barriers \(b_\nu\) from Proposition~\ref{prop:source-carrier}. Let
\(\Omega_\nu\Subset U\) be connected domains with smooth boundary
containing \(W_1\), and let \(H_\nu^{\mathrm D}\) be smooth positive
Hermitian metrics satisfying
\[
 H_\nu^{\mathrm D}=k_\nu\quad\text{on }\partial\Omega_\nu,\qquad
 \det H_\nu^{\mathrm D}=\det k_\nu,\qquad
 \ii\Lambda_{\omega_\nu}\Psi_{H_\nu^{\mathrm D}}=c_\nu\Id.
\]
Set
\[
 s_\nu=\log(k_\nu^{-1}H_\nu^{\mathrm D}),\qquad
 L_\nu=\int_{W_1}|s_\nu|_{k_\nu}\,dV_{\omega_\nu}.
\]
There is a constant \(C_0\), independent of \(\nu\) and of the
domains, such that
\[
 |s_\nu|_{\op,k_\nu}\le C_0(1+L_\nu)+b_\nu
 \quad\text{on }\Omega_\nu.
\]
For every compact exhaustion \(\mathcal C_m\Subset U\) containing
\(W_0\), along any subsequence with \(L_\nu\to\infty\), the normalized
endomorphisms \(u_\nu=s_\nu/L_\nu\) satisfy
\[
 \lim_{m\to\infty}\limsup_{\nu\to\infty}
 \int_{\Omega_\nu\setminus\mathcal C_m}
 |\Tr(\Gamma_\nu u_\nu)|\,dV_{\omega_\nu}=0.
 \tag{7.31}\label{eq:normalized-pairing-ui}
\]
The same assertions hold with a rational stage fixed and only the
Dirichlet domains varying; the constants may then depend on that stage.
\end{lem}

\begin{proof}
Put \(s_\nu=\log(k_\nu^{-1}H_\nu^{\mathrm D})\) and
\[
 v_\nu=\log\frac{\Tr(e^{s_\nu})+\Tr(e^{-s_\nu})}{2r}.
\]
The Bochner comparison for the two identity maps, recalled before
Proposition~\ref{prop:source-carrier}, gives
\[
 d_{\omega_\nu}^*dv_\nu\le2\mathfrak e_\nu,\qquad v_\nu=0
 \quad\text{on }\partial\Omega_\nu.
\]
On \(W_1\), the coefficients of \(\omega_\nu\) and the sources
\(\mathfrak e_\nu\) have uniform local bounds, with uniform
ellipticity. The local mean-value inequality for the nonnegative
subsolution \(v_\nu\), together with
\(v_\nu\le |s_\nu|_{\op,k_\nu}\), therefore gives
\[
 \sup_{W_0}v_\nu\le C\left(1+\int_{W_1}v_\nu\,dV_{\omega_\nu}\right)
 \le C(1+L_\nu).
\]
The maximum principle applied to
\(v_\nu-C(1+L_\nu)-b_\nu\) outside \(W_0\) gives the same upper bound globally.
The definition of \(v_\nu\) bounds the absolute value of every
eigenvalue of \(s_\nu\) by \(v_\nu+\log(2r)\), proving the asserted
bound on \(|s_\nu|_{\op,k_\nu}\).

Along a subsequence with \(L_\nu\to\infty\), put
\(u_\nu=s_\nu/L_\nu\). For all sufficiently large \(\nu\), division
gives, after enlarging the constants,
\[
 |u_\nu|_{\op,k_\nu}\le C_0+\frac{C_1b_\nu}{L_\nu}.
\]
Choose a compact exhaustion \(\mathcal C_m\Subset U\) containing \(W_0\).
Then
\[
 \begin{aligned}
 \int_{\Omega_\nu\setminus \mathcal C_m}
 |\Tr(\Gamma_\nu u_\nu)|
 &\le rC_0\int_{\Omega_\nu\setminus \mathcal C_m}\mathfrak e_\nu\\
 &\quad+\frac{rC_1}{L_\nu}
 \int_{\Omega_\nu\setminus W_0}\mathfrak e_\nu b_\nu.
 \end{aligned}
\]
First take the limsup along the subsequence with \(L_\nu\to\infty\),
then let \(m\to\infty\).
The second term vanishes by \(L_\nu\to\infty\) and the uniform bound
on \(\int_U\mathfrak e_\nu b_\nu\,dV_{\omega_\nu}\) from
\eqref{eq:carrier-first-moment}.
For each fixed \(t>0\), the set
\(Y\setminus\bigcup_b\{s_b<t\}\) is a compact subset of \(U\), hence
is contained in \(\mathcal C_m\) for all sufficiently large \(m\).
The uniform boundary limit in
Proposition~\ref{prop:local-contraction-estimates} therefore removes
the first term. This proves \eqref{eq:normalized-pairing-ui}.
For fixed \(\nu\), the same proof uses the smooth
coefficients on \(W_1\) and finiteness of the integral of
\(\mathfrak e_\nu(1+b_\nu)\). The integral over the complement tends
to zero as the compact sets exhaust \(U\). The constants are independent of the
Dirichlet domain.
\end{proof}

\subsection{Parabolic degree identities and metric bounds}
\label{subsec:stability-metric-bounds}
We first show that finite-energy weak Higgs-invariant projections
for \(k_\infty\) define parabolic subsheaves whose analytic degrees
equal their parabolic degrees. Stability then bounds the integrals of the
relative logarithms on the fixed domain \(W_1\). Together with
Lemma~\ref{lem:carrier-response}, these bounds give uniform metric
comparison on compact subsets of \(U\).

\begin{prop}
\label{prop:downstairs-reflexive-degree-bridge}
Retain the original regular parabolic Higgs summand
\((E_*,\theta)\) of rank \(r\ge2\) on \((X,D)\), and let
\(k_\infty\) be the limiting comparison metric of
Proposition~\ref{prop:determinant-selector-limit}, viewed on
\(U=X\setminus D\). Write \(\Phi_\infty\) for its contracted
curvature in \eqref{eq:downstairs-contracted-source}.
Let \(p\in W^{1,2}_{\mathrm{loc}}(U,\End E)\) have almost-everywhere
constant rank \(k\), with \(1\le k<r\), and satisfy
\[
 p^2=p=p^{\dagger k_\infty},\qquad
 (\Id-p)\bar\partial p=0,\qquad
 (\Id-p)\theta p=0
\]
almost everywhere. Assume
\[
 \bar\partial_\theta p=\bar\partial p+[\theta,p]
 \in L^2(U,k_\infty,\omega).
\]
Then its image extends uniquely to a nonzero proper saturated reflexive
parabolic Higgs subsheaf \(F_*\subset E_*\), with the parabolic
structure induced from \(E_*\). On the common subbundle locus, \(p\)
is the \(k_\infty\)-orthogonal projection onto \(F\). Moreover,
\[
 \frac1{2\pi}\int_U
 \left(\Tr(p\Phi_\infty)-|\bar\partial_\theta p|^2\right)dV_\omega
 =\pardeg_\omega(F_*),
 \tag{7.32}\label{eq:downstairs-degree-bridge}
\]
and
\[
 \frac1{2\pi}\int_U\Tr\Phi_\infty\,dV_\omega
 =\pardeg_\omega(E_*).
 \tag{7.33}\label{eq:downstairs-ambient-degree}
\]
Both integrals are absolutely convergent. All norms and integrals use
\(k_\infty\) and the fixed Kähler form \(\omega\) on \(X\setminus D\).
\end{prop}

\begin{proof}
Write \(h=k_\infty\) and \(U_X=U\) in this proof. Recall that
\(Z=\operatorname{Sing}(E)\cup Z_{\mathrm{flag}}\subset D\)
is the closed analytic set of codimension at least three in \(X\)
fixed in Theorem~\ref{thm:fixed-terminal-data}. The original parabolic
flags split simultaneously on \(X\setminus Z\).

\noindent\textbf{Step 1. Properties of the comparison metric.}
Over \(X\setminus Z\), the model is the identity and the limiting
weights give the original parabolic filtration. The local formulas of
Lemma~\ref{lem:doubled-reference}, their two-sided comparison under
gluing, and the fact that a fixed logarithmic power is smaller than
every negative power of a divisor coordinate show that \(h\) is
adapted there at every real multi-index.

Put \(L_E=(\bigwedge^rE)^{**}\). By
Proposition~\ref{prop:determinant-selector-limit},
\(\det h=Q_{X,\infty}\). Thus, in every local holomorphic generator
\(\tau_E\) of \(L_E\) on an SNC polydisc,
\[
 \det h(\tau_E,\tau_E)
 =b_h\prod_{j=1}^m|z_j|^{2\mu_{E,j}},
 \qquad b_h>0,
 \tag{7.34}\label{eq:downstairs-canonical-determinant}
\]
where \(\mu_{E,j}\) is the sum, counted with multiplicity, of the
original parabolic weights along \(z_j=0\), and \(b_h\) is smooth
and strictly positive on the whole polydisc, including at \(Z\).
The contracted curvature of \(h\) is integrable by
\eqref{eq:downstairs-contracted-source}.

Outside \(Z\) and the intersections of the components of \(D\),
choose a one-divisor chart \(D_a=\{z=0\}\) and one of the fixed
flag-compatible frames. Let \(\beta_i\) be its original parabolic
weights, and retain the labels \(p_i\) and
\(\bar p=r^{-1}\sum_i p_i\) from Lemma~\ref{lem:doubled-reference}.
Define on the punctured chart
\[
 g_a(e_i,e_j)=\delta_{ij}|z|^{2\beta_i}
              \log(e/|z|)^{2\kappa_i},
 \qquad \kappa_i=-2(p_i-\bar p).
\]
These exponents are fixed by the construction and satisfy
\(\sum_i\kappa_i=0\). The relative Gram matrix estimates in that lemma,
after passing to the limiting weights, give
\[
 C_a^{-1}g_a\le h\le C_ag_a,\qquad
 \bar\partial\log(h^{-1}g_a)\in L^2.
 \tag{7.35}\label{eq:downstairs-power-log-energy}
\]
For the derivative assertion, first use
\(\tau=1-\log|z|^2\) and denote the corresponding diagonal metric
by \(g_a^\tau\). Let \(S_\lambda\) be the limiting relative
Gram endomorphism of the partitioned metric in
Lemma~\ref{lem:doubled-reference}. It and its inverse are bounded;
its covariant normal derivative is bounded by \(C/(|z|\tau)\),
and its tangential derivatives are bounded. Self-adjointness and
metric compatibility give the same bounds for
\(\bar\partial S_\lambda\).
The ratio \(b_a=Q_{X,\infty}/\det g_a^\tau\) extends smoothly
and strictly positively across the chart, since the determinant
powers agree and the logarithmic exponents sum to zero. Thus
\[
 (g_a^\tau)^{-1}h
 =b_a^{1/r}(\det S_\lambda)^{-1/r}S_\lambda.
\]
Write \(T=(g_a^\tau)^{-1}h\). In the displayed formula, the
derivative of \(b_a\) is bounded and
\(\bar\partial\log\det S_\lambda
=\Tr(S_\lambda^{-1}\bar\partial S_\lambda)\).
Thus \(\bar\partial T\) has the same normal and tangential bounds
as \(\bar\partial S_\lambda\). The derivative formula
\[
 \bar\partial\log T
 =\int_0^\infty
 (T+t\Id)^{-1}(\bar\partial T)(T+t\Id)^{-1}\,dt
\]
and the uniform positive lower and upper bounds for \(T\) give
\(|\bar\partial\log T|\le C|\bar\partial T|\).
Since \(\log(h^{-1}g_a^\tau)=-\log T\), its normal derivative is
bounded by \(C/(|z|\tau)\) and its tangential derivatives are bounded.
The normal bound is square integrable because
\[
 \int_0^\epsilon\frac{dr}{r(1-\log r^2)^2}<\infty.
\]
Replacing the factor
\(1-\log|z|^2\) in the construction by \(\log(e/|z|)\) only
changes it by a bounded positive factor whose logarithmic derivative
also has finite \(L^2\) norm. The omitted set is closed analytic of
codimension at least two, as required by the local argument in
Proposition~\ref{prop:power-log-interface}.

We also need a growth bound near \(Z\), where \(E\) may fail to be
locally free or its parabolic flags may fail to split simultaneously.
Fix \(x\in Z\), choose SNC polydiscs \(P'\Subset P\) about \(x\), and write
\(D\cap P=\{f=0\}\), with \(f=z_1\cdots z_m\).
Choose coherent generators \(s_1,\ldots,s_{N_E}\) of \(E|_P\).
The injection
\[
 \pi^*E/\Tor\longrightarrow V
\]
from Proposition~\ref{prop:fixed-model-correction} makes their
pullbacks holomorphic in the fixed frames on \(Y\). Cover
\(\pi^{-1}(\overline P')\) by finitely many smaller resolution charts,
with local boundary equations \(y_{\lambda,b}=0\).
The determinant formula \eqref{eq:downstairs-canonical-determinant}
pulls back to a smooth positive factor times finitely many real
boundary powers. Its identification with a metric on \(L_Y|_U\)
adds only the integral exceptional powers in
\eqref{eq:determinant-line-discrepancy}. Hence the limiting local
normalization formulas in Proposition~\ref{prop:determinant-selector-limit}
still involve only finite powers and logarithmic factors on these
charts. If the partition has \(N\) terms, at least one coefficient
is at least \(1/N\) at each point. The sum of positive Hermitian
forms is then bounded below by \(1/N\) times that local metric,
and its determinant is bounded below by \(N^{-r}\) times the
local determinant. This gives a power lower bound; the finitely many
local upper bounds give a power upper bound for the sum. Refining the
finite cover if necessary, the final determinant normalization therefore
preserves these bounds. Absorbing logarithmic factors into
boundary powers gives a constant \(C_\lambda>0\)
and a growth exponent \(M_\lambda>0\), depending on the chart but not
on the generator index \(i\), such that
\[
 |\pi^*s_i|_h^2\le C_\lambda
               \prod_b|y_{\lambda,b}|^{-M_\lambda}.
\]
Since \(B=(\pi^{-1}D)_{\mathrm{red}}\), we have
\[
 f\circ\pi=u_\lambda\prod_b y_{\lambda,b}^{d_{\lambda,b}},
 \qquad d_{\lambda,b}\ge1,
\]
where \(u_\lambda\) and its reciprocal are bounded on the smaller
chart. Choosing a single sufficiently large power over this finite
cover and using the identification on \(U\), we obtain
\[
 |s_i|_h^2\le C_x\prod_{a=1}^m|z_a|^{-M_x}
 \quad\text{on }P'\cap U_X.
 \tag{7.36}\label{eq:downstairs-coherent-generator-growth}
\]

\noindent\textbf{Step 2. Extension of the projection.}
On a relatively compact coordinate chart in \(U_X\), the local
construction in the proof of \cite[Theorem~0.1.1]{Pop03} turns
the image of \(p\) into a coherent
holomorphic subsheaf, which is a subbundle away from a
codimension-at-least-two analytic subset.

Fix one of the one-divisor charts and the metric \(g_a\) from Step~1.
Let \(q\) be the \(g_a\)-orthogonal projection onto the same
generic image, and put
\[
 S=h^{-1}g_a,\qquad B_p=pSp+\Id-p\in\End E.
\]
The comparison in \eqref{eq:downstairs-power-log-energy} bounds
\(S^{\pm1}\) and \(B_p^{\pm1}\): on the \(h\)-orthogonal
decomposition \(\im p\oplus\ker p\), the latter acts by
\(pSp\) and \(\Id\), respectively. Orthogonality gives
\[
 q=B_p^{-1}pS.
\]
The differentiation formulas used to prove
\eqref{eq:projection-transfer-bound} apply here as well: the product
and inverse rules bound \(\bar\partial q\) by
\(C(|\bar\partial p|+|\bar\partial S|)\), and the exponential
derivative formula bounds \(|\bar\partial S|\) by
\(C|\bar\partial\log S|\). Hence
\[
 |\bar\partial q|_{g_a}
 \le C\left(
 |\bar\partial p|_h+
 |\bar\partial\log(h^{-1}g_a)|
 \right).
 \tag{7.37}\label{eq:downstairs-projection-transfer}
\]
The right side belongs to \(L^2\). Fubini's theorem therefore gives
finite normal-disk energy for almost every tangential parameter.

Apply the normal-disk argument in the proof of
Proposition~\ref{prop:power-log-interface}. The curvature of \(g_a\)
and the squared normal derivative of \(q\) are integrable on almost
every such disk. Circular averages of the induced determinant norm
then give finite Fubini--Study energy for the Grassmannian map, so its
puncture is removable by \cite[Theorem~3.6]{SacksUhlenbeck1981}.
Shiffman's separate-meromorphicity theorem gives meromorphic Plücker
ratios across each one-divisor chart. Meromorphic Hartogs extension
across the remaining analytic sets of codimension at least two,
including \(Z\), gives a rank-\(k\) meromorphic generic subspace
\[
 M\subset E\otimes\mathcal M_X
\]
on all of \(X\).

Choose local meromorphic generators of \(M\), clear their denominators,
and let \(F_0\subset E\) be the resulting coherent image. Define
\[
 F=\ker\left(
 E\longrightarrow(E/F_0)/\Tor(E/F_0)
 \right).
\]
Then \(Q=E/F\) is torsion-free. At a height-one point the regular local
ring is a discrete valuation ring, so \(F\) is free there. At a point of
local dimension at least two, reflexivity gives
\(\operatorname{depth}E\ge2\), while torsion-freeness gives
\(\operatorname{depth}Q\ge1\). The depth lemma yields
\(\operatorname{depth}F\ge2\). Thus \(F\) is torsion-free and \(S_2\),
hence reflexive. It has rank \(k\), so it is nonzero and proper.

Higgs invariance is equivalent to vanishing of the induced morphism
\[
 F\longrightarrow Q\otimes\Omega_X^1(\log D).
\]
Its target is torsion-free, and
the Higgs-invariance assumption on \(p\) makes it vanish on a dense open
set. It therefore vanishes identically. Intersect \(F\) with every
one-divisor filtration term and saturate along that divisor. The
successive quotients are torsion-free. Intersecting the resulting
one-divisor lattices inside \(F(*D)\) and taking their reflexive
saturations gives the induced parabolic Higgs subsheaf \(F_*\), with
lattices defined at every real multi-index. Saturation along the divisors
and reflexive extension determine these lattices uniquely and preserve
their compatibility.

\noindent\textbf{Step 3. The degree identities.}
Put \(s=\rk F\) and \(L_F=(\bigwedge^sF)^{**}\), and define
\[
 \kappa_h=\ii\Tr\bigl(p\Lambda_\omega
 (F_h+[\theta,\theta^{\dagger h}])\bigr)
 -|\bar\partial_\theta p|^2.
\]
On the common
subbundle locus let \(\ell_h\) be the determinant metric induced by
\(h\). The projection has bounded operator norm, so
\eqref{eq:downstairs-contracted-source} and the assumed \(L^2\) bound on
\(\bar\partial_\theta p\) show that
\(\kappa_h\in L^1\). Higgs Gauss--Codazzi gives
\[
 \ii\Lambda_\omega F_{\ell_h}=\kappa_h.
 \tag{7.38}\label{eq:downstairs-gauss-codazzi}
\]

Choose a smooth metric \(k_F\) on \(L_F\) and smooth divisor metrics
\(k_a\). Let \(\mu_{F,a}\) be the sum of the induced parabolic weights
of \(F\) along \(D_a\), counted with multiplicity and without reduction
modulo one. Set
\[
 q_F=k_F\prod_a|\sigma_a|_{k_a}^{2\mu_{F,a}}.
\]
A direct Chern-connection calculation shows that its curvature on
\(X\setminus D\) extends as the smooth representative of
\(\parc_1(F_*)\). Hence
\[
 \frac{\ii}{2\pi}\int_{X\setminus D}
 F_{q_F}\wedge\frac{\omega^{n-1}}{(n-1)!}
 =\pardeg_\omega(F_*).
 \tag{7.39}\label{eq:downstairs-raw-factor-degree}
\]

Put \(\phi_F=\log(\ell_h/q_F)\). On a split SNC chart in
\(X\setminus Z\), let \(d\) be the diagonal monomial metric
with the original parabolic weights, and put
\[
 \ell_d=\det(F,d),\qquad
 v=\log(\ell_h/\ell_d),\qquad
 \phi_F=v+\log(\ell_d/q_F).
\]
Adaptedness and
\eqref{eq:downstairs-canonical-determinant} give
\[
 |v|\le C_\epsilon+\epsilon\sum_a-\log|z_a|^2
 \quad\text{for every }\epsilon>0.
\]
We now apply the two local arguments in the proof of
Proposition~\ref{prop:power-log-interface}. First, the positive-current
argument for \(\log\ell_d\) identifies its prescribed divisorial
coefficients. After their removal, \(\log(\ell_d/q_F)\) is
locally \(L^2\) and satisfies its contracted curvature equation
with an \(L^1\) density and no additional distribution on the
divisor or the non-subbundle locus. Second, the scalar cutoff
argument applies to \(v\): its curvature equation has an \(L^1\)
right side, and the displayed arbitrarily small logarithmic bound
removes the cutoff errors. Adding the two equations proves
\(\phi_F\in L^2_{\mathrm{loc}}(X\setminus Z)\) and
\[
 \ii\Lambda_\omega\bar\partial\partial\phi_F
 =\kappa_h-\ii\Lambda_\omega F_{q_F}
\]
distributionally on \(X\setminus Z\), with the \(L^1\) densities
extended from the common subbundle locus in \(U_X\).

It remains to extend the distributional equation across \(Z\subset D\). Fix
\(x\in Z\), choose the generators in
\eqref{eq:downstairs-coherent-generator-growth}, and choose meromorphic
frames \(f_1,\ldots,f_s\) of \(F\) and meromorphic lifts
\(q_1,\ldots,q_{r-s}\) to \(E\) of a meromorphic frame of
\(E/F\). After multiplying by finitely many
nonzero holomorphic denominators, all these sections become holomorphic
linear combinations of the \(s_i\). Hadamard's inequality and
\eqref{eq:downstairs-coherent-generator-growth} give an upper bound for
the determinant norm of \(f_1\wedge\cdots\wedge f_s\) by a finite product
of negative powers of the divisor coordinates and those denominators. For the
lower bound, apply Gram--Schmidt to the combined meromorphic frame. Its
ambient determinant is
\[
 |\det(f_1,\ldots,f_s,q_1,\ldots,q_{r-s})|^2\det h.
\]
The quotient determinant is at most
\(\prod_j|q_j|_h^2\). Hence the ambient determinant formula
\eqref{eq:downstairs-canonical-determinant} and the same generator bounds
give a lower bound for the factor determinant by another finite product
of powers of nonzero holomorphic functions. Changing from this
meromorphic determinant frame to a local holomorphic generator of
\(L_F\) adds only logarithms of absolute values of meromorphic
functions. Absorbing these terms and the prescribed powers in
\(q_F\), we obtain nonzero holomorphic functions \(g_{u,x}\)
and \(C_x'<\infty\) such that
\[
 |\phi_F|\le C_x'\left(
 1+\sum_u\log\frac e{|g_{u,x}|}
 \right).
\]
Holomorphic logarithms belong locally to \(L^2\), so
\(\phi_F\in L^2_{\mathrm{loc}}\) at \(Z\).

Take a finite smooth stratification of \(Z\) with the frontier
condition, and proceed in decreasing stratum dimension. At each
stage take compactly supported test functions away from the closed
union of lower-dimensional strata. On their supports, the strata
being removed are locally closed smooth submanifolds, and finitely
many tubular charts suffice. For a stratum of complex codimension
\(q\ge3\), choose a normal cutoff
\(\xi_\delta\) which is zero at radius at most \(\delta\), one at radius
at least \(2\delta\), and satisfies
\[
 |d\xi_\delta|\le C\delta^{-1},\qquad
 |d^2\xi_\delta|\le C\delta^{-2}.
\]
Its derivative support has volume \(O(\delta^{2q})\). Testing the
distributional equation off the stratum against a smooth test function
times \(\xi_\delta\), Cauchy--Schwarz and the logarithmic bound on \(\phi_F\) bound the two cutoff errors by
\[
 C\delta^{q-1}\|\phi_F\|_{L^2},
 \qquad
 C\delta^{q-2}\|\phi_F\|_{L^2}.
\]
Both tend to zero. The \(L^1\) source passes to the limit by absolute
continuity. This extends the equation across the strata of the
current dimension; descending through the remaining dimensions
extends it to all of \(X\).

Test this global distributional equation against the constant function
one. The integral of the contracted curvature difference between
\(\ell_h\) and \(q_F\) is zero. Equations
\eqref{eq:downstairs-gauss-codazzi} and
\eqref{eq:downstairs-raw-factor-degree} now give
\eqref{eq:downstairs-degree-bridge}.

For the ambient identity, define \(q_E\) from a smooth metric on \(L_E\)
and the determinant exponents \(\mu_{E,a}\). Equation
\eqref{eq:downstairs-canonical-determinant} says that
\(\det h/q_E\) extends as a smooth strictly positive function across all
of \(X\). Its curvature difference is
globally exact, so compact Stokes gives
\eqref{eq:downstairs-ambient-degree}. The trace of the Higgs commutator
vanishes, and \eqref{eq:downstairs-contracted-source} guarantees absolute
integrability. This completes the proof.
\end{proof}

\begin{prop}
\label{prop:normalized-tail-no-escape}
Retain the comparison metrics \(k_\nu\), the K\"ahler forms
\(\omega_\nu\), and the fixed domain \(W_1\) from
Lemma~\ref{lem:carrier-response}, allowing a subsequence of the rational
stages. Suppose the original summand \((E_*,\theta)\), of rank
\(r\ge2\), is \(\omega\)-stable. Let
\(\Omega_\nu\Subset U\) be connected domains with smooth boundary
containing \(W_1\), such that every compact subset of \(U\) is
eventually contained in \(\Omega_\nu\). For the determinant-fixed
Hermitian--Einstein Dirichlet solutions \(H_\nu^{\mathrm D}\) with
boundary value \(k_\nu\), as in that lemma, we have
\[
 \sup_\nu L_\nu
 =\sup_\nu\int_{W_1}
   |\log(k_\nu^{-1}H_\nu^{\mathrm D})|_{k_\nu}
   \,dV_{\omega_\nu}<\infty.
\]
For each \(\nu\) for which \(V_{\nu,*}\) is stable, and any such exhaustion
\(\Omega_k\) and the corresponding solutions \(H_{\nu,k}\) with
boundary value \(k_\nu\), we also have
\[
 \sup_k\int_{W_1}
   |\log(k_\nu^{-1}H_{\nu,k})|_{k_\nu}
   \,dV_{\omega_\nu}<\infty,
\]
with the bound allowed to depend on \(\nu\).
\end{prop}

\begin{proof}
\noindent\emph{Varying rational stages.}
We first prove the bound as \(\nu\to\infty\), using the stability of
\(E_*\) on \((X,\omega)\). On the common open set
\(U=Y\setminus B=X\setminus D\), we have
\(\omega_\nu\to\omega\) and \(k_\nu\to k_\infty\) smoothly on
compact subsets. Put \(s_\nu=\log(k_\nu^{-1}H_\nu^{\mathrm D})\).

The determinant and boundary conditions give \(\Tr s_\nu=0\) and
\(s_\nu=0\) on \(\partial\Omega_\nu\). The integrated Donaldson
identity \cite[Proposition~2.6]{ZZZ18}, applied on each compact domain
with boundary, therefore gives
\[
 \int_{\Omega_\nu}\Tr(\Gamma_\nu s_\nu)
 +\int_{\Omega_\nu}
 \langle\Psi_{\exp}(s_\nu)\bar\partial_\theta s_\nu,
 \bar\partial_\theta s_\nu\rangle=0,
 \tag{7.40}\label{eq:moving-donaldson-identity}
\]
where the integrals use \(k_\nu\) and \(dV_{\omega_\nu}\).
Lemma~\ref{lem:carrier-response} and the compact bounds on \(b_\nu\)
give, on every \(\mathcal C\Subset U\),
\[
 \sup_{\mathcal C}|s_\nu|\le A_{0,\mathcal C}+A_{1,\mathcal C}L_\nu,
 \tag{7.41}\label{eq:linear-normalization-control}
\]
with constants independent of the sequence index. The same lemma gives
the tail estimate \eqref{eq:normalized-pairing-ui} along any subsequence
with \(L_\nu\to\infty\). It remains to prove
\[
 \sup_\nu L_\nu<\infty.
 \tag{7.42}\label{eq:no-escape-core}
\]
If this fails, pass to a subsequence with \(L_\nu\to\infty\) and put
\(u_\nu=s_\nu/L_\nu\). Then
\[
 \Tr u_\nu=0,\qquad
 \int_{W_1}|u_\nu|_{k_\nu}\,dV_{\omega_\nu}=1,
\]
and \eqref{eq:linear-normalization-control} bounds \(u_\nu\) on every compact
set. Divide \eqref{eq:moving-donaldson-identity} by \(L_\nu\). Its positive
kernel is
\[
 \mathcal K_{L_\nu}(x,y)
 =L_\nu\Psi_{\exp}(L_\nu x,L_\nu y)
 =\int_0^{L_\nu}e^{t(y-x)}\,dt.
\]
The tail condition \eqref{eq:normalized-pairing-ui} and compact convergence
of \(\Gamma_\nu\) bound the source pairing. On a compact spectral square,
\(\mathcal K_{L_\nu}\) has a positive lower bound independent of \(\nu\).
Hence the \(\bar\partial\)-derivatives are bounded in \(L^2\)
on every compact subset. Self-adjointness and the smoothly converging
metric coefficients control the complementary derivatives as well,
so \(u_\nu\) is bounded in \(W^{1,2}\) there. Rellich
compactness and a diagonal extraction give a nonzero self-adjoint
trace-free limit \(u_\infty\), strongly in \(L^2_{\mathrm{loc}}\) and
weakly in \(W^{1,2}_{\mathrm{loc}}\).

The normalization on \(W_1\) ensures that this limit is not zero.
On every compact subset,
strong \(L^2\) convergence and smooth convergence of the coefficients
give convergence of the source integrals. Condition
\eqref{eq:normalized-pairing-ui} controls their complements uniformly,
so
\[
 \int_{\Omega_\nu}\Tr(\Gamma_\nu u_\nu)\,dV_{\omega_\nu}
 \longrightarrow
 \int_U\Tr(\Gamma_\infty u_\infty)\,dV_\omega,
\]
and the limiting scalar pairing is absolutely integrable. Throughout
this part of the proof, limiting norms and inner products use
\(k_\infty\) and \(\omega\).

For each fixed \(R>0\), the kernels \(\mathcal K_{L_\nu}\)
dominate \(\mathcal K_R\) for all sufficiently large \(\nu\).
Weak lower semicontinuity on compact subsets, followed by exhaustion,
therefore bounds
\[
 \int_U\left\langle
 \mathcal K_R(u_\infty)\bar\partial_\theta u_\infty,
 \bar\partial_\theta u_\infty\right\rangle
 dV_\omega
 \le-\int_U\Tr(\Gamma_\infty u_\infty)\,dV_\omega.
\]
The bound is independent of \(R\). As \(R\to\infty\), the kernel
increases to \((x-y)^{-1}\) when \(x>y\) and to infinity when
\(x\le y\). Thus the components of \(\bar\partial_\theta u_\infty\)
mapping an \(x\)-eigenspace to a \(y\)-eigenspace with \(x\le y\)
vanish. In particular,
\(\bar\partial\Tr(u_\infty^j)=0\) for \(1\le j\le r\).
These traces are real, hence constant on connected \(U\); Newton's
identities give a globally constant finite spectrum
\[
 \lambda_1<\cdots<\lambda_{m_{\mathrm{spec}}},
\]
The trace-free normalization and the nonzero limit give
\(m_{\mathrm{spec}}\ge2\). The projections \(p_\alpha\) onto eigenspaces
with eigenvalues at most \(\lambda_\alpha\),
\(1\le\alpha<m_{\mathrm{spec}}\), satisfy the weak
holomorphic and Higgs equations and
\[
 \int_U|\bar\partial_\theta p_\alpha|^2<\infty.
\]
The spectral projection calculation in the proof of
Theorem~\ref{thm:fixed-stage-he-metrics} now applies. Put
\(\delta_\alpha=\lambda_{\alpha+1}-\lambda_\alpha\).
A block from an \(x\)-eigenspace to a \(y\)-eigenspace with \(y<x\)
contributes to exactly the projections whose cut lies between \(y\)
and \(x\). The sum of those spectral gaps is \(x-y\), so its
coefficient in \(\sum_\alpha\delta_\alpha|D''p_\alpha|^2\) is
\((x-y)/(x-y)^2=(x-y)^{-1}\). We therefore obtain
\[
 \int_U\Tr(\Gamma_\infty u_\infty)
 +\sum_{\alpha=1}^{m_{\mathrm{spec}}-1}\delta_\alpha
 \int_U|\bar\partial_\theta p_\alpha|^2\le0.
 \tag{7.43}\label{eq:spectral-limit-inequality}
\]
Proposition~\ref{prop:downstairs-reflexive-degree-bridge}, applied to
\(k_\infty\) on \(X\setminus D\), identifies the degrees of these
projections. Their images extend to nonzero proper saturated parabolic
Higgs subsheaves \(F_{\alpha,*}\subset E_*\), and
\[
 \frac1{2\pi}\int_U
 \left(\ii\Tr(p_\alpha\Lambda_\omega\Psi_{k_\infty})
       -|\bar\partial_\theta p_\alpha|^2\right)dV_\omega
 =\pardeg_\omega(F_{\alpha,*}).
\]
Taking traces of the stage equations and using the determinant
condition gives
\(\ii\Lambda_\omega F_{\det k_\infty}=rc_\infty\) on \(U\).
The corresponding ambient degree identity yields
\[
 c_\infty\Vol_\omega(X)
 =2\pi\mu_\omega(E_*).
\]
We have
\[
 u_\infty=\lambda_{m_{\mathrm{spec}}}\Id
 -\sum_{\alpha=1}^{m_{\mathrm{spec}}-1}
  \delta_\alpha p_\alpha,\qquad
 r\lambda_{m_{\mathrm{spec}}}
 =\sum_{\alpha=1}^{m_{\mathrm{spec}}-1}
  \delta_\alpha\rk F_\alpha.
\]
Substitution of these degree identities into
\eqref{eq:spectral-limit-inequality} makes its left side
\[
 2\pi\sum_{\alpha=1}^{m_{\mathrm{spec}}-1}
 \delta_\alpha\rk(F_\alpha)
 \bigl[\mu_\omega(E_*)-\mu_\omega(F_{\alpha,*})\bigr].
\]
It is strictly positive by stability, a contradiction. This proves
\eqref{eq:no-escape-core}. Inserting it into
\eqref{eq:linear-normalization-control} gives uniform metric comparison
on compact subsets.

\medskip
\noindent\emph{Exhaustion with \(\nu\) fixed.}
Fix a stable stage \(\nu\), and write
\[
 s_{\nu,k}=\log(k_\nu^{-1}H_{\nu,k}),\qquad
 L_{\nu,k}=\int_{W_1}|s_{\nu,k}|_{k_\nu}\,dV_{\omega_\nu}.
\]
Suppose \(L_{\nu,k}\to\infty\) along a subsequence. Apply the preceding
normalization and compactness argument to
\(u_{\nu,k}=s_{\nu,k}/L_{\nu,k}\), now keeping
\(k_\nu\), \(\omega_\nu\), \(\Gamma_\nu\), and \(c_\nu\) fixed while
\(k\to\infty\). The compact estimates and the normalized pairing
estimate of Lemma~\ref{lem:carrier-response} apply with constants
allowed to depend on \(\nu\). They give a nonzero self-adjoint
trace-free limit \(u_{\nu,\infty}\) with constant spectrum. Write
\(m_{\mathrm{spec}}\ge2\) for its number of distinct eigenvalues.
Its proper lower spectral projections \(p_\alpha\) satisfy the weak
holomorphic and Higgs equations and have finite
\(L^2(U,k_\nu,\omega_\nu)\) energy.
The same kernel argument yields
\[
 \int_U\Tr(\Gamma_\nu u_{\nu,\infty})\,dV_{\omega_\nu}
 +\sum_{\alpha=1}^{m_{\mathrm{spec}}-1}\delta_\alpha
 \int_U|\bar\partial_\theta p_\alpha|_{k_\nu,\omega_\nu}^2
       \,dV_{\omega_\nu}\le0,
\]
where \(\delta_\alpha>0\) are the successive spectral gaps of
\(u_{\nu,\infty}\).

To identify these degrees on \(Y\), use
Proposition~\ref{prop:power-log-interface} for \(k_\nu\).
Lemma~\ref{lem:doubled-reference} gives adaptedness, the power-log
comparison, and the required \(L^2\) bound on the relative logarithmic
derivative. Proposition~\ref{prop:determinant-selector-limit} provides
the determinant formula, and
Proposition~\ref{prop:local-contraction-estimates} gives integrability
of the contracted curvature. Thus the projections extend to nonzero
proper saturated parabolic Higgs subsheaves
\(F_{\alpha,*}\subset V_{\nu,*}\), with
\[
 \frac1{2\pi}\int_U
 \left(\ii\Tr(p_\alpha\Lambda_{\omega_\nu}\Psi_{k_\nu})
       -|\bar\partial_\theta p_\alpha|_{k_\nu,\omega_\nu}^2\right)
       dV_{\omega_\nu}
 =\pardeg_{\omega_\nu}(F_{\alpha,*}).
\]
The ambient degree identity is
\(c_\nu\Vol_{\omega_\nu}(Y)=2\pi\mu_{\omega_\nu}(V_{\nu,*})\).
Substituting these identities into the preceding inequality gives
\[
 2\pi\sum_{\alpha=1}^{m_{\mathrm{spec}}-1}
 \delta_\alpha\rk(F_\alpha)
 \bigl[\mu_{\omega_\nu}(V_{\nu,*})
       -\mu_{\omega_\nu}(F_{\alpha,*})\bigr]\le0.
\]
This contradicts the \(\omega_\nu\)-stability of \(V_{\nu,*}\).
Hence \(\sup_kL_{\nu,k}<\infty\), proving the second assertion.
\end{proof}

\subsection{Compactness and parabolic growth of the global solutions}
\label{subsec:global-compactness-growth}
We now return to the global Hermitian--Einstein metrics \(h_\nu\)
on \(U\) constructed in Theorem~\ref{thm:fixed-stage-he-metrics}
and used in Corollary~\ref{cor:factor-energy}. We apply the preceding
estimates through auxiliary Dirichlet solutions to prove subsequential
convergence after constant rescaling, together with the boundary growth
bounds needed for Theorem~\ref{thm:harmonic-main}.

\begin{prop}\label{prop:moving-endpoints}
Fix one original stable summand \((E_*,\theta)\), and retain its
fixed model and sufficiently late stable rational stages from
Setup~\ref{setup:literal-moving-schedule}. Let \(h_\nu\) be the global
Hermitian--Einstein metrics of Theorem~\ref{thm:fixed-stage-he-metrics}.
With \(a_\nu\) and \(k_\nu\) from
Proposition~\ref{prop:determinant-selector-limit}, put
\[
 h_\nu^{\mathrm{act}}=a_\nu^{1/r}h_\nu.
\]
Then the following hold.

\begin{enumerate}
\item The sequence \(h_\nu^{\mathrm{act}}\) has a subsequence converging
in \(C^\infty_{\mathrm{loc}}(U)\) to a smooth positive Hermitian metric.

\item On every relatively compact SNC chart \(P\subset Y\) whose
closure avoids the exceptional divisors, write
\(B\cap P=\{y_1\cdots y_q=0\}\). For every \(\eta>0\), there is
\(C_{P,\eta}\ge1\), independent of \(\nu\), such that
\[
 C_{P,\eta}^{-1}\prod_{a=1}^q|y_a|^{2\eta}k_\nu
 \le h_\nu^{\mathrm{act}}
 \le C_{P,\eta}\prod_{a=1}^q|y_a|^{-2\eta}k_\nu
 \quad\text{on }P\setminus B.
 \tag{7.44}\label{eq:stage-growth-comparison}
\]
\end{enumerate}
\end{prop}

\begin{proof}
\noindent\textbf{Step 1. Determinant normalization.}
By Theorem~\ref{thm:fixed-stage-he-metrics} and
Proposition~\ref{prop:determinant-selector-limit},
\[
 \det h_\nu^{\mathrm{act}}=a_\nu q_\nu
 =Q_{Y,\nu}=\det k_\nu.
\]
Multiplication by a spatially constant positive scalar leaves the
Chern connection and Higgs
adjoint unchanged, as well as the curvature energy used in
Corollary~\ref{cor:factor-energy} and every growth lattice.
In particular, \(h_\nu^{\mathrm{act}}\) is adapted and solves
\[
 \ii\Lambda_{\omega_\nu}
 \bigl(F_{h_\nu^{\mathrm{act}}}
 +[\theta,\theta^{\dagger h_\nu^{\mathrm{act}}}]\bigr)
 =c_\nu\Id_V.
\]

If \(r=1\), both \(h_\nu^{\mathrm{act}}\) and \(k_\nu\) are metrics on the
same line with determinant \(Q_{Y,\nu}\), so
\(h_\nu^{\mathrm{act}}=k_\nu\). Proposition~\ref{prop:determinant-selector-limit}
then gives smooth convergence on compact sets, and the growth estimate
is immediate from \(h_\nu^{\mathrm{act}}=k_\nu\).
We henceforth assume \(r\ge2\).

Let \(b_\nu\) and \(W_0\Subset W_1\) be as in
Proposition~\ref{prop:source-carrier}, and put
\[
 s_\nu^{\mathrm{act}}=\log(k_\nu^{-1}h_\nu^{\mathrm{act}}),
 \qquad
 \ell_\nu=\int_{W_1}|s_\nu^{\mathrm{act}}|_{k_\nu}\,dV_{\omega_\nu}.
\]
We first prove
\[
 \sup_\nu\ell_\nu<\infty,
 \tag{7.45}\label{eq:core-bound}
\]
and then obtain, with \(C\) independent of \(\nu\),
\[
 |s_\nu^{\mathrm{act}}|_{\op,k_\nu}\le C+b_\nu
 \quad\text{on }U.
 \tag{7.46}\label{eq:global-endpoint-bound}
\]
The compact bounds and boundary growth of \(b_\nu\) will then give
the two conclusions of the proposition.

\noindent\textbf{Step 2. Dirichlet solutions for fixed \(\nu\).}
Fix \(\nu\). To compare the global metric \(h_\nu^{\mathrm{act}}\)
with \(k_\nu\), choose connected domains with smooth boundary
\[
 \Omega_1\Subset\Omega_2\Subset\cdots\Subset U_Y
\]
containing \(W_1\) and exhausting \(U_Y\). Let \(H_{\nu,k}\) solve
\[
 H_{\nu,k}=k_\nu\text{ on }\partial\Omega_k,\qquad
 \det H_{\nu,k}=Q_{Y,\nu},
\]
\[
 \ii\Lambda_{\omega_\nu}
 \bigl(F_{H_{\nu,k}}+[\theta,\theta^{\dagger H_{\nu,k}}]\bigr)
 =c_\nu\Id_V.
\]
Existence and uniqueness follow from \cite[Propositions~3.4 and Theorem~5.1]{ZZZ18}. Put
\[
 s_{\nu,k}=\log(k_\nu^{-1}H_{\nu,k}),\qquad
 L_{\nu,k}=\int_{W_1}|s_{\nu,k}|_{k_\nu}\,dV_{\omega_\nu}.
\]

Lemma~\ref{lem:carrier-response} and the compact bounds on \(b_\nu\)
give, for each \(\mathcal C\Subset U_Y\),
\[
 \sup_{\mathcal C}|s_{\nu,k}|_{k_\nu}
 \le A_{0,\nu,\mathcal C}+A_{1,\nu,\mathcal C}L_{\nu,k}.
 \tag{7.47}\label{eq:mean-to-sup}
\]
The second assertion of
Proposition~\ref{prop:normalized-tail-no-escape} applies to these Dirichlet
solutions and gives
\[
 \sup_kL_{\nu,k}<\infty.
\]
The interior estimate \cite[Proposition~3.5]{ZZZ18} and elliptic
bootstrapping then give subsequences converging smoothly on compact sets.

\noindent\textbf{Step 3. Identifying the exhaustion limit.}
Lemma~\ref{lem:carrier-response} gives
\[
 |s_{\nu,k}|_{\op,k_\nu}
 \le C_\nu(1+L_{\nu,k})+b_\nu.
 \tag{7.48}\label{eq:fixed-stage-carrier-bound}
\]
For any subsequence converging smoothly on compact sets, denote its limit temporarily
by \(\widetilde h_\nu\). Passing to the limit in
\eqref{eq:fixed-stage-carrier-bound} and using the uniform bound on
\(L_{\nu,k}\) at fixed \(\nu\) gives
\[
 |\log(k_\nu^{-1}\widetilde h_\nu)|_{\op,k_\nu}
 \le C_\nu'+b_\nu.
 \tag{7.49}\label{eq:limit-carrier-bound}
\]

By \eqref{eq:carrier-growth}, at fixed \(\nu\) the barrier is bounded
by a constant depending on \(\nu\) plus a fixed multiple of
\(\sum_a\log(1-\log|y_a|^2)\). This sum grows more slowly than every
positive multiple of \(\sum_a\log(1/|y_a|)\). Thus
\eqref{eq:limit-carrier-bound} bounds \(\widetilde h_\nu\) above and below
by \(k_\nu\) times arbitrarily small positive and negative powers of
the divisor coordinates, including on exceptional charts.
Since \(k_\nu\) is adapted, splitting the exponent tolerance between
the two inequalities proves that \(\widetilde h_\nu\) is adapted to the same
lattices at every real index.

The parabolic Higgs bundle \(V_{\nu,*}\) is stable and rational,
and \(h_\nu^{\mathrm{act}}\) is adapted by Step~1. Both determinant
metrics are \(Q_{Y,\nu}\), whose local formula has the prescribed rational
exponents and a smooth positive coefficient, and both metrics solve
the central equation with constant \(c_\nu\).
Theorem~\ref{thm:central-he-same-lattice} therefore gives
\[
 \widetilde h_\nu=h_\nu^{\mathrm{act}}.
\]
\noindent\textbf{Step 4. A bound independent of \(\nu\).}
Fix a compact exhaustion \(\mathcal C_m\Subset U_Y\), independent
of \(\nu\).
Smooth convergence to \(h_\nu^{\mathrm{act}}\) on \(W_1\) gives
\[
 L_{\nu,k}\longrightarrow\ell_\nu
\]
along the convergent subsequence chosen for this \(\nu\).
Suppose \eqref{eq:core-bound} fails. Choose \(\nu_m\uparrow\infty\)
with \(\ell_{\nu_m}\to\infty\). For each \(m\), choose \(k_m\)
from the corresponding subsequence so large that
\[
 \mathcal C_m\subset\Omega_{k_m},\qquad
 |L_{\nu_m,k_m}-\ell_{\nu_m}|<1,
 \tag{7.50}\label{eq:two-index-selector}
\]
Every compact subset of \(U\) eventually lies in these domains, and
\(L_{\nu_m,k_m}\to\infty\).

The first assertion of
Proposition~\ref{prop:normalized-tail-no-escape}, using the original
\(\omega\)-stability of \(E_*\), applies to
\[
 (\Omega_{k_m},\omega_{\nu_m},k_{\nu_m},H_{\nu_m,k_m}).
\]
It gives \(\sup_m L_{\nu_m,k_m}<\infty\), contradicting
\eqref{eq:two-index-selector} and \(\ell_{\nu_m}\to\infty\).
This proves \eqref{eq:core-bound}.

\noindent\textbf{Step 5. Global estimate and compactness.}
Lemma~\ref{lem:carrier-response} gives, with \(C_0\) independent of
\(\nu\) and \(k\),
\[
 |s_{\nu,k}|_{\op,k_\nu}
 \le C_0(1+L_{\nu,k})+b_\nu.
\]
For fixed \(\nu\), pass to \(h_\nu^{\mathrm{act}}\), then use
\eqref{eq:core-bound}. This gives \eqref{eq:global-endpoint-bound}.
Since \(b_\nu\) is uniformly bounded on every \(\mathcal C\Subset U\),
\[
 C_{\mathcal C}^{-1}k_\nu\le h_\nu^{\mathrm{act}}
 \le C_{\mathcal C}k_\nu\quad\text{on }\mathcal C,
\]
with \(C_{\mathcal C}\) independent of \(\nu\).
The metrics \(k_\nu\) converge smoothly to the positive metric
\(k_\infty\), and \(\omega_\nu\to\omega\) smoothly on compact
subsets of \(U\). The interior estimates used in Step~2 therefore
give a subsequence converging smoothly on compact subsets; the lower
comparison bound makes its limit positive. This proves the first
assertion. The Chern curvatures and Higgs adjoints converge as well.

\noindent\textbf{Step 6. Nonexceptional growth.}
On a strict-transform chart whose closure avoids the exceptional locus,
\eqref{eq:carrier-growth} and the elementary bound
\(\log(1+2R)\le C_\epsilon+\epsilon R\), valid for \(R\ge0\) and
every \(\epsilon>0\), give, for every \(\eta>0\),
\[
 b_\nu\le C_{P,\eta}+2\eta\sum_a\log\frac1{|y_a|},
\]
with \(C_{P,\eta}\) independent of \(\nu\). Combining this with
\eqref{eq:global-endpoint-bound} and exponentiating gives
\eqref{eq:stage-growth-comparison}.
\end{proof}

% \begin{rem}\label{rem:moving-endpoints}
% Finite-domain metrics serve as selectors. At each fixed stage,
% Theorem~\ref{thm:central-he-same-lattice} identifies their
% exhaustion limit, obtained smoothly on compact sets, with the pre-existing actual endpoint by
% checking adaptedness to the same rational lattice, the common
% determinant metric \(Q_{Y,\nu}\), and equality of the central equation. A two-index
% diagonal then transfers the normalized-log no-escape estimate to this
% same family of actual endpoints. Zeroth-order growth comparison is not
% used as a substitute for a first-order flux calculation.
% \end{rem}

\begin{cor}\label{cor:nonexceptional-growth}
Retain one stable summand and the normalized global metrics of
Proposition~\ref{prop:moving-endpoints}. Suppose a subsequence satisfies
\[
 h_{\nu_m}^{\mathrm{act}}\longrightarrow H_Y
\]
smoothly on compact subsets in the fixed local holomorphic frames. Under the identification
\[
 Y\setminus B\simeq X\setminus D
\]
we also denote this limit by \(H_X\) on \(X\setminus D\). Retain the limiting comparison
metric \(k_\infty\) from
Proposition~\ref{prop:determinant-selector-limit}, viewed on this
common open set. Then, on every relatively compact SNC chart
\(P\Subset X\setminus Z\), with
\[
 D\cap P=\{z_1\cdots z_\ell=0\},\qquad |z_a|<1,
\]
and for every \(\eta>0\), there is \(C_{P,\eta}\ge1\) such that
\[
 C_{P,\eta}^{-1}
 \prod_{a=1}^\ell|z_a|^{2\eta}k_\infty
 \le H_X\le
 C_{P,\eta}
 \prod_{a=1}^\ell|z_a|^{-2\eta}k_\infty
 \quad\text{on }P\setminus D.
 \tag{7.51}\label{eq:limit-growth}
\]
In particular, for every real local multi-index \(\mathbf c\),
\[
 P^{\mathbf c}(H_X)|_{X\setminus Z}
 =P^{\mathbf c}(k_\infty)|_{X\setminus Z}
 =E^{\mathbf c}|_{X\setminus Z}.
\]
\end{cor}

\begin{proof}
Since \(\pi\) is an isomorphism over \(X\setminus Z\), the chart
\(P\) identifies with an SNC chart on \(Y\) whose closure avoids
the exceptional divisors. Apply \eqref{eq:stage-growth-comparison}
with exponent \(\eta\). Its constant is independent of the stage,
so the compact convergence of \(h_{\nu_m}^{\mathrm{act}}\) and
\(k_{\nu_m}\) allows passage to the limit at every point of
\(P\setminus D\), giving \eqref{eq:limit-growth} with the same constant.

Step~1 of the proof of
Proposition~\ref{prop:downstairs-reflexive-degree-bridge} shows that
\(k_\infty\) is adapted to \(E_*\) on \(X\setminus Z\). That local
calculation also applies in rank one.

Finally, \eqref{eq:limit-growth} holds for every \(\eta>0\).
Taking square roots and using \(\eta=\epsilon/2\), together with
the growth bound of a section with tolerance \(\epsilon/2\), gives
both inclusions between the growth lattices of \(H_X\) and
\(k_\infty\). This proves the asserted equality at every real index.
\end{proof}

\subsection{Adapted pluriharmonic metrics}
\label{subsec:adapted-pluriharmonic-metrics}

We now combine the compactness and growth results of the preceding
subsection with the energy decay in
Corollary~\ref{cor:factor-energy} to prove
Theorem~\ref{thm:harmonic-main}.

\begin{proof}[Proof of Theorem~\ref{thm:harmonic-main}]
Write
\[
 (E_*,\theta)=\bigoplus_{\alpha\in A}
 (E_{\alpha,*},\theta_\alpha)
\]
as a finite direct sum of stable summands. Fix one summand and
suppress \(\alpha\) until the last paragraph. Retain its fixed
logarithmic flag model \(\pi:(Y,B)\to(X,D)\) and the stable rational
stages constructed in Section~\ref{sec:fixed-model}. We identify
\(U=Y\setminus B=X\setminus D\).

Let \(h_\nu^{\mathrm{act}}\) be the normalized global
Hermitian--Einstein metrics of
Proposition~\ref{prop:moving-endpoints}. That proposition gives a
subsequence converging to a smooth positive Hermitian metric \(H_X\):
\[
 h_{\nu_j}^{\mathrm{act}}\longrightarrow H_X
 \quad\text{in }C^\infty_{\mathrm{loc}}(U).
\]
The normalization is by spatially constant factors and leaves the
curvature energies unchanged. The two numerical vanishing hypotheses
of Theorem~\ref{thm:harmonic-main} therefore give, by
Corollary~\ref{cor:factor-energy},
\[
 \int_U
 \left(
 |F_{h_\nu^{\mathrm{act}}}
   +[\theta,\theta^{\dagger h_\nu^{\mathrm{act}}}]|^2
 +2|\partial_{h_\nu^{\mathrm{act}}}\theta|^2
 \right)dV_{\omega_\nu}\longrightarrow0,
\]
where the norms use \(h_\nu^{\mathrm{act}}\) and \(\omega_\nu\).
On \(U\), \(\omega_\nu\to\omega\) smoothly on compact subsets.
Lemma~\ref{lem:energy-closure} consequently gives
\[
 F_{H_X}+[\theta,\theta^{\dagger H_X}]=0,
 \qquad \partial_{H_X}\theta=0
 \quad\text{on }U.
\]
Thus \(H_X\) is pluriharmonic.

Recall that the model is unchanged outside the analytic set
\(Z=\operatorname{Sing}(E)\cup Z_{\mathrm{flag}}\subset D\) of
codimension at least three from
Theorem~\ref{thm:fixed-terminal-data}.
Corollary~\ref{cor:nonexceptional-growth} gives
\[
 P^{\mathbf c}(H_X)|_{X\setminus Z}
 =E^{\mathbf c}|_{X\setminus Z}
 \quad\text{for every real local multi-index }\mathbf c.
\]
This is precisely the growth hypothesis
\eqref{eq:growth-off-Z} of Proposition~\ref{prop:acceptable-prolongation}.
That proposition establishes acceptability of \((E|_U,H_X)\), and shows
that its growth prolongations are locally free and satisfy
\[
 P^{\mathbf c}(H_X)=E^{\mathbf c}
 \quad\text{on }X
\]
at every real local multi-index. The identifications respect all
filtration inclusions, periodicity maps, and Higgs maps, and the
filtration admits a simultaneous local splitting.

Restoring the summand index, set
\[
 H=\bigoplus_{\alpha\in A}H_{\alpha,X}.
\]
The Chern connection and Higgs field are block diagonal, so the two
flatness equations hold for \(H\). The Hermitian bundle is acceptable
because the Chern curvature of each summand has bounded Poincaré norm.
For a local section \(s=\bigoplus_\alpha s_\alpha\),
\[
 |s|_H^2=\sum_{\alpha\in A}|s_\alpha|_{H_{\alpha,X}}^2.
\]
The defining growth bound for \(s\) therefore holds exactly when it
holds for every \(s_\alpha\). Hence
\[
 P^{\mathbf c}(H)
 =\bigoplus_{\alpha\in A}P^{\mathbf c}(H_{\alpha,X})
 =\bigoplus_{\alpha\in A}E_\alpha^{\mathbf c}
 =E^{\mathbf c}.
\]
Thus \(H\) is adapted to \(E_*\). Taking the direct sum of
the local frames splitting the summand filtrations shows that
\((E_*,\theta)\) is a locally abelian parabolic logarithmic Higgs
bundle, completing the proof.
\end{proof}

\section*{AI Disclosure}
The overall conception, framework, and proof strategy of this paper
were developed by the authors, drawing on the mathematical literature
and their earlier work, rather than generated by AI. In particular,
the use of perturbations to graded-semisimple parabolic Higgs objects
follows Mochizuki's treatment of the algebraic setting
\cite[Section~3.3]{Mochizuki2006}. The decision to use the
correspondence between orbifold sheaves and parabolic sheaves as the
basis of the resolution strategy came from the authors. The
compactness and boundary estimates for Hermitian--Einstein metrics in
Section~\ref{sec:compactness}, which constitute the main analytic
contribution of this paper, develop and refine ideas from the authors'
previous work and their work with collaborators, in particular Li's
joint work with Chuanjing Zhang and Xi Zhang \cite{Li-Zh-Zh} and
the work of Jiang and Li \cite{JiangLi2026}.

The reasoning agent Danus, whose underlying reasoning engine is
ChatGPT 5.6, assisted with specific parts of the argument within this
framework. While the authors were using Danus to review the logical
correctness of the resolution procedure, it identified a gap and
pointed out the need for destackification, a step that the authors
had not recognized in their initial argument. Danus proposed a
strategy for repairing this gap, leading to the addition of
Subsection~\ref{sec:destackification}. It then made corresponding
revisions to Subsection~\ref{subsec:resolving-logarithmic-flags} to
incorporate this additional step into the resolution of the
logarithmic flags.

For the estimates in Section~\ref{sec:compactness}, the authors
proposed the use of a scalar barrier and determined its role in
controlling the Hermitian--Einstein metrics and their boundary
growth. They regard this as a key idea in the argument. Danus,
using ChatGPT 5.6, helped find a suitable barrier function for this
purpose, as developed in Subsection~\ref{subsec:curvature-barriers}.

The authors have manually checked all AI-assisted arguments,
constructions, and calculations, including the destackification
repair, the resulting changes to the resolution procedure, and the
barrier construction. They take full responsibility for the
mathematical content and correctness of the paper.

\bibliographystyle{amsalpha}
\bibliography{bibliography}

\end{document}